\documentclass[12pt]{article}

\usepackage[margin=0.9in]{geometry} 
\usepackage{amsmath,amsthm,amssymb,mathrsfs,bbm,titling,mathtools,hyperref,nameref,framed,bm,enumitem,dsfont,esint}
\usepackage{color}
\usepackage[style=numeric,backend=bibtex]{biblatex}
 \let\oldref\ref
\renewcommand{\ref}[1]{(\oldref{#1})}

\newcommand{\N}{\mathbb{N}}
\newcommand{\Z}{\mathbb{Z}}
\newcommand{\C}{\mathbb{C}}
\newcommand{\R}{\mathbb{R}}
\newcommand{\Q}{\mathbb{Q}}

\newcommand{\F}{\mathbb{F}}

\newcommand{\eps}{\varepsilon}
\newcommand{\supp}{\mathrm{supp}\,}

\theoremstyle{definition}

\newtheorem{theorem}{Theorem}[section]
\newtheorem{lemma}[theorem]{Lemma}
\newtheorem{corollary}[theorem]{Corollary}

\newtheorem{definition}[theorem]{Definition}

\newtheorem{proposition}[theorem]{Proposition}

\theoremstyle{remark}
\newtheorem{remark}[theorem]{Remark}

\numberwithin{equation}{section}

\begin{document}
 

\title{Restricted projections and Fourier decoupling in $\Q_p^n$}
\author{Ben Johnsrude, Zuo Lin}
 \date{}
\maketitle
\vspace{-0.9in}

\begin{abstract}
    We prove a restricted projection theorem for Borel subsets of $\Q_p^n$ in the regime $p>n$. This generalizes results of Gan-Guo-Wang in the real setting.
\end{abstract}

\vspace{1em}

\section{Introduction}\hfill

Let $1\leq m<n$, and $V=(v_1,\ldots,v_k)$ be a tuple of vectors in $\Q_p^n$. Write $P_V:\Q_p^n\to\Q_p^m$ for the function
\begin{equation*}
    P_V(x)=\begin{bmatrix}
        \,\,\,\rule[1mm]{0.6cm}{0.01mm} & v_1^\perp & \rule[1mm]{0.6cm}{0.01mm}\,\,\,\\
        &\vdots&\\
        \,\,\,\rule[1mm]{0.6cm}{0.01mm} & v_m^\perp & \rule[1mm]{0.6cm}{0.01mm}\,\,\,
    \end{bmatrix}\begin{bmatrix}
        x_1\\\vdots\\x_n
    \end{bmatrix}
\end{equation*}
We will be interested in the problem of determining the relation between the sizes of a Borel set $A\subseteq\Q_p^n$ and its projection $P_V[A]$, for various choices of $V$ and $A$. In real Euclidean space $\R^n$, much work has been done: Marstrand's projection theorem \cite{Mar54} states that
\begin{equation*}
    \dim(P_V[A])=\min(\dim(A),m)\quad\text{for a.e. $V$ such that $|v_1\wedge\cdots\wedge v_k|\sim 1$.}
\end{equation*}
Recent developments in Fourier analysis have permitted analogous results to be proved when the tuple of vectors $V$ is set to range over a much more sparse set, e.g. a curve. Again in the real case, \cite{gan2022restricted} demonstrated that, for $\gamma$ any smooth nondegenerate curve in $\R^n$ and $A\subseteq\R^n$ a Borel set of dimension $\dim(A)$, it holds that for almost every $t$ and each $1\leq m<n$, the orthogonal projection of $A$ onto the span of $\gamma^{(1)}(t),\ldots,\gamma^{(m)}(t)$ has dimension $\min(\dim(A),m)$. Theorems of this form are termed \textit{restricted projection theorems}.

We now state our main result.
\begin{theorem}\label{mainthm}
    Let $A$ be a Borel subset of $\Q_p^n$. For each $t\in\Z_p$, let $V=(\gamma^{(1)}(t),\ldots,\gamma^{(m)}(t))$, where $\gamma(t)=(\frac{t}{1!},\ldots,\frac{t^n}{n!})$ is the moment curve. Then, for almost every $t\in\Z_p$, it holds that
    \begin{equation*}
        \mathrm{dim}_H(P_V[A])=\min\left(m,\mathrm{dim}_H(A)\right).
    \end{equation*}
    Here and throughout $\mathrm{dim}_H$ denotes the Hausdorff dimension of a metric space.
\end{theorem}

The restricted projection theorem has applications in homogeneous dynamics, see \cite{Lindenstrauss2021PolynomialED}, \cite{lindenstrauss2022effective} and \cite{lindenstrauss2023effective}. Using the $(m, n) = (1, 3)$ case, Lindenstrauss--Mohammadi and Lindenstrauss--Mohammadi--Wang proved effective density and equidistribution for certain 1-parameterized unipotent flow in quotient of $\sf{SL}_2(\mathbb{C})$ and $\sf{SL}_2(\mathbb{R}) \times \sf{SL}_2(\mathbb{R})$ with finite volume. Using the $(m, n) = (2, 5)$ case, Lindenstrauss--Mohammadi--Wang--Yang proved effective equidistribution for certain unipotent flow in $\sf{SL}_3(\mathbb{R})/\sf{SL}_3(\mathbb{Z})$. 

The equidistribution results on certain unipotent flow in compact quotient of $\sf{SL}_2(\mathbb{Q}_p) \times \sf{SL}_2(\mathbb{Q}_p)$ has some important applications in number theory. It plays a crucial role in the proofs of uniform distribution of Heegner points by Vatsal,
and Mazur conjecture on Heegner points by C. Cornut; and their generalizations in
their joint work on CM-points and quaternion algebras\cite{Vat02, Cor02, CV05}. Motivated by these applications, we seek to prove an effective density and equidistribution result on certain unipotent flow in compact quotient of $\sf{SL}_2(\mathbb{Q}_p) \times \sf{SL}_2(\mathbb{Q}_p)$, which lead us to prove a restricted projection in the $p$-adic setting. 

The purpose of this paper is to generalize the results of \cite{gan2022restricted} to the $p$-adic setting. One of the motivations is an application to homogeneous dynamics; see Theorem \ref{homog}.

We set out the following notation and convention:
\begin{itemize}
    \item $\gamma(t) = (t, \frac{t^2}{2!}, ..., \frac{t^n}{n!})$ is a curve $\Z_p\to\Q_p^n$;
    \item for each $t\in\Z_p$ and $1\leq m      <n$, we set $\Pi_t^{(m)}$ be the following projection from $\Q_p^n$ to $\Q_p^m$:
    \begin{equation*}
        \Pi_t^{(m)}(x_1, ..., x_n) = (x_1 + tx_2 + ... + \frac{t^{n-1}}{(n - 1)!}x_n, ..., x_m + tx_{m + 1} + ... +\frac{t^{n-m}}{(n - m)!}x_n),
    \end{equation*}
    i.e.
    \begin{equation*}
        \Pi_t^{(m)}(x_1,\ldots,x_n)=\begin{bmatrix}
            \, & \gamma^{(1)}(t)^\top & \,\\
            \, & \vdots & \,\\
            \, & \gamma^{(m)}(t)^\top & \,
        \end{bmatrix}.\begin{bmatrix}
            x_1\\\vdots\\x_n
        \end{bmatrix};
    \end{equation*}
    \item $\mu$ is the Haar measure on $\Q_p^n$ with normalized measure $\mu(\Z_p^n)=1$;
    \item $|\cdot|_b$ covering number or packing number (which agree in $\Q_p$);
    \item $\#$ cardinality of a finite set;
    \item $\nu$ will always be the uniform probability measure on a finite set $F$, unless otherwise specified;
    \item $\mathds{1}_A$ will always be the characteristic function on the set $A$. 
\end{itemize}

The following projection theorem is the one needed in the proof of effective equidistribution of unipotent flow on quotient of $\sf{SL}_2(\mathbb{Q}_p) \times \sf{SL}_2(\mathbb{Q}_p)$. The statement of the following theorem is a generalized verision of the same as Theorem 5.1 in \cite{Lindenstrauss2021PolynomialED} in $p$-adic setting. By an application of Frostman's lemma, it implies Theorem \ref{mainthm}.

\begin{theorem}\label{proj}
    For $t\in\Z_p$, Let $\alpha \in (0,m)$, $b_0 < b_1 \in (0, 1)$ be three parameters. Suppose $F\subseteq\Z_p^n$ is a finite subset satisfying the following $\alpha$-dimensional condition at scales $\geq b_0$:
    \begin{equation}\label{dimcond}
        \frac{\#(F\cap B(x,b))}{\#F}\leq C (b/b_1)^\alpha\quad\forall x\in\Z_p^n\quad\forall b \geq b_0.
    \end{equation}
    Let $\nu$ be the uniform probability measure on $F$ and $\nu_t = \bigl(\Pi_t^{(m)}\bigr)_{*} \nu$ be the pushforward measure.
    
    Then, for all $\varepsilon \in (0, \frac{\alpha}{100})$, there exists $C_\epsilon > 0$ such that $\forall b \geq b_0$, there exists $J_b$ s.t. $\mu(\Z_p \backslash J_b) \leq C_{\varepsilon}b^\varepsilon$ s.t. $\forall t \in J_{b}$, there exists $F_{b,t} \subseteq F$ with $\nu(F \backslash F_{b,t}) \leq C_{\varepsilon}b^{\varepsilon}$ s.t., $\forall w \in F_{b,t}$,
    \begin{equation*}
        \nu_{t}\bigl(B(\Pi_t^{(m)}(w),b)\bigr)\leq C_\varepsilon \Bigl(\frac{b}{b_1}\Bigr)^{\alpha - O(\sqrt{\varepsilon})}.
    \end{equation*} 
    The term $O(\sqrt{\eps})$ can be taken to be $4 \cdot 10^{10n} \sqrt{\eps}$. The constant $C_{n,p,\eps}(c)$ can be chosen as 
    \begin{equation*}
        C_{\eps} = 4\max(1,C)\exp\left(10^4(\log p)\eps^{-5n\log n}n^{20 n^2}\right).
    \end{equation*}
\end{theorem}

We now present the following application of Theorem \ref{proj} to the setting of homogeneous dynamics on quotient of $\sf{SL}_2(\mathbb{Q}_p) \times \sf{SL}_2(\mathbb{Q}_p)$. We first set out some notation. Let $\mathfrak{r}=\mathfrak{sl}_2(\Q_p)$ be the trace-zero $2\times 2$ matrices over $\Q_p$, and equip $\mathfrak{r}$ with the maximum-entry norm, with respect to $|\,\cdot\,|_p$. For each $r\in\Z_p$, we write $\xi_r:\mathfrak{r}\to\Q_p$ for the map
\begin{equation*}
    \xi_r(w)=w_{12}-2rw_{11}-w_{21}r^2,
\end{equation*}
where $w_{ij}$ denotes the corresponding matrix entry of $w$.

\begin{theorem}\label{homog} Let $0<\alpha<1,0<b_0=p^{-l_0}<b_1=p^{-l_1}<1$ be three parameters. Let $F\subseteq B_{\mathfrak{r}}(0,b_1)$ (the closed ball in $\mathfrak{r}$ centered at $0$ of radius $b_1$) be such that 
\begin{equation*}
  \frac{\#(F\cap B_{\mathfrak{r}}(w,b))}{\#F}\leq D'(b/b_1)^\alpha,
\end{equation*}
for all $w\in\mathfrak{r}$ and all $b\geq b_0$, and some $D'\geq 1$. Let $0<\eps<0.01$ and let $J$ be a metric ball in $\Z_p$.

Then there exists $J'\subseteq J$ such that $\mu(J') \geq(1-\frac{1}{p})\mu(J)$ satisfying the following. For each $r\in J'$, there exists a subset $F_r\subseteq F$ with
\begin{equation*}
    \#F_r\geq\left(1-\frac{1}{p}\right)\#F
\end{equation*}
such that for all $w\in F_r$ and $b\geq b_0$ we have
\begin{equation}\label{dimproj}
    \frac{\#\{w'\in F:|\xi_r(w')-\xi_r(w)|_p\leq b\}}{\#F}\leq C_\eps(b/b_1)^{\alpha-\eps},
\end{equation}
where $C_\eps$ depends on $\eps,\#J$, and $D'$.
\end{theorem}

\begin{remark}
    The maps $\xi_r$ may alternately be written as $\xi_r(w)=\bigl(\mathrm{Ad}_{u_r}(w)\bigr)_{12}$, where $u_r=\begin{pmatrix}1 & r \\ 0 & 1\end{pmatrix}$ and $\mathrm{Ad}$ is the adjoint action of $\sf{SL}_2(\Q_p)$ on its Lie algebra $\mathfrak{sl}_2(\Q_p)$.
\end{remark}

\begin{proof}[Proof of Theorem \ref{homog} from Theorem \ref{proj}]

    Let $\eps_1 = (10^{-30}\eps)^2$. Identifying $\mathfrak{r}=\mathfrak{sl}_2(\Q_p)$ with $\Q_p^3$ (with the latter equipped with the usual $\ell^\infty$ norm), we may appeal to the $(m,n)=(1,3)$ case of Theorem \ref{proj} with this $\eps_1$. Choose $l_2$ such that
    \begin{equation*}
        \sum_{l=l_2}^\infty C_{\eps_1} p^{-\eps_1 l}<\frac{1}{p}\mu(J).
    \end{equation*}
    If $l_2>l_0$, let
    \begin{equation*}
        C_{\eps_1}'=(p^{l_2})^{\alpha-\eps},
    \end{equation*}
    which depends only on $\eps$ and $\mu(J)$. We then have
    \begin{equation*}
        \begin{split}
            1&\leq C_{\eps_1}'(b_0)(p^{-l_2})^{\alpha-O(\sqrt{\eps})}\\
            &\leq C_{\eps_1}'(b_0)^{\alpha-\eps}.
        \end{split}
    \end{equation*}
    Thus we may take $J'=J$ and $F_r=F$, and \ref{dimproj} holds trivially.
    
    Now we assume $l_2\leq l_0$. Let $J'=\bigcap_{l=l_2}^{l_0}J_{p^{-l}}\cap J$, where $J_{p^{-l}}$ is the set obtained from Theorem \ref{proj}. We compute:
    \begin{equation*}
        \begin{split}
            \mu(J')&\geq \mu(J)-\sum_{l=l_2}^\infty\mu(\Z_p\setminus J_{p^{-l}})\\
            &\geq \mu(J)-\sum_{l=l_2}^\infty C_{\eps_1} p^{-\eps_1 l}\\
            &\geq\left(1-\frac{1}{p}\right)\mu(J).
        \end{split}
    \end{equation*}
    For all $r\in J'$, let $F_r=\bigcap_{l=l_2}^{l_0}F_{p^{-l},r}$, where the sets $F_{p^{-l},r}$ are as obtained from Theorem \ref{proj}. From the choice of $l_2$ and the union bound, we conclude that $\#(F\setminus F_r)\leq\frac{1}{p}\#F$, so $\#F_r\geq(1-\frac{1}{p})\#F$.

    Now, for all $w\in F_r$ and $l_2\leq l\leq l_0$, by Theorem \ref{proj}, we have
    \begin{equation*}
        \#\{w'\in F:|\xi_r(w')-\xi_r(w)|_p\leq p^{-l}\}\leq C_{\eps_1} (b/b_1)^{(\alpha-\eps)(l_1-l)}\#F.
    \end{equation*}
    so that \ref{dimproj} holds.

    Finally, we consider scales $p^{-l}>p^{-l_2}$, i.e. $l<l_2$. In this scale, we let $C_\eps'\geq p^{l_2(\alpha-O(\sqrt{\eps}))}$. We have
    \begin{equation*}
        1\leq C_{\eps_1}'p^{-l_2(\alpha-\eps)}\leq C_{\eps_1}' (b/b_1)^{(\alpha-\eps)(l_1-l)},
    \end{equation*}
    so that \ref{dimproj} holds trivially.

\end{proof}

\begin{proof}[Proof of Theorem \ref{mainthm} from Theorem \ref{proj}]

Identical to the ``Proof of Theorem 1.2 assuming Theorem 2.1,'' from \cite{gan2022restricted}. Note that the Frostman lemma holds for Borel sets in compact metric spaces, and that $\mathrm{dim}_H(\Z_p^m)=m$. Note also that the relevant covering lemma is valid in separable metric spaces, and that $\mu$ is doubling.
    
\end{proof}

We mention one final result of this paper. In the interest of obtaining explicit bounds for the projection theorems, motivated by the problem of producing effective estimates in the homogeneous dynamics application, we have in particular needed a fully explicit bound on $p$-adic decoupling for the moment curve; this is proved in Theorem \ref{momentcurve} below. To our knowledge, this gives the first fully explicit bound for the main conjecture of Vinogoradov's mean value theorem in the range $n\geq 3$, which we state here.
\begin{theorem}[Explicit Vinogradov bound]\label{vinogradov}
    For $n\geq 2,s\geq 2,$ and $N\geq 2$, we write
    \begin{equation*}
        J_{s,n}(N)=\#\left\{\mathbf{a},\mathbf{b}\in[N]^n:\sum_{j=1}^s(a_j^d-b_j^d)=0\,\,\forall 1\leq d\leq n\right\},
    \end{equation*}
    which is the number of solutions to the Vinogradov system of Diophantine equations. For each such $s,n$, and each $N\geq \exp(\exp(3n(4n\log n+1)))$, we have
    \begin{equation*}
        J_{s,n}(N)\leq\exp\left(10^5 s e^{3n}(\log N)^{1-\frac{1}{4n\log n+1}}\right)(N^s+N^{2s-\frac{n(n+1)}{2}}).
    \end{equation*}
\end{theorem}
\begin{proof}[Proof of Theorem \ref{vinogradov}, assuming Theorem \ref{momentcurve}]
    We will first show the inequality
    \begin{equation}\label{partialvinoest}
        J_{s,n}(N)\leq\exp\left(6\cdot 10^4 s\eps^{-4n\log n}n^{12n^2}\right)N^{2s\eps}(N^s+N^{2s-\frac{n(n+1)}{2}})
    \end{equation}
    for each $\eps\in(0,1)$. Subsequently, we will optimize this estimate over $\eps$.

    Let $p\in[n,2n]$ be a prime. Assume temporarily that $N=p^\ell$ for some $\ell\in\N$. For each $1\leq a\leq N$ integral, we write $I_a=a+p^\ell\Z_p$; these form a partition of $\Z_p$. Let $\{\mathcal{U}_{I_a,a}\}_{a\in[N]}$ be the associated family of anisotropic boxes adapted to the $n$-dimensional moment curve, as defined in Section 6.1 below. By Theorem \ref{momentcurve},
    \begin{equation*}
        \mathrm{Dec}_{\ell^2L^{n(n+1)}}(\{\mathcal{U}_{I_a,a}\}_{a\in[N]})\leq\exp\left(10^4(\log p)\eps^{-4n\log n}n^{10n^2}\right)p^{\eps\ell},
    \end{equation*}
    for each $\eps\in\frac{1}{\N}$. By Lemma \ref{interpolate}, we have
    \begin{equation*}
        \mathrm{Dec}_{\ell^2L^{2s}}(\{\mathcal{U}_{I_a,a}\}_{a\in[N]})\leq\exp\left(10^4(\log p)\eps^{-4n\log n}n^{10n^2}\right)p^{\eps\ell}(1+p^{\frac{\ell}{2}(1-\frac{n(n+1)}{2s})})
    \end{equation*}

    For each $1\leq a\leq N$ integral, write $g_a:\Q_p^k\to\C$ for the function
    \begin{equation*}
        g_a(x)=\chi(x\cdot\gamma(a)) 1_{p^{-\ell n}\Z_p^n}(x).
    \end{equation*}
    Then the Fourier support of $g_a$ is $\gamma(a)+p^{\ell n}\Z_p^n\subseteq\mathcal{U}_{I_a,a}$. Thus, by decoupling,
    \begin{equation*}
        \left\|\sum_{a=1}^Ng_a\right\|_{L^{2s}(\Q_p^n)}^{2s}\leq 2^{2s}\exp\left(2\cdot 10^4 s(\log p)\eps^{-4n\log n}n^{10n^2}\right)p^{2s\ell\eps}(1+p^{\ell(s-\frac{n(n+1)}{2})})p^{(n+s)\ell}.
    \end{equation*}
    By a standard manipulation, the left-hand side is $p^{\ell n}J_{s,n}(N)$. Thus, in this case, we obtain
    \begin{equation*}
        J_{s,n}(N)\leq 2^{2s}\exp\left(2\cdot 10^4 s(\log p)\eps^{-4n\log n}n^{10n^2}\right)N^{2s\eps}(N^s+N^{2s-\frac{n(n+1)}{2}})
    \end{equation*}
    If instead $p^\ell<N<p^{\ell'}$, then the preceding implies
    \begin{equation*}
        J_{s,n}(N)\leq 2^{2s}p^{2s(1+\eps)}\exp\left(2\cdot 10^4 s(\log p)\eps^{-4n\log n}n^{10n^2}\right)N^{2s\eps}(N^s+N^{2s-\frac{n(n+1)}{2}}).
    \end{equation*}
    Finally, appealing to $n\leq p\leq 2n$, and various elementary estimates, we conclude that
    \begin{equation*}
        J_{s,n}(N)\leq \exp\left(6\cdot 10^4 s\eps^{-4n\log n}n^{11n^2}\right)N^{2s\eps}(N^s+N^{2s-\frac{n(n+1)}{2}}).
    \end{equation*}
    Finally, interpolating between the cases $\frac{1}{\ell+1}<\eps<\frac{1}{\ell}$, we obtain \ref{partialvinoest}.

    Finally, we select $\eps=e^{3n}(\log N)^{-\frac{1}{4n\log n+1}}$ in \ref{partialvinoest}, using the lower bound on $N$. It transpires that
    \begin{equation*}
        \log\left(\frac{J_{s,n}(N)}{N^s+N^{2s-\frac{n(n+1)}{2}}}\right)\leq 6\cdot 10^4s(\log N)^{\frac{4n\log n}{4n\log n+1}}+2e^{3n}s(\log N)^{\frac{4n\log n}{4n\log n+1}}.
    \end{equation*}
    By trivial estimates, we conclude.
\end{proof}

Finally, we outline the remaining sections. In Section \ref{discr}, we reduce the proof of Theorem \ref{proj} to a problem of covering sets with tubes, which we refer to as a Kakeya estimate. In Section \ref{fourier}, we demonstrate that the Kakeya estimate may be proved with a suitable decoupling theorem. In Section \ref{decsection}, we prove the decoupling theorem, assuming that the usual Bourgain-Demeter-Guth decoupling theorem for the moment curve may be extended to the $p$-adic setting. Finally, in the appendices, we discuss the proof of moment curve decoupling in the $p$-adic setting, by modifying an argument of \cite{guo2021short}.

\subsection{Acknowledgements}

We would like to thank Amir Mohammadi and Hong Wang for suggesting this problem. We would also like to thank Terence Tao and Zane Kun Li for helpful suggestions and comments. 

\section{Discretization}\label{discr}\hfill

In this section, we reduce the projection theorem \ref{proj} to a Kakeya estimate, whose proof will be established by Fourier analysis in following sections. 

Let $\delta = p^{-l}$ and let $\mathbb{D}^{(m)} = \{x + p^{l}\Z_p^m: x\in\{0,\ldots,p^l-1\}^m\}$ be the set of $\delta$-balls in $\Z_p^m$. Let $\mathbb{T}_t^{(m)} = \Z_p^n\cap\{(\Pi_t^{(m)})^{-1}(D): D \in \mathbb{D}^{(m)}\}$. Elements in $\mathbb{T}_t$ are tilted $\delta^m \times 1^{n - m}$ boxes. We will use $T_t^{(m)}$ to denote elements in $\mathbb{T}_t$. We will drop the superscript if it is clear that we are dealing with the $(m, n)$ case. 

\begin{theorem}[Kakeya estimate]\label{kakeya}
Let $\delta, \delta_0 \in p^{-\mathbb{N}}$ with $\delta > \delta_0$. Let $\Lambda_\delta$ be a maximal $\delta$-separted set of $\mathbb{Z}_p$. Given $\varepsilon>0$ and $\alpha \in(0, m)$, let $\nu$ be a finite non-zero Borel measure supported in $\mathbb{Z}_p^n$ with $c_\alpha^{\delta_0}(\nu) = \sup_{\mathbf{x} \in \mathbb{Q}_p^n, r > \delta_0} \frac{\nu(B(\mathbf{x}, r))}{r^{\alpha}}<\infty$. Take $\mathbb{W}_\theta \subset \mathbb{T}_\theta$ arbitrary and denote $\mathbb{W}:=\cup_{\theta \in \Lambda_{\delta}} \mathbb{W}_\theta$. Suppose that
\[
\sum_{T \in \mathbb{W}} \mathds{1}_T(\mathbf{x}) \geq c\delta^{\varepsilon-1}, \quad \forall \mathbf{x} \in \operatorname{supp}(\nu).
\]
Then
\[
\#\mathbb{W} \geqslant C_{n,p,\eps}(c)\cdot\nu(\mathbb{Q}_p^n)c_\alpha^{\delta_0}(\nu)^{-1} \delta^{-1-\alpha} \delta^{O(\sqrt{\epsilon})}
\]
Here it is important that the constant $C_{n,p,\eps}(c)$ does not depend on $\delta$. The term $O(\sqrt{\eps})$ can be taken to be $10^{10n} \sqrt{\eps}$. The constant $C_{n,p,\eps}(c)$ can be chosen as 
\begin{equation*}
    C_{n,p,\eps}(c)=\min(1,c^{\eps^{-1}})\exp\left(-10^4(\log p)\eps^{-5n\log n}n^{20 n^2}\right).
\end{equation*}
\end{theorem}

\begin{proof}[Proof of Theorem \ref{proj} assuming Theorem \ref{kakeya}]\hfill

The proof is a finitary version of the one in section 2 of \cite{gan2022restricted}. 
Let $\varepsilon_0 = 10^{10n} \sqrt{2\varepsilon}$. Fix $s = \alpha - 2\varepsilon_0 < \alpha$. Note that $s < \alpha - 2\varepsilon - \varepsilon_0$. For each $b\geq b_0$ and each $b$-separated set $\Lambda_b\subseteq\Z_p$, we define the set
\begin{equation*}
    F_{b,t}^{\mathrm{bad}}=\Big\{w\in F:\nu(\{w'\in F:|\Pi_t^{(m)}(w)-\Pi_t^{(m)}(w')|\leq b\}) >  c_{\alpha}^{b_0}(\nu) b^{s}\Big\}
\end{equation*}
for all $t \in \Lambda_b$. 

We will first demonstrate that there exists $C(\alpha, s)$ such that 
\[
\sum_{t \in \Lambda_b} \nu(F_{b, t}^{\mathrm{bad}}) \leq C(\alpha, s) b^{2\varepsilon-1}.
\]

Suppose not, we have that 
\begin{align*}
    \sum_{t\in\Lambda_b}\nu(F_{b,t}^{\mathrm{bad}})>Cb^{2\varepsilon-1}.
\end{align*}

Note that for all $t$, $\bigl|\Pi_{t}^{(m)}(F_{b, t}^{\mathrm{bad}})\bigr|_{b} \leq \frac{1}{c_\alpha^{b_0}(\nu)} b^{-s}$. Hence we could cover it by a collection $\mathbb{D}_t$ of balls $D$ where $\#\mathbb{D}_t\leq \frac{1}{c_\alpha^{b_0}(\nu)} b^{-s}$. Let $\mathbb{W}_t = \{p_t^{-1}(D) \bigcap \mathbb{Z}_p^n: D \in \mathbb{D}_t\}$, $\mathbb{W} = \bigcup_{t} \mathbb{W}_t$. Consider the following set 
\[A = \{(t, w) \in \Lambda_b \times F: w \in F_{b, t}^{\mathrm{bad}}\}. \]
Let $\lambda$ denote the counting measure on $\Lambda_b$. We have 
\[(\lambda \otimes \nu)(A) = \sum_{t \in \Lambda_b} \nu(F_{b, t}^{\mathrm{bad}}) > C b^{2\varepsilon-1}. \]
Therefore
\[\int \#\{t \in \Lambda_b: w \in F_{b,t}^{\mathrm{bad}}\} d\nu(w) > C b^{2\varepsilon-1},\]
so that, dividing the integral into the domains where the integrand is larger/smaller than $\frac{C}{2}b^{2\eps-1}$,
\[b^{-1}\nu\Big(\big\{w \in F: \sum_{T \in \mathbb{W}} \mathds{1}_{T}(x) > \frac{C}{2} b^{2\varepsilon} b^{-1}\big\}\Big) + \frac{C}{2} b^{2\varepsilon-1}  > C b^{2\varepsilon-1},\]
i.e.
\[
\nu\Big(\big\{w \in F: \sum_{T \in \mathbb{W}} \mathds{1}_{T}(x) > \frac{C}{2} b^{2\varepsilon-1}\big\}\Big) > \frac{C}{2}  b^{2\varepsilon}. 
\]

Let $F_{b}^{\mathrm{bad}} = \big\{w \in F: \sum_{T \in \mathbb{W}} \mathds{1}_{T}(x) > \frac{C}{2} b^{2\varepsilon-1}\big\}$, so that $\nu(F_{b}^{\mathrm{bad}}) > \frac{C}{2}  b^{2\varepsilon}$. Note that for all $x \in F_{b}^{\mathrm{bad}}$, $\sum_{T \in \mathbb{W}} \mathds{1}_{T}(x) > \frac{C}{2} b^{2\varepsilon-1}$. 

We apply Theorem~\ref{kakeya} to $\nu|_{F_{b}^{\mathrm{bad}}}$, scale $b$ and $2\varepsilon$. There exists $C_{2\varepsilon, \alpha}$ 

\[
\#\mathbb{W} \geq C_{2\varepsilon, \alpha} \cdot \frac{C}{2}  b^{2\varepsilon} c_\alpha^{b_0}(\nu)^{-1} b^{-1-\alpha} b^{\varepsilon_0} . 
\]
By pigeonholing, this implies that there exists $t \in \Lambda_b$ such that
\[
\#\mathbb{W}_t \geq C_{2\varepsilon, \alpha} \cdot \frac{C}{2}  b^{2\varepsilon} c_\alpha^{b_0}(\nu)^{-1} b^{-\alpha} b^{\epsilon_0} . 
\]
This is a contradiction to the assumption that $\#\mathbb{W}_t < \frac{1}{c_\alpha^{b_0}(\nu)} b^{-s}$ if $C > \frac{2}{C_{2\varepsilon, \alpha}}$. Therefore, 
\[
\sum_{t \in \Lambda_b} \nu(F_{b, t}^{bad}) \leq \frac{4}{C_{2\varepsilon, \alpha}} b^{2\varepsilon} \cdot b^{-1}. 
\]

Now let $E_b$ be the `exceptional' set of parameters $t \in \mathbb{Z}_p$ where $F_{b,t}^{\mathrm{bad}}$ is large, namely, 
\[E_b = \{t \in \mathbb{Z}_p:\nu(F_{b,t}^{bad}) > \frac{4}{C_{2\varepsilon, \alpha}} b^{\varepsilon}\}. \]
Pick a maximal $b$-separated set of $E_b$ and extend it to be a maximal $b$-separated set $\Lambda_b$ in $\Z_p$, we have
\begin{align*}
    b^{-1} \cdot \mu(E_b) \cdot \frac{4}{C_{2\varepsilon, \alpha}} b^{\frac{\varepsilon}{2}} \leq{}& \#(\Lambda_b \cap E_b)\cdot \frac{4}{C_{2\varepsilon, \alpha}} b^{\frac{\varepsilon}{2}}\\
   \leq{}& \sum_{t \in \Lambda_b \bigcap E_b} \nu(F_{b, t}^{bad})\\
   \leq{}& \frac{4}{C_{2\varepsilon, \alpha}} b^{\varepsilon} \cdot b^{-1}. 
\end{align*}

Therefore, $\mu(E_b) < b^{\varepsilon}$. Let $C_{\varepsilon} = \max \{c_{\alpha}^{b_0}(\nu), \frac{4}{C_{2\varepsilon, \alpha}}\}$ and $J_b = \Z_p \backslash E_b$, we complete the proof. 
\end{proof}

\section{Kakeya estimate via decoupling cones over moment curves}\label{fourier}

In this section, we formulate the decoupling estimate Proposition \ref{decprop}, and indicate how it may be used to prove Theorem \ref{kakeya}. We begin by setting out some notation that will be helpful in studying the wave packet expansions of functions with restricted Fourier support.

For each $\theta\in\Lambda_\delta$ and $\alpha,\beta\in p^{-\Z}$, write
\begin{equation}\label{rescale}
    A_{\theta,\alpha,\beta}=[\alpha^{-1}\gamma^{(1)}(\theta),\ldots,\alpha^{-1}\gamma^{(m)}(\theta),\beta^{-1}\gamma^{(m+1)}(\theta),\ldots,\beta^{-1}\gamma^{(n)}(\theta)].
\end{equation}
When the third subscript is supressed, we will understand it to be $1$. Write also
\begin{equation*}
    \tau_\theta=A_{\theta,\delta^{-1}}[\Z_p^n].
\end{equation*}
Notice in particular that $\tau_\theta$ has dimensions $\delta^{-1}\times\cdots\times \delta^{-1}\times 1\times\cdots\times 1$, with $m$ copies of $\delta^{-1}$ and $(n-m)$ copies of $1$. If $f_\theta$ has Fourier support within $\tau_\theta$, then $f_\theta$ may be expanded into wave packets of the form $a_T\chi(x\cdot\gamma(\theta))\mathds{1}_T(x)$ for $a_T\in\C$ and $T$ a translate of $A_{\theta,\delta^{-1}}^{-\top}\Z_p^n$.
Note in particular that each $T$ has $p$-adic volume $\delta^{-(n-m)}$. 

It will be convenient to observe that
\begin{equation}
    A_{\theta,1}^{-1}=A_{-\theta,1}
\end{equation}
Indeed, when $i\geq j$,
\begin{equation*}
    (A_{\theta,1}A_{-\theta,1})_{i,j}=\sum_{k=1}^n(A_{\theta,1})_{i,k}(A_{-\theta,1})_{k,j}=\theta^{i-j}\sum_{k=j}^i\frac{(-1)^{k-j}}{(i-k)!(k-j)!}
\end{equation*}
The sum may be rewritten as
\begin{equation*}
    \sum_{k=j}^k\frac{(-1)^{k-j}}{(i-k)!(k-j)!}=\frac{1}{(2r)!}\sum_{h=0}^{i-j}(-1)^h\binom{i-j}{h}=\frac{1}{(i-j)!}(1-1)^{i-j},
\end{equation*}
and the claim follows.

\begin{proposition}[Decoupling estimate]\label{decprop} For each $\eps\in(0,1)$ and $n\in\N$, we may find $D_{n,p,\eps}\geq 1$ such that the following holds. Suppose $\delta\in p^{-\N}$ and $\Lambda_\delta$ is a $\delta$-separated subset of $\Z_p$. For each $\theta\in\Lambda_\delta$, let $f_\theta$ have Fourier support in the set $\delta^{-1}\tau_\theta\cap(\Z_p^n\setminus p\Z_p^n)$. Write $q_n=n(n+1)$. Then
\begin{equation*}
    \left\|\sum_{\theta\in\Lambda_\delta}f_\theta\right\|_{L^{q_n}(\Q_p^n)}\leq D_{n,p,\eps}\delta^{-1+\frac{n-m+1}{q_n}-\eps}\left(\sum_{\theta\in\Lambda_\delta}\|f_\theta\|_{L^{q_n}(\Q_p^n)}^{q_n}\right)^{1/q_n}.
\end{equation*}
for each $\eps>0$. We may choose $D_{n,p,\eps}$ to be the quantity
\begin{equation*}
    D_{n,p,\eps}=\exp\left(10^4(\log p)\eps^{-5n\log n}n^{10 n^2}\right).
\end{equation*}

\end{proposition}

Before proving Prop. \ref{decprop}, we indicate how it implies Theorem \ref{kakeya}. 

\begin{proof}[Proof of Theorem \ref{kakeya} using Proposition \ref{decprop}]\hfill

We claim the particular inequality
\begin{equation*}
    \#\mathbb{W}\geq \widetilde{C}_{n,p,\eps_k}(c_0)\nu(\Q_p^n)c_\alpha^{\delta_0}(\nu)^{-1}\delta^{-\alpha-1+10^{10n}\sqrt{\eps_k}},
\end{equation*}
for the particular sequence $\{\eps_k\}_{k=1}^\infty$, defined by
\begin{equation*}
    \eps_1=\frac{1}{2},\quad \eps_{k+1}=\frac{1}{4}\left(\sqrt{\eps_k^2+4\eps_k}-\eps_k\right),\quad k\in\N.
\end{equation*}
Then $0<\eps_{k+1}<\eps_k$ for all $k$, and $\widetilde{\eps_{k+1}}=\eps_k$, where
\begin{equation*}
    \tilde{\eps}=\frac{\eps}{1-\sqrt{\eps}}.
\end{equation*}
We have also written $c_0=\min(1,c)$ and
\begin{equation*}
    \widetilde{C}_{n,p,\eps}(c_0)=(10^{-2}\eps)^{q_n}(c_0/4p)^{\eps^{-1}}D_{n,p,\eps}^{-q_n}.
\end{equation*}
From Prop. \ref{decprop}, and observing that $\tilde{\eps}^{-1}+1\leq\eps^{-1}$, the original claim holds for each $\eps_k$. By trivial inequalities, the full result holds.

Following \cite{gan2022restricted}, we proceed by induction on $\eps>0$. By a trivial estimate when $\eps=\eps_1=\frac{1}{2}$, we have the base case. It suffices to show that, if Theorem \ref{kakeya} holds for $\tilde{\eps}=\frac{\eps}{1-\sqrt{\eps}}$, then it holds for $\eps$.

Proceeding to the induction, we assume the result for $\tilde{\eps}$. The tiles in $\mathbb{T}_\theta$ of $\Q_p^n$ have dimensions $\delta\times\cdots\times \delta\times 1\times\cdots\times 1$. We further have, for each $\theta$, a subfamily $\mathbb{W}_\theta\subseteq\mathbb{T}_\theta$ such that $\sum_{T\in\mathbb{W}}\mathds{1}_{T}(x)\geq c_0\delta^{\eps-1}$ for all $x$ in the support of a special measure $\nu$. We wish to demonstrate a suitable lower bound on $\#\mathbb{W}$.

To this end, first observe the calculation
\begin{equation*}
    \Pi_\theta^{(m)}A_{\theta,\delta^{-1}}^{-\top}=\left.
\begin{bmatrix}\,
\smash{
  \underbrace{
    \begin{matrix}
        \delta^{-1} & 0 & \cdots & 0 & 0 & \cdots & 0\\
        0 & \delta^{-1} & \cdots & 0 & 0 & \cdots & 0\\
        \vdots & \vdots & \ddots & \vdots & \vdots & \ddots & \vdots\\
        0 & 0 & \cdots & \delta^{-1} & 0 & \cdots & 0
    \end{matrix}
  }_{n}
}
\vphantom{
  \begin{matrix}
  \smash[b]{\vphantom{\Big|}}
  0\\0\\vdots\\0
  \smash[t]{\vphantom{\Big|}}
  \end{matrix}
}
\,\end{bmatrix}
\,\right\rbrace{\scriptstyle k}
\vphantom{\underbrace{\begin{matrix}0\\0\\vdots\\0\end{matrix}}_{k}};
\end{equation*}
consequently, for each $x\in\Z_p^m$,
\begin{equation*}
    \Z_p^n\cap(\Pi_\theta^{(m)})^{-1}[x+\delta^{-1}\Z_p^m]=A_{\theta,\delta^{-1}}^{-\top}[\delta (x,0)+\Z_p^m\times\Z_p^{n-m}].
\end{equation*}
Thus, we will write members of $\mathbb{T}_\theta$ as translates $c_T+A_{\theta,\delta^{-1}}^{-\top}[\Z_p^n]$ for various choices of $c_T\in\Q_p^n$.

For each $T=c_{T}+A_{\theta,\delta^{-1}}^{-\top}\Z_p^n\in\mathbb{T}_\theta$, consider the function $\mathds{1}_{T}$. If we recall that $\tau_\theta=A_{\theta,\delta^{-1}}[\Z_p^n]$, then we may verify that
\begin{equation*}
    \hat{\mathds{1}}_{T}(\xi)=\chi(-c_{T}\cdot\xi)\delta^{m} \mathds{1}_{\tau_\theta}(\xi).
\end{equation*}
Observe that $\tau_\theta$ has dimensions $\delta^{-1}\times\cdots\times \delta^{-1}\times 1\times\cdots\times 1$, with long sides parallel to $\gamma^{(1)}(\theta),\ldots,\gamma^{(m)}(\theta)$; observe also that $\tau_\theta$ is symmetric about the origin.

Let $\kappa$ be a positive integer such that $\delta^{-\sqrt{\eps}}\leq p^\kappa\leq p\delta^{-\sqrt{\eps}}$, and write $\psi_\delta=\mathds{1}_{B(0,\delta^{-1}p^{-\kappa})}$. Then, for some subset $F\subseteq\Q_p^n$ with $\nu(F)\geq\frac{1}{2}\nu(\Q_p^n)$, we either have
\begin{equation}\label{low}
    \frac{c_0}{2}\delta^{-1+\eps}\leq \left|\sum_{T\in\mathbb{W}}\mathds{1}_{T}*\psi_\delta^\vee(x)\right|\quad\forall x\in F,
\end{equation}
or
\begin{equation}\label{high}
    \frac{c_0}{2}\delta^{-1+\eps}\leq \left|\sum_{T\in\mathbb{W}}\mathds{1}_{T}*(\mathds{1}_{\delta\Z_p^n}-\psi_\delta)^\vee(x)\right|\quad\forall x\in F.
\end{equation}

Observe from the outset that 
\begin{equation*}
    \psi_\delta^\vee=\delta^{-n}p^{-n\kappa}\mathds{1}_{B(0,\delta p^\kappa)}.
\end{equation*}
 We consider case \ref{low} first. For each fixed $T$, we may compute
 \begin{equation*}
     \mathds{1}_{T}*\psi_\delta^\vee=p^{-m\kappa}\mathds{1}_{\widetilde{T}},
 \end{equation*}
where $\widetilde{T}=c_{T}+A_{\theta,p^{-\kappa}\delta^{-1}}^{-\top}[\Z_p^n]$, recalling that $T=c_{T}+A_{\theta,\delta^{-1}}^{-\top}(\Z_p^n)$.
 
 Observe that $\widetilde{T}$ is the $(\delta p^\kappa\times\cdots\times \delta p^\kappa\times 1\times\cdots\times 1)$-plate with the same center as $T$ and the same short directions. As such, for each $\theta\in\Lambda_\delta$ we fix the tiling $\mathcal{T}_\theta$ of $\Q_p^n$ by translates of $A_{\theta,p^{-\kappa}\delta^{-1}}^{-\top}[\Z_p^n]$.

 We investigate the relationship between $\mathbb{T}$ and $\mathcal{T}$. Suppose $T\in\mathbb{T}_{\theta'}$ and $\widetilde{T}\in\mathcal{T}_{\theta}$ are such that $T\subseteq\widetilde{T}$. Let $q\in\Q_p^n$ be the unique element such that $\{q_j\}_p=q_j$ for all $1\leq j\leq n$ and such that $\widetilde{T}=A_{\theta,p^{-\kappa}\delta^{-1}}^{-\top}[q+\Z_p^n]$. Then $-q+A_{\theta,p^{-\kappa}\delta^{-1}}^{\top}[T]=:B$ is a subset of $\Z_p^n$. Moreover, writing $T=A_{\theta',\delta^{-1}}^{-\top}[b+\Z_p^n]$ for the unique $b\in\Q_p^n$ satisfying the preceding equality and $\{b_j\}_p=b_j$ for all $1\leq j\leq n$, we see that
 \begin{equation*}
     B=-q+A_{\theta,p^{-\kappa}\delta^{-1}}^{\top}A_{\theta',\delta^{-1}}^{-\top}b+A_{\theta,p^{-\kappa}\delta^{-1}}^{\top}A_{\theta',\delta^{-1}}^{-\top}[\Z_p^n].
 \end{equation*}

 We will bound, for each $\theta\in\Z_p$, the number of $\theta'\in\Z_p$ such that $A_{\theta,p^{-\kappa}\delta^{-1}}^\top A_{\theta',\delta^{-1}}^{-\top}[\Z_p^n]\subseteq\Z_p^n$. We begin by noticing that
 \begin{equation*}
     A_{\theta,1}^\top A_{\theta',1}^{-\top}=A_{\theta-\theta',1}^\top,
 \end{equation*}
 so that for each $j\leq k$
 \begin{equation*}
     \left(A_{\theta,1}^\top A_{-\theta',1}^{\top}\right)_{jk}=\frac{(\theta-\theta')^{k-j}}{(k-j)!}.
 \end{equation*}
 If $\theta,\theta'$ are such that we have the inequality
 \begin{equation*}
     |\theta-\theta'|_p>p^\kappa\delta,
 \end{equation*}
 then it follows that
 \begin{equation*}
     (A_{\theta,1}^\top A_{-\theta',1}^\top)(e_m)\not\in p^{-\kappa}\delta^{-1}\Z_p^m\times\Z_p^{n-m},
 \end{equation*}
 and hence
 \begin{equation*}
     \operatorname{diag}(p^\kappa\delta,\ldots,p^\kappa\delta,1,\ldots,1)A_{\theta,1}^\top A_{\theta',1}^{-\top}\operatorname{diag}(\delta^{-1},\ldots,\delta^{-1},1,\ldots,1)[\Z_p^n]\not\subseteq\Z_p^n,
 \end{equation*}
 whereas the left-hand side is just $A_{\theta,p^{-\kappa}\delta^{-1}}^\top A_{\theta',\delta^{-1}}^{-\top}[\Z_p^n]$.  It follows that, for each $\widetilde{T}$,
 \begin{equation}\label{trivpacketct}
     \#\big\{T\in\mathbb{W}:T\subseteq\widetilde{T}\big\}\leq p^{\kappa(m+1)}.
 \end{equation}
 On the other hand, if $|\theta-\theta'|_p\leq p^{\kappa}\delta$, we note that
 \begin{equation*}
     A_{\theta,p^{-\kappa}\delta^{-1}}^{\top}A_{\theta',p^{-\kappa}\delta^{-1}}^{-\top}[\Z_p^n]\subseteq\Z_p^n,
 \end{equation*}
 so $A_{\theta,p^{-\kappa}\delta^{-1}}^{-\top}$ and $A_{\theta',p^{-\kappa}\delta^{-1}}^{-\top}$ define the same thick wave packets unless $|\theta-\theta'|_p>p^\kappa\delta$.
 
 
 Now, writing $\mathcal{T}=\bigcup_{\theta\in\Lambda_\delta}\mathcal{T}_\theta$, it holds that
 \begin{equation*}
     \sum_{T\in\mathbb{W}}\mathds{1}_T*\psi_\delta^\vee\leq p^{-m\kappa}\sum_{\widetilde{T}\in\mathcal{T}}(\#\{T\in\mathbb{W}:T\subseteq\widetilde{T}\})\mathds{1}_{\widetilde{T}}.
 \end{equation*}
If we set
\begin{equation*}
    \mathcal{T}_{\theta,\mathrm{light}}=\Big\{\widetilde{T}\in\mathcal{T}_\theta:\#\{T\in\mathbb{W}:T\subseteq\widetilde{T}\}\leq \frac{c_0}{4}\delta^{\eps}p^{(m+1)\kappa}\Big\},
\end{equation*}
and $\mathcal{T}_{\theta,\mathrm{heavy}}=\mathcal{T}_\theta\setminus\mathcal{T}_{\theta,\mathrm{light}}$, then
\begin{equation*}
    p^{-m\kappa}\left\|\sum_{\theta\in\Lambda_\delta}\sum_{\widetilde{T}\in\mathcal{T}_{\theta,\mathrm{light}}}\big(\#\{T\in\mathbb{W}:T\subseteq\widetilde{T}\}\big)\mathds{1}_{\widetilde{T}}\right\|_{L^\infty(\Q_p^n)}\leq \frac{c_0}{4}\delta^{-1+\eps},
\end{equation*}
which implies, comparing with \ref{low},
\begin{equation*}
    \frac{c_0}{4}\delta^{-1+\eps}\leq p^{-m\kappa}\sum_{\theta\in\Lambda_\delta}\sum_{\widetilde{T}\in\mathcal{T}_{\theta,\mathrm{heavy}}}(\#\{T\in\mathbb{W}_\theta:T\subseteq\widetilde{T}\})\mathds{1}_{\widetilde{T}}(x),\quad x\in F.
\end{equation*}
From the upper bound \ref{trivpacketct}, we have on $F$
\begin{equation}\label{dilconc}
    \delta^{-1+\eps+\sqrt{\eps}}\leq p\delta^{-1+\eps}p^{-\kappa}\leq p\frac{4}{c_0}\sum_{\theta\in\Lambda_\delta}\sum_{\tilde{T}\in\mathcal{T}_{\theta,\mathrm{heavy}}}\mathds{1}_{\tilde{T}}(x),\quad x\in F.
\end{equation}

Observe that
\begin{equation*}
    \delta^{-1+\eps+\sqrt{\eps}}=(\delta^{1-\sqrt{\eps}})^{-1+\frac{\eps}{1-\sqrt{\eps}}},
    \end{equation*}
so that the dilated arrangement $\{\widetilde{T}\in\mathcal{T}_{\theta,\mathrm{heavy}}:\theta\in\Lambda_\delta\}$ satisfies our Kakeya hypothesis with $\eps$ replaced by $\tilde{\eps}=\frac{\eps}{1-\sqrt{\eps}}$, $\delta$ replaced by $\delta^{1-\sqrt{\eps}}$, and constant $\frac{c_0}{4p}$. By the induction hypothesis, we obtain the estimate
\begin{equation*}
    \#\bigcup_\theta\mathcal{T}_{\theta,\mathrm{heavy}} \geq \widetilde{C}_{n,p,\eps}(c_0)\nu(\Q_p^n)c_\alpha^{\delta_0}(\nu)^{-1}\delta^{(1-\sqrt{\eps})(-1-\alpha)}\delta^{10^{10}(1-\sqrt{\eps})(\sqrt{\frac{\eps}{1-\sqrt{\eps}}})}.
\end{equation*}
Recall that, if $\widetilde{T}_1\in\mathcal{T}_{\theta_1,\mathrm{heavy}}$ and $\widetilde{T}_2\in\mathcal{T}_{\theta_2,\mathrm{heavy}}$ and $T\in\mathbb{W}$ are such that $T\subseteq\widetilde{T}_1\cap\widetilde{T}_2$, then $|\theta_1-\theta_2|_p\leq p^k\delta$. Thus we may bound
\begin{equation*}
    \#\mathbb{W}\geq p^{\kappa m}\delta^\eps\frac{c_0}{4p}\#\bigcup_{\theta\in\Lambda_\delta}\mathcal{T}_{\theta,\mathrm{heavy}}.
\end{equation*}
Combining the previous two displays,
\begin{equation*}
    \begin{split}
    \#\mathbb{W}&\geq\widetilde{C}_{n,p,\eps}(c_0)\nu(\Q_p^n)c_\alpha^{\delta_0}(\nu)^{-1}\delta^{(1-\sqrt{\eps})(-1-\alpha)}\delta^{\eps-m\sqrt{\eps}}\delta^{10^{10n}\sqrt{\eps-\eps\sqrt{\eps}}}\\
    &=\widetilde{C}_{n,p,\eps}(c_0)\nu(\Q_p^n)c_\alpha^{\delta_0}(\nu)^{-1}\delta^{-1-\alpha}\delta^{(\alpha-m)\sqrt{\eps}}\delta^{\eps+10^{10n}\sqrt{\eps-\eps\sqrt{\eps}}}\\
    &\geq\widetilde{C}_{n,p,\eps}(c_0)\nu(\Q_p^n)c_\alpha^{\delta_0}(\nu)^{-1}\delta^{-1-\alpha}\delta^{(\alpha-m)\sqrt{\eps}}\delta^{10^{10n}\sqrt{\eps}}
    \end{split}.
\end{equation*}
Since $\alpha<m$, we obtain
\begin{equation*}
    \#\mathbb{W}\geq\widetilde{C}_{n,p,\eps}(c_0)\nu(\Q_p^n)c_\alpha^{\delta_0}(\nu)^{-1}\delta^{-1-\alpha+10^{10n}\sqrt{\eps}}.
\end{equation*}
Recalling the form of $\widetilde{C}_{n,p,\eps}(c_0)$, we are done.

Next, we assume \ref{high} holds. For each $\theta$ write $g_\theta=\sum_{T\in\mathbb{W}_\theta}(\mathds{1}_{T}*(\mathds{1}_{\delta\Z_p^n}-\psi_\delta)^\vee)(x)$. Then
\begin{equation*}
    \hat{g}_\theta(\xi)=\delta^m\sum_{T\in\mathbb{W}_\theta}\chi(-c_{T}\cdot\xi) \mathds{1}_{\tau_\theta}(\xi)(\mathds{1}_{\delta\Z_p^n}-\psi_\delta)(\xi).
\end{equation*}
Write $f_\theta(x)=g_\theta(\delta^{-1}x)$. Then the preceding display shows that $\hat{f}_\theta$ is supported in $\delta^{-1}\tau_\theta\setminus p^k\Z_p$. If we further decompose
\begin{equation*}
    f_\theta=\sum_{j=0}^{k-1}f_\theta*\mathds{1}_{p^j\Z_p^n\setminus p^{j+1}\Z_p^n}^\vee,
\end{equation*}
and notice that each $x\mapsto (f_\theta*\mathds{1}_{p^j\Z_p^n\setminus p^{j+1}\Z_p^n}^\vee)(p^{-j}x)$ satisfies the hypotheses of Prop. \ref{decprop}, then we conclude
\begin{equation*}
    \left\|\sum_{\theta\in\Lambda_\delta}f_\theta*\mathds{1}_{p^j\Z_p^n\setminus p^{j+1}\Z_p^n}^\vee\right\|_{L^{q_n}(\Q_p^n)}\leq D_{n,p,\eps/2}\delta^{-1+\frac{n-m+1}{q_n}-\eps/2}\left(\sum_{\theta\in\Lambda_\delta}\|f_\theta*\mathds{1}_{p^j\Z_p^n\setminus p^{j+1}\Z_p^n}^\vee\|_{L^{q_n}(\Q_p^n)}^{q_n}\right)^{1/q_n}.
\end{equation*}
By the triangle inequality and Young's convolution inequality, we see that
\begin{equation*}
    \left\|\sum_{\theta\in\Lambda_\delta}f_\theta\right\|_{L^{q_n}(\Q_p^n)}\leq \frac{4}{\eps}D_{n,p,\eps/2} \delta^{-1+\frac{n-m+1}{q_n}-\eps}\left(\sum_{\theta\in\Lambda_\delta}\|f_\theta\|_{L^{q_n}(\Q_p^n)}^{q_n}\right)^{1/q_n},
\end{equation*}
using $k=-\log_p\delta\leq \frac{2}{\eps}\delta^{-\eps/2}$.

Rescaling both sides of the previous display, we reach the estimate
\begin{equation}\label{decapp}
    \left\|\sum_{\theta\in\Lambda_\delta}g_\theta\right\|_{L^{q_n}(\Q_p^n)}\leq \frac{4}{\eps}D_{n,p,\eps/2} \delta^{-1+\frac{n-m+1}{q_n}-\eps}\left(\sum_{\theta\in\Lambda_\delta}\|g_\theta\|_{L^{q_n}(\Q_p^n)}^{q_n}\right)^{1/q_n}.
\end{equation}
For each $\theta\in\Lambda_\delta$,
\begin{equation*}
    \|g_\theta\|_{L^{q_n}(\Q_p^n)}=\left\|(\mathds{1}_{\delta\Z_p^n}-\psi_\delta)^\vee*\sum_{T\in\mathbb{W}_\theta}\mathds{1}_{T}\right\|_{L^{q_n}(\Q_p^n)}\leq\|(\mathds{1}_{\delta\Z_p^n}-\psi_\delta)^\vee\|_{L^1(\Q_p^n)}\left\|\sum_{T\in\mathbb{W}_\theta}\mathds{1}_{T}\right\|_{L^{q_n}(\Q_p^n)}.
\end{equation*}
By the definition of the family $\mathbb{W}_\theta$,
\begin{equation*}
    \left\|\sum_{T\in\mathbb{W}_\theta}\mathds{1}_{T}\right\|_{L^{q_n}(\Q_p^n)}=(\#\mathbb{W}_\theta)^{\frac{1}{q_n}}\delta^{\frac{m}{q_n}}.
\end{equation*}
Since $\psi_\delta(\xi)=\mathds{1}_{B(0,p^{-\kappa})}(\delta^{-1}\xi)$ and $\mathds{1}_{\delta\Z_p^n}(\xi)=\mathds{1}_{\Z_p^n}(\delta^{-1}\xi)$, an application of change-of-variable reveals
\begin{equation*}
    \|(\mathds{1}_{\delta\Z_p^n}-\psi_\delta)^\vee\|_{L^1(\Q_p^n)}\leq 2,
\end{equation*}
and thus
\begin{equation}\label{rhsexp}
    \left(\sum_{\theta\in\Lambda_\delta}\|g_\theta\|_{L^{q_n}(\Q_p^n)}^{q_n}\right)^{1/q_n}\leq 2\delta^{\frac{m}{q_n}}(\#\mathbb{W})^{\frac{1}{q_n}}.
\end{equation}

Since $\frac{c_0}{2}\delta^{-1+\eps}\leq|\sum_{T\in\mathbb{W}}\mathds{1}_{T}*(\mathds{1}_{\delta\Z_p^n}-\psi_\delta)^\vee(x)|$ for all $x\in F$, we have that
\begin{equation*}
    \frac{c_0}{2}\delta^{-1+\eps}\leq\left|\sum_{\theta\in\Lambda_\delta}g_\theta(x)\right|,\quad\forall x\in F,
\end{equation*}
so that
\begin{equation*}
    (c_0/2)^{q_n}\delta^{-q_n+q_n\eps}\nu(F)\leq\int\left|\sum_{\theta\in\Lambda_\delta}g_\theta\right|^{q_n}d\nu.
\end{equation*}
Note that $|\sum_{\theta\in\Lambda_\delta} g_\theta|$ is constant on balls of radius $\delta$; thus, using $0<c_\alpha^{\delta_0}(\nu)<\infty$ and $\delta>\delta_0$,
\begin{equation*}
    \int\left|\sum_{\theta\in\Lambda_\delta}g_\theta\right|^{q_n}d\nu\leq c_\alpha^{\delta_0}(\nu)\delta^{\alpha-n}\int\left|\sum_{\theta\in\Lambda_\delta}f_\theta\right|^{q_n}d\mu,
\end{equation*}
so that, using also $\nu(F)\gtrsim\nu(\Q_p^n)$,
\begin{equation}\label{dimest}
    2^{-q_n-1}c_0^{q_n}c_\alpha^{\delta_0}(\nu)^{-1}\nu(\Q_p^n)\delta^{-q_n+q_n\eps}\leq\delta^{\alpha-n}\int\left|\sum_{\theta\in\Lambda_\delta}g_\theta\right|^{q_n}.
\end{equation}

Collecting estimates \ref{decapp}, \ref{rhsexp}, and \ref{dimest}, we conclude
\begin{equation*}
    \#\mathbb{W}\geq 2^{-4q_n-1}\eps^{q_n}D_{n,p,\eps/2}^{-q_n}c_0^{q_n}\nu(\Q_p^n)c_\alpha^{\delta_0}(\nu)^{-1}\delta^{-\alpha-1+2q_n\eps}.
\end{equation*}
As $2q_n\eps\leq 10^{10n}\sqrt{\eps}$ for every $0<\eps\leq 1$, we are done.

\end{proof}

\section{Decoupling bound for restricted projections}\label{decsection}

In this section we prove Prop. \ref{decprop}. We will do so by adapting the decoupling procedure of \cite{gan2022restricted} to the $p$-adic setting. We will take for granted $p$-adic decoupling for moment curves in dimensions $n<p$; for a proof of the latter, see Corollary \ref{convexdec} in Appendix B. We emphasize that this decoupling theorem (together with elementary rescaling arguments) will be the only Fourier-analytic inputs for this section. Instead, we will be primarily concerned with a decomposition of the Fourier support of $\sum_{\theta\in\Lambda_\delta}f_\theta$ into subsets over which the decoupling theorem may be used.

The decoupling procedure described in this section is virtually identical to the real setting. As a consequence, we will present a very terse accounting of the analysis; the interested reader may compare with \cite{gan2022restricted} for motivation. At the end, we state the output of the algorithm and observe that the estimate obtained suffices to prove Proposition \ref{decprop}.

We may assume that $\delta$ is restricted to sufficiently regular powers of $p$, to facilitate taking various roots; similarly, we assume that $\eps$ is a sufficiently divisible reciprocal of an integer. To this end, write $\kappa=(n!)^{2n}$ and assume that $\eps=\frac{1}{\ell\kappa}$ for some $\ell\in\N_{\geq 2}$. We assume also that $\delta\in p^{-\kappa^2\N}$. After we have established this special case, we will be able to conclude the general statement via trivial estimates.

We begin by defining a decomposition of frequency space which will facilitate the proof of Proposition \ref{decprop}. These are adapted to the support of the Fourier transform of $f$, the function to be estimated. See Figure \ref{fig:decomp} for an illustration of the geometry, when regarded over $\R$.

For each subset $J\subseteq\Z_p$ and $1\leq m_1\leq m$, define
\begin{equation*}
    \Omega_{J}=\Big\{\sum_{j=1}^n\lambda_j\gamma^{(j)}(\theta):\theta\in J\cap\Lambda_\delta,\lambda_j\in\Z_p\,\,\forall j,\max_{1\leq j\leq n}|\lambda_j|_p=1,\,|\lambda_j|_p\leq\delta\,\,\forall j\in(m,n]\Big\}\subseteq\Q_p^n
\end{equation*}
and
\begin{equation*}
    \Omega_{J,m_1}=\Big\{\sum_{j=1}^n\lambda_j\gamma^{(j)}(\theta)\in\Omega_J:|\lambda_{m_1}|_p=1,|\lambda_j|_p<1\,\,\forall j\in (m_1,m]\Big\},
\end{equation*} 
so that $\{\Omega_{J,m_1}\}_{1\leq m_1\leq m}$ partition $\Omega_J$ for each $J$.

For each $s_1\in p^{-\eps^{-1}\N}$ with $\delta^{\frac{1}{n-m_1}}\leq s_1<1$, write
\begin{equation*}
    \begin{split}
        \Omega_{J,m_1,s_1}=\Big\{&\sum_{j=1}^n\lambda_j\gamma^{(j)}(\theta)\in\Omega_{J,m_1}:(s_1=\delta^{\frac{1}{n-m_1}}\text{ or }\exists\iota\in[1,m-m_1]\text{ s.t. }s_1^\iota\leq|\lambda_{m_1+\iota}|_p),\\
        &\forall\iota\in[1,m-m_1]\,p^{\iota\eps^{-1}} s_1^{\iota}>|\lambda_{m_1+\iota}|_p\Big\}
    \end{split}
\end{equation*}
so that
\begin{equation*}
    \Omega_{J,m_1}=\bigcup_{\delta^{\frac{1}{n-m_1}}\leq s_1<p^{-1}}\Omega_{J,m_1,s_1}.
\end{equation*}
We remark that $\Omega_{J,m_1}$ is essentially a segment of the rim of a thick cone over an $(n-m_1)$-dimensional moment curve, and each $\Omega_{J,m_1,s_1}$ is a thin slice of that cone to facilitate the standard cone-decoupling trick of comparing with a cylinder. See Figure \ref{fig:decomp} for an illustration. We further decompose $\Omega_{J,m_1,s_1}$ by: for each tuple $\mathfrak{R}=(R_1,\ldots,R_{m_1-1})\in\mathcal{P}(\Z_p,s_1^\eps)^{m_1-1}$ and each $B\in\mathcal{P}(\Z_p\setminus p\Z_p,s_1^\eps)$,
\begin{equation*}
    \Omega_{J,m_1,s_1}^{B,\mathfrak{R}}=\Big\{\sum_{j=1}^n\lambda_j\gamma^{(j)}(\theta):\lambda_j\in R_j\,\forall j\in[1,m_1),\lambda_{m_1}\in B\Big\}.
\end{equation*}

We will eventually decouple along these regions; to this end, for each $k\in\N$, write
\begin{equation*}
    \mathfrak{D}_k=\frac{k(k+1)+2}{2},
\end{equation*}
so that the $\ell^{q_n}L^{q_n}$ decoupling constant for the $k$-dimensional moment curve at scale $\delta$ has size $\lesssim_\eps\delta^{-(1-\frac{\mathfrak{D}_k}{q_n})-\eps}$ for each $k\leq n$.

With this established, we now proceed to describing the proof of Prop. \ref{decprop}.

\begin{proof}[Proof of Prop. \ref{decprop}]

We first observe that the sets above describe the Fourier support of $f$. Indeed, $\hat{f}_\theta$ is supported in the set
\begin{equation*}
    A_{\theta,1,\delta}^{-\top}(\Z_p^n)\setminus p\Z_p^n
\end{equation*}
where we again are adopting the notation
\begin{equation*}
    A_{\theta,\alpha,\beta}=[\alpha^{-1}\gamma^{(1)}(\theta),\ldots,\alpha^{-1}\gamma^{(m)}(\theta),\beta^{-1}\gamma^{(m+1)}(\theta),\ldots,\beta^{-1}\gamma^{(n)}(\theta)],\quad (\theta\in\Z_p,\alpha,\beta\in\Q_p).
\end{equation*}
 Consequently, $\hat{f}_\theta$ is supported in $\Omega_{\{\theta\}}$, so the preceding decomposition applies.

By H\"older,\footnote{Indeed, consider the weighting of each particular configuration by $s_1^{m_1+1}$.} we have that one of the following holds: either there exists $m_1<m$, $\delta^{\frac{1}{n-m_1}}\leq s_1\in p^{-\eps^{-1}\N}$, $J\in\mathcal{P}(\Z_p,s_1^\eps)$, $B\in\mathcal{P}(\Z_p\setminus p\Z_p,s_1^\eps)$, $\mathfrak{R}\in\mathcal{P}(\Z_p,s_1^\eps)^{m_1-1}$ such that
\begin{equation}\label{regcase}
    \left\|\sum_{\theta\in \Lambda_\delta}f_\theta\right\|_{L^{q_n}(\Q_p^n)}\leq m(\log_p\delta^{-1})s_1^{-(m_1+1)\eps}\left\|\sum_{\theta\in J\cap\Lambda_\delta}\mathcal{P}_{\Omega_{J,m_1,s_1}^{B,\mathfrak{R}}}f_\theta\right\|_{L^{q_n}(\Q_p^n)}
\end{equation}
or else we set $m_1=m,s_1=\delta^{\frac{1}{n-m_1}}$, and there is $J\in\mathcal{P}(\Z_p,s_1^\eps)$ such that
\begin{equation}\label{rimcase} 
    \left\|\sum_{\theta\in \Lambda_\delta}f_\theta\right\|_{L^{q_n}(\Q_p^n)}\leq m(\log_p\delta^{-1})s_1^{-(m_1+1)\eps}\left\|\sum_{\theta\in J\cap\Lambda_\delta}\mathcal{P}_{\Omega_{J,m}}f_\theta\right\|_{L^{q_n}(\Q_p^n)}.
\end{equation}
We will focus on the case that \eqref{regcase} holds, and abbreviate $F=\sum_{\theta\in J\cap\Lambda_\delta}\mathcal{P}_{\Omega_{J,m_1,s_1}^{B,\mathfrak{R}}}f_\theta$.
We will demonstrate the following:

\begin{lemma}\label{decstep} Suppose that $s_1=p^{-\ell\eps^{-1}}$ for some $\ell\geq 2$. For any 
\begin{equation*}
    0\leq k\leq k_*:=(n-m_1)\lfloor\eps^{-1}-2-\frac{\eps^{-1}}{\ell}\rfloor
\end{equation*}
and any $L\in\mathcal{P}(J,s_1^{(1+\frac{k}{n-m_1})\eps})$, we have that
\begin{equation*}
    \left\|\sum_{\theta\in L\cap\Lambda_\delta}F_\theta\right\|_{L^{q_n}(\Q_p^n)}\leq C_{n-m_1,(n-m_1)\eps}s_1^{-\frac{\eps^2}{n-m_1}}s_1^{-\frac{\eps}{n-m_1}(1-\frac{\mathcal{D}_{n-m_1}}{q_n})}\left(\sum_{I\in\mathcal{P}\Big(L,s_1^{(1+\frac{k+1}{n-m_1})\eps}\Big)}\left\|\sum_{\theta\in I\cap\Lambda_\delta}F_\theta\right\|_{L^{q_n}(\Q_p^n)}^{q_n}\right)^{1/q_n},
\end{equation*}
where $C_{n-m_1,(n-m_1)\eps}$ is as in Remark \ref{rectsetfamdec}.

\end{lemma}

\begin{proof}[Proof of Lemma \ref{decstep}]

By applying parabolic rescaling, it suffices to assume that $L=B(0,s_1^{(1+\frac{k}{n-m_1})\eps})$ and $B=B(1,s_1^\eps)$. Then, for any $\theta,\lambda_1,\ldots,\lambda_n$ as in the definition of $\Omega_{L,m_1,s_1}^{B,\mathfrak{R}}$, we have the relations
\begin{equation*}
    \sum_{j=1}^n\lambda_j\gamma_{\iota}^{(j)}(\theta)=\sum_{j=1}^\iota\lambda_j\frac{\theta^{\iota-j}}{(\iota-j)!}\quad(1\leq\iota\leq m_1),
\end{equation*}
\begin{equation*}
    \left|\sum_{j=1}^n\lambda_j\gamma_\iota^{(j)}(\theta)-\frac{(\lambda_{m_1}\theta)^{\iota-m_1}}{(\iota-m_1)!}\right|_p\leq s_1^{(1+\frac{k}{n-m_1})\eps(\iota-m_1)+\eps}\quad(m_1<\iota\leq n).
\end{equation*}
The second of these follows from the inequality $|\theta^{\iota-j}|_p\leq s_1^{(1+\frac{k}{n-m_1})\eps(\iota-j)}$ for $j\leq\iota$, the inequality $|\lambda_{m_1}-(\lambda_{m_1})^{\iota-m_1}|_p\leq s_1^\eps$, the inequality $|\lambda_j|_p\leq p^{\eps^{-1}(j-m_1)}s_1^{j-m_1}$ for $m_1<j\leq\iota$, the ultrametric triangle inequality, and the fact that $(1+\frac{k_*}{n-m_1})\eps\leq 1-\eps-\frac{1}{\ell}$.

Consequently, it holds that
\begin{equation*}
    \begin{split}
    \Omega_{L,m_1,s_1}^{i,\mathfrak{r}}\subseteq\Big\{\sum_{\iota=1}^{m_1}\rho_\iota e_\iota +\sum_{\iota=m_1+1}^n\left(\frac{\theta^{\iota-m_1}}{(\iota-m_1)!}+\rho_\iota\right)e_\iota &:\rho_\iota\in B(b_{r_\iota},s_1^{\eps})\,\,(\iota<m_1),\rho_{m_1}\in B(1,s_1^\eps),\\
    &|\rho_\iota|_p\leq s_1^{(1+\frac{k}{n-m_1})\eps(\iota-m_1)+\eps} (\iota>m_1),\theta\in L\cap\Lambda_\delta\Big\}.
    \end{split}
\end{equation*}
Applying time rescaling $\theta\mapsto s_1^{(1+\frac{k}{n-m_1})\eps}\theta$ (n.b. that this is regarded as a product of two elements of $\Q_p$), and decoupling over the $(n-m_1)$-dimensional moment curve, we  conclude
\begin{equation*}
    \left\|\sum_{\theta\in L}F_\theta\right\|_{L^{q_n}(\Q_p^n)}\leq C_{n-m_1,(n-m_1)\eps} s_1^{-\frac{\eps^2}{n-m_1}}s_1^{-\frac{\eps}{n-m_1}(1-\frac{\mathfrak{D}_{n-m_1}}{q_n})}\left(\sum_{I\in\mathcal{P}(L,s_1^{(1+\frac{k+1}{n-m_1})\eps})}\big\|\sum_{\theta\in I}F_\theta\big\|_{L^{q_n}(\Q_p^n)}^{q_n}\right)^{1/q_n}.
\end{equation*}

\end{proof}

Since decoupling constants are sub-multiplicative, we may repeatedly apply Lemma \ref{decstep} to conclude
\begin{equation*}
    \|F\|_{L^{q_n}(\Q_p^n)}\leq C_{n-m_1,(n-m_1)\eps}^{\frac{n-m_1}{\eps}} s_1^{-\eps-(1-\frac{\mathfrak{D}_{n-m_1}}{q_n})}\left(\sum_{I\in\mathcal{P}(\Z_p,s_1^{(1+\frac{k_*+1}{n-m_1})\eps})}\|F_I\|_{L^{q_n}(\Q_p^n)}^{q_n}\right)^{1/q_n}.
\end{equation*}
By the triangle inequality and Cauchy-Schwarz, we arrive at the estimate
\begin{equation*}
    \|F\|_{L^{q_n}(\Q_p^n)}\leq C_{n-m_1,(n-m_1)\eps}^{\frac{n-m_1}{\eps}} p^{\eps^{-1}}s_1^{-2\eps-(1-\frac{\mathfrak{D}_{n-m_1}}{q_n})}\left(\sum_{I\in\mathcal{P}(\Z_p,s_1)}\|F_I\|_{L^{q_n}(\Q_p^n)}^{q_n}\right)^{1/q_n},
\end{equation*}
valid whenever we had $s_1\in p^{-\eps^{-1}\N_{\geq 2}}$. If instead $s_1=p^{-\eps^{-1}}$, then trivial estimates supply
\begin{equation*}
    \|F\|_{L^{q_n}(\Q_p^n)}\leq p^{\eps^{-1}}\left(\sum_{I\in\mathcal{P}(\Z_p,s_1)}\|F_I\|_{L^{q_n}(\Q_p^n)}^{q_n}\right)^{1/q_n}.
\end{equation*}
\begin{figure}
    \centering
    \includegraphics[scale=0.25]{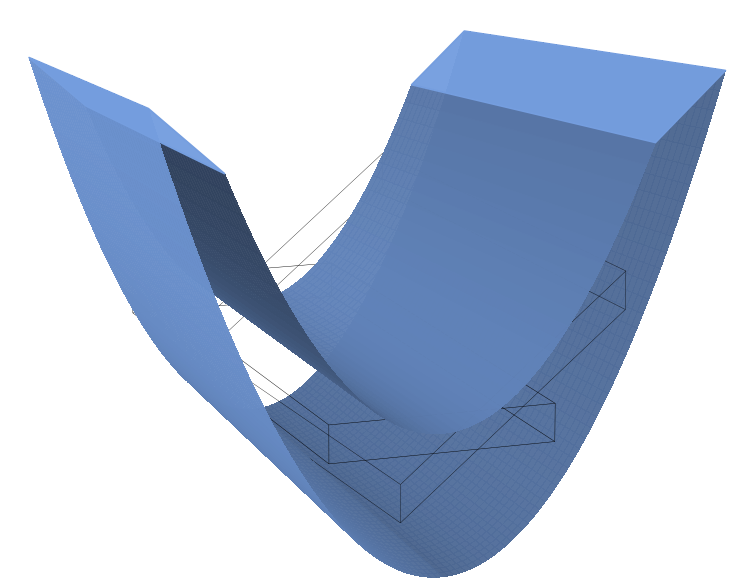}\hspace{2em}
    \includegraphics[scale=0.5]{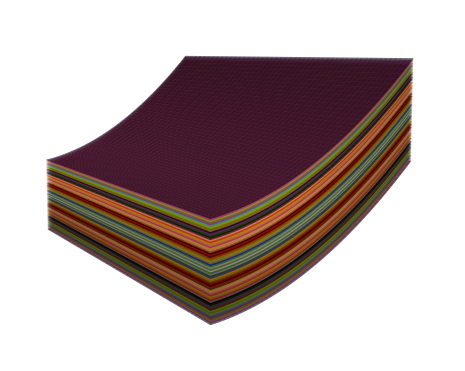}
    \caption{Left: the union of the truncated plates $\delta^{-1}\tau_\theta\cap(\Z_p^n\setminus p\Z_p^n)$ is contained in a thick neighborhood $\Omega_I$ of a cone over a lower-dimensional nondegenerate curve. Right: the decomposition $\Omega_{J,m_1}=\bigcup\Omega_{J,m_1,s_1}$ in the case $(m_1,m,n)=(1,2,3)$ over $\R$.}
    \label{fig:decomp}
\end{figure}
The remainder of the algorithm involves rescaling each $F_I$ and repeating the above procedure, by finding a new parameter $s_2$ to treat the Fourier support of $F_I$ as the cylinder over a moment curve. We summarize the inductive step in Lemma \ref{alginduction} below. Prior, we adopt the following notation, largely identical to that of \cite{gan2022restricted}. For any $m_1\in[1,m]$ and $n_1\in[m,n]$, and any $s_1\in p^{-\eps^{-1}\N}$, we set
\begin{equation*}
    \mathcal{L}_{m_1,s_1}(\xi)=(\xi_1,\ldots,\xi_{m_1}, s_1^{-1}\xi_{m_1+1},\ldots, s_1^{m_1-n}\xi_n),
\end{equation*}
\begin{equation*}
    \mathcal{R}_{m_1,s_1}(\xi)=(s_1^{1-1}\xi_1,s_1^{1-2}\xi_2,\ldots,s_1^{1-m_1}\xi_{m_1},s_1^{1-m_1}\xi_{m_1+1}\ldots,s_1^{1-m_1}\xi_n),
\end{equation*}
and
\begin{equation*}
    \mathcal{D}_{m_1,n_1}^{s_1}(\xi)_j=\begin{cases}
        s_1^{j-1}\xi_j & 1\leq j\leq m_1,\\
        p^{-(j-m_1)\eps^{-1}}s_1^{m_1-1}\xi_j & m_1<j\leq m,\\
        s_1^{m_1-1}\xi_j & m<j\leq n_1,\\
        \xi_j & n_1<j\leq n.
    \end{cases}
\end{equation*}
If $\mathbf{s}_J=(s_1,\ldots,s_J),\mathbf{m}_j=(m_1,\ldots,m_J)$, and $\mathbf{n}_J=(n_1,\ldots,n_J)$ are entrywise as abovve, then we write
\begin{equation*}
    \mathcal{L}_{\mathbf{m}_J,\mathbf{s}_J}=\mathcal{L}_{m_J,s_J}\circ\cdots\circ\mathcal{L}_{m_1,s_1},
\end{equation*}
\begin{equation*}
    \mathcal{R}_{\mathbf{m}_J,\mathbf{s}_J}=\mathcal{R}_{m_J,s_J}\circ\cdots\circ\mathcal{R}_{m_1,s_1},
\end{equation*}
\begin{equation*}
    \mathcal{D}_{\mathbf{m}_J,\mathbf{n}_J}^{\mathbf{s}_J}=\mathcal{D}_{m_J,n_J}^{s_J}\circ\cdots\circ\mathcal{D}_{m_1,n_1}^{s_1}.
\end{equation*}
We will use the abbreviations
\begin{equation*}
    \mathbf{s}_J^\circ=\prod_{j=1}^Js_j,\quad\mathbf{s}_J^{-\circ}=(\mathbf{s}_J^\circ)^{-1}.
\end{equation*}
When a given tuple $\mathbf{s}_J$ is already understood and $0\leq j<J$, we'll write $\mathbf{s}_j$ for the corresponding initial segment of $\mathbf{s}_J$. We'll also write
\begin{equation*}
    \gamma_{m_1,s_1}=\mathcal{R}_{m_1,s_1}\circ\gamma,
\end{equation*}
and
\begin{equation*}
    \gamma_{\mathbf{m}_J,\mathbf{s}_J}=\mathcal{R}_{\mathbf{m}_J,\mathbf{s}_J}\circ\gamma;
\end{equation*}
observe then that
\begin{equation}\label{rescaledg}
    \gamma_{\mathbf{m}_J,\mathbf{s}_J}(\theta)=\mathbf{s}_J^\circ\mathcal{L}_{\mathbf{m}_J,\mathbf{s}_J}^{-1}\gamma\left(\mathbf{s}_J^{-\circ}\theta\right).
\end{equation}
The tuples will need to satisfy the following compatibility relation: we write $s_{J+1}\in\mathrm{Adap}_{\mathbf{m}_{J+1},\mathbf{n}_{J+1}}^{\mathbf{s}_J}$ for quantities $s_{J+1}\in p^{-\eps^{-1}\N}$, and call them \textit{adapted}, if
\begin{equation}\label{adapted}
    \left(\delta\prod_{j=1}^Js_j^{-(n_{J}-1-m_j)}\right)^{\frac{1}{n_{J+1}-m_{J+1}}}\leq s_{J+1}<1.
\end{equation}
We also set out the following regions in frequency space: given tuples $\mathbf{s}_{J},\mathbf{m}_{J},\mathbf{n}_{J}$, and $\theta\in\Z_p$ with $|\theta|_p\leq\mathbf{s}_J^\circ$, we write
\begin{equation*}
    \begin{split}
        \Omega_{\theta,m_1,s_1}^{\mathrm{res}}=\Big\{\,\,&\Big[\gamma_{m_1,s_1}^{(1)}\big(s_1\theta\big),\ldots,\gamma_{m_1,s_1}^{(n)}\big(s_1\theta\big)\Big]\cdot\mathcal{D}_{m_1,n}^{s_1}(\bm{\lambda}):\\
    &\forall\iota\in[1,n]\,\,|\lambda_\iota|_p\leq 1,|\lambda_{m_1}|_p=1,\\
    &\exists\iota\in[1,m-m_{1}]\text{ s.t. } p^{-\iota\eps^{-1}}<|\lambda_{m_{1}+\iota}|_p,\\
    &\forall\iota\in[1,m-m_{1}]\,\,|\lambda_{m_{1}+\iota}|_p\leq 1,\\
    &\forall\iota\in[1,n-m]\,\,|\lambda_{m+\iota}|_p\leq\delta s_1^{-\iota}\hspace{5em}\Big\}.
    \end{split}
\end{equation*}
Here, we emphasize the convention that each $\gamma^{(j)}$ is a column vector, $[\gamma^{(1)},\ldots,\gamma^{(n)}]$ denotes the matrix whose $j^{th}$ column is $\gamma^{(j)}$, ${\bm{\lambda}}$ is the column vector $(\lambda_1,\ldots,\lambda_n)$, and the $\cdot$ denotes matrix multiplication. We similarly write
\begin{equation*}
    \begin{split}
    \Omega_{\theta,\mathbf{m}_{J},\mathbf{s}_{J},\mathbf{n}_{J}}^{\mathrm{res}}=\Big\{\,\,&\Big[\gamma_{\mathbf{m}_J,\mathbf{s}_J}^{(1)}\big(\mathbf{s}_J^\circ\theta\big),\ldots,\gamma_{\mathbf{m}_J,\mathbf{s}_J}^{(n)}\big(\mathbf{s}_J^\circ\theta\big)\Big]\cdot\mathcal{D}_{\mathbf{m}_J,\mathbf{n}_J}^{\mathbf{s}_J}(\bm{\lambda}):\\
    &\forall\iota\in[1,n]\,\,|\lambda_\iota|_p\leq 1,|\lambda_{m_J}|_p=1,\\
    &\exists\iota\in[1,m-m_{J}]\text{ s.t. }p^{-\iota\eps^{-1}}<|\lambda_{m_{J}+\iota}|_p,\\
    &\forall\iota\in[1,m-m_{J}]\,\,|\lambda_{m_{J}+\iota}|_p\leq 1,\\
    &\forall\iota\in[1,n_J-m]\,\,|\lambda_{m+\iota}|_p\leq\delta \prod_{j=1}^Js_j^{-(m+\iota-m_j)}\quad\quad\quad\Big\},
    \end{split}
\end{equation*}
and, for each choice of $s_{J+1},m_{J+1},n_{J+1}$, we write
\begin{equation*}
    \begin{split}
        \Omega_{\theta,\mathbf{m}_{J},\mathbf{s}_{J},\mathbf{n}_{J}}^{\mathrm{res},m_{J+1},s_{J+1},n_{J+1}}=\Big\{\,\,&\Big[\gamma_{\mathbf{m}_J,\mathbf{s}_J}^{(1)}\big(\mathbf{s}_J^\circ\theta\big),\ldots,\gamma_{\mathbf{m}_J,\mathbf{s}_J}^{(n)}\big(\mathbf{s}_J^\circ\theta\big)\Big]\cdot\mathcal{D}_{\mathbf{m}_J,\mathbf{n}_J}^{\mathbf{s}_J}(\bm{\lambda})\in\Omega_{\theta,\mathbf{m}_{J},\mathbf{s}_{J},\mathbf{n}_{J}}^{\mathrm{res}}:\\
        &|\lambda_{m_{J+1}}|_p=1,\\
        &\Big(s_{J+1}=\left(\delta\prod_{j=1}^Js_j^{-(n_{J}-1-m_j)}\right)^{\frac{1}{n_{J+1}-m_{J+1}}}\\
        &\,\,\,\,\,\text{ or }\exists\iota\in[1,m-m_{J+1}]\text{ s.t. }1\leq|\lambda_{m_{J+1}+\iota}|_p\Big),\\
        &\forall\iota\in[1,m-m_{J+1}]\,\,|\lambda_{m_{J+1}+\iota}|_p< s_{J+1}^\iota p^{\iota\eps^{-1}}\hspace{7em}\Big\}.
    \end{split}
\end{equation*}

We immediately motivate the definition of these regions by the following lemma.
\begin{lemma}
    Suppose $h$ has Fourier support in $\Omega_{\theta,m_1,s_1}$, and $|\theta|_p\leq s_1$. Then $h\circ\mathcal{L}_{m_1,s_1}$ has Fourier support in $\Omega_{\theta,m_1,s_1}^{\mathrm{res}}$. More generally, if $h\circ\mathcal{L}_{\mathbf{m}_J,\mathbf{s}_J}$ has Fourier support in $\Omega_{\theta,\mathbf{m}_J,\mathbf{s}_J,\mathbf{n}_J}^{\mathrm{res},m_{J+1},s_{J+1},n_{J+1}}$, then $h\circ\mathcal{L}_{\mathbf{m}_{J+1},\mathbf{s}_{J+1}}$ has Fourier support in $\Omega_{\theta,\mathbf{m}_{J+1},\mathbf{s}_{J+1},\mathbf{n}_{J+1}}^{\mathrm{res}}$.
\end{lemma}
\begin{proof}
    By induction on $J$. Consider the base case $J=1$. Relabelling $\lambda_{m_1+\iota}\mapsto s_1^{\iota}\lambda_{m_1+\iota}$, we see that $h\circ\mathcal{L}_{m_1,s_1}$ is Fourier supported in the set
    \begin{equation*}
        \begin{split}
            \Big\{\,\,&\mathcal{L}_{m_1,s_1}^{-1}\cdot\left[\gamma^{(1)}(\theta),\ldots,\gamma^{(n)}(\theta)\right]\cdot\mathcal{L}_{m_1,s_1}(\bm{\lambda}):|\lambda_j|\leq 1\,\forall j<m_1,\,|\lambda_{m_1}|_p=1,\\
            &(s_1=\delta^{\frac{1}{n-m_1}}\text{ or }\exists\iota\in[1,m-m_1]\text{ s.t. }1\leq |\lambda_{m_1+\iota}|_p),\\
            &\forall\iota\in[1,m-m_1]\,|\lambda_{m_1+\iota}|_p<p^{\iota\eps^{-1}},\\
            &\forall\iota\in[1,n-m]\,|\lambda_{m+\iota}|\leq s_1^{-\iota}\delta\hspace{18em}\Big\}
        \end{split}
    \end{equation*}
    For a particular (column) vector $\lambda$, we may manipulate the corresponding sum via \ref{rescaledg} as
    \begin{equation*}
        \mathcal{L}_{m_1,s_1}^{-1}\cdot\left[\gamma^{(1)}(\theta),\ldots,\gamma^{(n)}(\theta)\right]\cdot\mathcal{L}_{m_1,s_1}(\bm{\lambda})=\left[s_1^{1-1}\gamma_{m_1,s_1}^{(1)}(s_1\theta),\ldots,s_1^{n-1}\gamma_{m_1,s_1}^{(n)}(s_1\theta)\right]\cdot\mathcal{L}_{m_1,s_1}(\bm{\lambda}),
    \end{equation*}
    which we may write as
    \begin{equation*}
        \left[\gamma_{m_1,s_1}^{(1)}(s_1\theta),\ldots,\gamma_{m_1,s_1}^{(n)}(s_1\theta)\right]\cdot\mathcal{R}_{m_1,s_1}^{-1}(\bm{\lambda}).
    \end{equation*}
    Finally, for each $\iota\in[1,m-m_1]$, we relabel again $\lambda_{m_1+\iota}\mapsto p^{\iota\eps^{-1}}\lambda_{m_1+\iota}$; from the definition of $\mathcal{D}_{m_1,n}^{s_1}$, the desired result follows.

    In the general case, relabelling $\lambda_{m_{J+1}+\iota}\mapsto s_{J+1}^{\iota}\lambda_{m_{J+1}\iota}$, we see that $h\circ\mathcal{L}_{\mathbf{m}_{J+1},\mathbf{s}_{J+1}}$ is Fourier supported in the set
    \begin{equation*}
        \begin{split}
            \Big\{\,\,&\mathcal{L}_{m_{J+1},s_{J+1}}^{-1}\cdot\left[\gamma_{\mathbf{m}_J,\mathbf{s}_J}^{(1)}(\mathbf{s}_J^\circ\theta),\ldots,\gamma_{\mathbf{m}_J,\mathbf{s}_J}^{(n)}(\mathbf{s}_J^\circ\theta)\right]\cdot\mathcal{L}_{m_{J+1},s_{J+1}}(\mathcal{D}_{\mathbf{m}_J,\mathbf{n}_J}^{\mathbf{s}_J}(\bm{\lambda})):\\
            &|\lambda_j|\leq 1\,\forall j<m_{J+1},\,|\lambda_{m_{J+1}}|_p=1,\\
            &\Big(s_{J+1}=\left(\delta\prod_{j=1}^Js_j^{-(n_{J}-1-m_j)}\right)^{\frac{1}{n_{J+1}-m_{J+1}}}\text{ or }\exists\iota\in[1,m-m_{J+1}]\text{ s.t. }1\leq |\lambda_{m_{J+1}+\iota}|_p\Big),\\
            &\forall\iota\in[1,m-m_{J+1}]\,|\lambda_{m_{J+1}+\iota}|_p<p^{\iota\eps^{-1}},\\
            &\forall\iota\in[1,n_{J+1}-m]\,\,|\lambda_{m+\iota}|_p\leq\delta \prod_{j=1}^{J+1}s_j^{-(m+\iota-m_j)}\hspace{17em}\Big\}.
        \end{split}
    \end{equation*}
    Again, we may manipulate the corresponding sum via \ref{rescaledg} as
    \begin{equation*}
        \begin{split}
            \mathcal{L}_{m_{J+1},s_{J+1}}^{-1}\cdot&\left[\gamma_{\mathbf{m}_J,\mathbf{s}_J}^{(1)}(\mathbf{s}_J^\circ\theta),\ldots,\gamma_{\mathbf{m}_J,\mathbf{s}_J}^{(n)}(\mathbf{s}_J^\circ\theta)\right]\cdot\mathcal{L}_{m_{J+1},s_{J+1}}(\mathcal{D}_{\mathbf{m}_J,\mathbf{n}_J}^{\mathbf{s}_J}(\bm{\lambda}))\\
            &=\left[s_{J+1}^{1-1}\gamma_{\mathbf{m}_{J+1},\mathbf{s}_{J+1}}^{(1)}(\mathbf{s}_{J+1}^\circ\theta),\ldots,s_{J+1}^{n-1}\gamma_{\mathbf{m}_{J+1},\mathbf{s}_{J+1}}^{(n)}(\mathbf{s}_{J+1}^\circ\theta)\right]\cdot\mathcal{L}_{m_{J+1},s_{J+1}}(\mathcal{D}_{\mathbf{m}_J,\mathbf{n}_J}^{\mathbf{s}_J}(\bm{\lambda})),
        \end{split}
    \end{equation*}
    which we may write as
    \begin{equation*}
        \left[\gamma_{m_1,s_1}^{(1)}(s_1\theta),\ldots,\gamma_{m_1,s_1}^{(n)}(s_1\theta)\right]\cdot\mathcal{R}_{m_{J+1},s_{J+1}}^{-1}(\mathcal{D}_{\mathbf{m}_J,\mathbf{n}_J}^{\mathbf{s}_J}(\bm{\lambda})).
    \end{equation*}
    Finally, for each $\iota\in[1,m-m_{J+1}]$, we relabel again $\lambda_{m_{J+1}+\iota}\mapsto p^{\iota\eps^{-1}}\lambda_{m_{J+1}+\iota}$; from the definition of $\mathcal{D}_{\mathbf{m}_{J+1},{\mathbf{n}_{J+1}}}^{\mathbf{s}_{J+1}}$, the desired result follows.
    
\end{proof}

\begin{lemma}[Inductive localization step]\label{alginduction} Assume we have an integer $J\geq 1$, a rooted tree $T$ composed of sequences $(\Theta_0,\ldots,\Theta_j)$, $(j\leq J)$, of metric balls $\Theta_i$ in $\Z_p$, such that the set of children of $\Theta_i$, $(i<J)$, is a set of the form $\mathcal{P}(\Theta_i,s_{\Theta_i})$ with $s_{\Theta_i}\in p^{-\eps^{-1}\N}\cup\{1\}$, together with labels $n_{\Theta_i},m_{\Theta_i}$ of each $\Theta_i$, for which the following axioms are satisfied.
    \begin{itemize}
        \item $n_{\Theta_0}=n,m_{\Theta_0}=m$.
        \item If $(\Theta_0,\ldots,\Theta_J)\in T$, then the associated tuples $\{s_j=s_{\Theta_{j-1}}\}_{j=1}^J$, $\{n_j=n_{\Theta_j}\}_{j=0}^J,\{m_j=m_{\Theta_j}\}_{j=0}^J$ are such that:
        \begin{equation}\label{alginc}
            \forall j\in[0,J):\text{ either }\begin{cases}
                & n_{j+1}=n_j-1 \text{ and }m_{j+1}\geq m_j\\
                &\quad\quad\text{ and }s_{j+1}^{n_j-m_j}=\delta\prod_{\eta=1}^{j}s_\eta^{-(n_j-1-m_\eta)},\\
                &\\
                \text{or} & n_{j+1}=n_j\text{ and }m_{j+1}>m_{j},\\
                &\quad\quad\text{ and }s_{j+1}^{n_j-m_j}\geq\delta\prod_{\eta=1}^js_\eta^{-(n_j-1-m_\eta)},\\
                &\\
                \text{or} & n_{j+1}=n_j=m+1\text{ and }m_{j+1}=m_j=m\\
                &\quad\quad\text{ and }\prod_{\eta=1}^js_\eta^{n_j-m_\eta}=\delta\text{ and }s_{j+1}=1.
            \end{cases}
        \end{equation}
        \item In the setting above, we also have
        \begin{equation}\label{suppcond}
            \supp\widehat{F}_{\Theta_J}\subseteq\Omega_{\Theta_J,\mathbf{m}_J,\mathbf{s}_J,\mathbf{n}_J}.
        \end{equation}

    \end{itemize}
    For $\mathbf{s}_J,\mathbf{m}_J,\mathbf{n}_J$, we will write $T_{\mathbf{s}_J,\mathbf{m}_J,\mathbf{n}_J}$ for the set of tuples $(\Theta_0,\ldots,\Theta_J)\in T$ with that associated tuple, as in the second bullet point above. We assume also that we have the upper bound
    \begin{equation}\label{algdec}
        \begin{split}
        \|F\|_{L^{q_n}(\Q_p^n)}&\leq p^{J^2n\eps^{-1}}\sum_{\mathbf{s}_J,\mathbf{m}_J,\mathbf{n}_J}\prod_{j=1}^J\left[C_{n_j-m_j,(n_j-m_j)\eps}^{\frac{n_j-m_j}{\eps}}s_j^{-(m_j+3)\eps-(1-\frac{\mathfrak{D}_{n_j-m_j}}{q_n})}\right]\\
        &\times\left(\sum_{(\Theta_0,\ldots,\Theta_J)\in T_{\mathbf{s}_J,\mathbf{m}_J,\mathbf{n}_J}}\left\|\sum_{\theta\in\Theta_J\cap\Lambda_\delta}F_\theta\right\|_{L^{q_n}(\Q_p^n)}^{q_n}\right)^{1/q_n}.
        \end{split}
    \end{equation}

    Let $(\Theta_0,\ldots,\Theta_J)\in T$ be such that
    \begin{equation}\label{algnotdone}
        n_J>m+1\quad\text{or}\quad\prod_{j=1}^Js_j^{n_J-m_j}>\delta.
    \end{equation}
    Then we may find $m_{J+1}\in[m_J,m]$, $n_{J+1}\in(m,n_J]$, and $s_{J+1}\in p^{-\eps^{-1}\N}$, such that
    \begin{equation*}
        \begin{split}
        \|F_{\Theta_J}\|_{L^{q_n}(\Q_p^n)}&\leq p^{Jn\eps^{-1}}(\log_p\delta^{-1})C_{n_{J+1}-m_{J+1},(n_{J+1}-m_{J+1})\eps}^{\frac{n_{J+1}-m_{J+1}}{\eps}}s_{J+1}^{-(m_{J+1}+3)\eps-\Big(1-\frac{\mathfrak{D}_{n_{J+1}-m_{J+1}}}{q_n}\Big)}\\
        &\times\left(\sum_{\Theta'\in\mathcal{P}(\Theta,\mathbf{s}_{J+1}^\circ)}\left\|\sum_{\theta\in\Theta\cap\Lambda_\delta}F_\theta\right\|_{L^{q_n}(\Q_p^n)}^{q_n}\right)^{1/q_n}.
        \end{split}
    \end{equation*}
    
\end{lemma}

\begin{proof}
    We assume $\Theta=B(0,\mathbf{s}_J^\circ)$. For each $\theta\in\Theta\cap\Lambda_\delta$, we write
    \begin{equation*}
        \widehat{g}_\theta=\widehat{F_\theta\circ\mathcal{L}}_{\mathbf{m}_J,\mathbf{s}_J}.
    \end{equation*}
    Thus, $\widehat{g}_\theta$ is supported in the set $\Omega_{\Theta,\mathbf{m}_J,\mathbf{s}_J,\mathbf{n}_J}$. Suppose the second option of \ref{algnotdone} holds. Fix $n_{J+1}=n_J$. Define, for each $m\geq m_{J+1}>m_J$ and $\theta\in\Theta$,
    \begin{equation*}
        \begin{split}
            \Omega_{\theta,\mathbf{m}_J,\mathbf{s}_J,\mathbf{n}_J}^{\mathrm{res},m_{J+1}}=\Big\{\,\,&\Big[\gamma_{\mathbf{m}_J,\mathbf{s}_J}^{(1)}\big(\mathbf{s}_J^\circ\theta\big),\ldots,\gamma_{\mathbf{m}_J,\mathbf{s}_J}^{(n)}\big(\mathbf{s}_J^\circ\theta\big)\Big]\cdot\mathcal{D}_{\mathbf{m}_J,\mathbf{n}_J}^{\mathbf{s}_J}(\bm{\lambda}):\\
            &\forall j\in[1,n]\,\,|\lambda_j|\leq 1,\\
            &|\lambda_{m_{J+1}}|_p=1,\\
            &\forall\iota\in(m_{J+1},n_{J+1}]\,\,|\lambda_\iota|_p<1\hspace{8em}\Big\}.
        \end{split}
    \end{equation*}
    For each $s_{J+1}\in \mathrm{Adap}_{\mathbf{m}_{J+1},\mathbf{n}_{J+1}}^{\mathbf{s}_J}$, we write
    \begin{equation*}
        \begin{split}
            \Omega_{\theta,\mathbf{m}_J,\mathbf{s}_J,\mathbf{n}_J}^{\mathrm{res},m_{J+1},s_{J+1}}=\Big\{\,\,&\Big[\gamma_{\mathbf{m}_J,\mathbf{s}_J}^{(1)}\big(\mathbf{s}_J^\circ\theta\big),\ldots,\gamma_{\mathbf{m}_J,\mathbf{s}_J}^{(n)}\big(\mathbf{s}_J^\circ\theta\big)\Big]\cdot\mathcal{D}_{\mathbf{m}_J,\mathbf{n}_J}^{\mathbf{s}_J}(\bm{\lambda})\in \Omega_{\theta,\mathbf{m}_J,\mathbf{s}_J,\mathbf{n}_J}^{\mathrm{res},m_{J+1}}:\\
            &\exists\iota\in[1,m-m_{J+1}]\text{ s.t. }s_{J+1}^\iota\leq|\lambda_{m_{J+1}+\iota}|_p,\\
            &\forall\iota\in[1,m-m_{J+1}]\,\,|\lambda_{m_{J+1}+\iota}|_p< s_{J+1}^\iota p^{\iota\eps^{-1}}\hspace{7em}\Big\}.
        \end{split}
    \end{equation*}
    Finally, for each $m_{J+1},s_{J+1}$ and each $\mathfrak{R}=(R_1,\ldots,R_{m_{J+1}-1})\in\mathcal{P}(\Z_p,s_{J+1}^\eps)^{m_{J+1}-1}$ and each $B\in\mathcal{P}(\Z_p\setminus p\Z_p,p^{-\eps^{-1}(Jm_{J+1}-\sum_{j=1}^Jm_j)}s_{J+1}^\eps)$, we collapse notation and write
    \begin{equation*}
        \begin{split}
            \widetilde{\Omega}_{\theta,\mathbf{m}_{J+1},\mathbf{s}_{J+1},\mathbf{n}_{J+1}}^{B,\mathfrak{R}}=\Big\{\,\,&\Big[\gamma_{\mathbf{m}_J,\mathbf{s}_J}^{(1)}\big(\mathbf{s}_J^\circ\theta\big),\ldots,\gamma_{\mathbf{m}_J,\mathbf{s}_J}^{(n)}\big(\mathbf{s}_J^\circ\theta\big)\Big]\cdot\mathcal{D}_{\mathbf{m}_J,\mathbf{n}_J}^{\mathbf{s}_J}(\bm{\lambda})\in \Omega_{\theta,\mathbf{m}_J,\mathbf{s}_J,\mathbf{n}_J}^{\mathrm{res},m_{J+1},s_{J+1}}:\\
            &\lambda_j\in R_j\,(j<m_{J+1}),\lambda_{m_{J+1}}\in B\hspace{13em}\Big\}.
        \end{split}
    \end{equation*}
    If we omit the $B,\mathfrak{R}$, then we assume that the $\lambda_j$ range over all of $\Z_p$ (for $j<m_{J+1}$), and the $\lambda_{m_{J+1}}$ range over $\Z_p\setminus p\Z_p$. For each $I\in\mathcal{P}(\Theta,s_{J+1}^\eps)$, we write
    \begin{equation*}
        \widetilde{\Omega}_{I,\mathbf{m}_{J+1},\mathbf{s}_{J+1},\mathbf{n}_{J+1}}^{B,\mathfrak{R}}=\bigcup_{\theta\in I}\widetilde{\Omega}_{\theta,\mathbf{m}_{J+1},\mathbf{s}_{J+1},\mathbf{n}_{J+1}}^{B,\mathfrak{R}}.
    \end{equation*}
    By H\"older as before, one of the following holds: either there exists $m_J<m_{J+1}<m$, $s_{J+1}\in\mathrm{Adap}_{\mathbf{m}_{J+1},\mathbf{n}_{J+1}}^{\mathbf{s}_J}$, $I\in\mathcal{P}(\Theta,\mathbf{s}_J^\circ s_{J+1}^\eps)$, $B\in\mathcal{P}(\Z_p\setminus p\Z_p, p^{-\eps^{-1}(Jm_{J+1}-\sum_{j=1}^Jm_j)}s_{J+1}^\eps)$, and $\mathfrak{R}\in\mathcal{P}(\Z_p, s_{J+1}^\eps)^{m_{J+1}-1}$ such that
\begin{equation}\label{regcase2}
    \left\|\sum_{\theta\in \Lambda_\delta}g_\theta\right\|_{L^{q_n}(\Q_p^n)}\leq mp^{\eps^{-1}(Jm_{J+1}-\sum_{j=1}^Jm_j)}(\log_p\delta^{-1})s_{J+1}^{-(m_{J+1}+1)\eps}\left\|\sum_{\theta\in I\cap\Lambda_\delta}\mathcal{P}_{\widetilde{\Omega}_{I,\mathbf{m}_{J+1},\mathbf{s}_{J+1},\mathbf{n}_{J+1}}^{B,\mathfrak{R}}}g_\theta\right\|_{L^{q_n}(\Q_p^n)},
\end{equation}
or else we set $m_{J+1}=m,s_{J+1}=\left(\delta\prod_{j=1}^Js_j^{-(n_J-1-m_j)}\right)^{\frac{1}{n_{J+1}-m_{J+1}}}$, and there is $I\in\mathcal{P}(\Theta,\mathbf{s}_J^\circ s_{J+1}^\eps)$ such that
\begin{equation}\label{rimcase2} 
    \left\|\sum_{\theta\in \Lambda_\delta}g_\theta\right\|_{L^{q_n}(\Q_p^n)}\leq m(\log_p\delta^{-1})s_{J+1}^{-(m_{J+1}+1)\eps}\left\|\sum_{\theta\in I\cap\Lambda_\delta}\mathcal{P}_{\widetilde{\Omega}_{I,\mathbf{m}_{J+1},\mathbf{s}_{J+1},\mathbf{n}_{J+1}}}g_\theta\right\|_{L^{q_n}(\Q_p^n)}.
\end{equation}
We will focus on the case that \eqref{regcase2} holds, though the same argument will apply in each scenario. We abbreviate $G=\sum_{\theta\in I\cap\Lambda_\delta}\mathcal{P}_{\widetilde{\Omega}_{I,\mathbf{m}_{J+1},\mathbf{s}_{J+1},\mathbf{n}_{J+1}}^{B,\mathfrak{R}}}g_\theta$.
    For simplicity, we take $I=B(0,\mathbf{s}_{J}^\circ s_{J+1}^\eps)$ and $B=B(1, p^{-\eps^{-1}(Jm_{J+1}-\sum_{j=1}^Jm_j)}s_{J+1}^\eps)$. Then we have 
    \begin{equation*}
        \begin{split}
            \Big|\Big(\Big[\gamma_{\mathbf{m}_J,\mathbf{s}_J}^{(1)}\big(\mathbf{s}_J^\circ\theta\big),\ldots,\gamma_{\mathbf{m}_J,\mathbf{s}_J}^{(n)}\big(\mathbf{s}_J^\circ\theta\big)\Big]\cdot\mathcal{D}_{\mathbf{m}_J,\mathbf{n}_J}^{\mathbf{s}_J}(\bm{\lambda})\Big)_\iota&-\frac{(\lambda_{m_{J+1}}\mathbf{s}_J^\circ\theta)^{\iota-m_{J+1}}}{(\iota-m_{J+1})!}p^{-\eps^{-1}(Jm_{J+1}-\sum_{\ell=1}^Jm_\ell)}\Big|_p\\
            &\leq p^{\eps^{-1}(Jm_{J+1}-\sum_{\ell=1}^Jm_\ell)}s_{J+1}^{\eps(\iota-m_{J+1})+\eps},\quad (m_{J+1}<\iota\leq n_{J+1}).
        \end{split}
    \end{equation*}
    The curve
    \begin{equation*}
        \mathbf{s}_J^\circ\theta\mapsto\Big(p^{-\eps^{-1}(Jm_{J+1}-\sum_{\ell=1}^Jm_\ell)}\frac{(\lambda_{m_{J+1}}\mathbf{s}_J^\circ\theta)}{1!},\ldots,p^{-\eps^{-1}(Jm_{J+1}-\sum_{\ell=1}^Jm_\ell)}\frac{(\lambda_{m_{J+1}}\mathbf{s}_J^\circ\theta)^{n_{J+1}-m_{J+1}}}{(n_{J+1}-m_{J+1})!}\Big)
    \end{equation*}
    is rescaled moment curve over $B(0,1)$ of degree $n_{J+1}-m_{J+1}$; thus, the decoupling inequality provides
    \begin{equation*}
        \begin{split}
            \left\|\sum_{\theta\in I\cap\Lambda_\delta}g_\theta\right\|_{L^{q_n}(\Q_p^n)}\leq &\,C_{n_{J+1}-m_{J+1},(n_{J+1}-m_{J+1})\eps}s_{J+1}^{-\frac{\eps}{n_{J+1}-m_{J+1}}(1-\frac{\mathfrak{D}_{n_{J+1}-m_{J+1}}}{q_n}+\eps)}\\
            &\times\left(\sum_{\Theta'\in\mathcal{P}\Big(\Theta,\mathbf{s}_{J}^\circ s_{J+1}^{(1+\frac{1}{n_{J+1}-m_{J+1}})\eps}\Big)}\left\|\sum_{\theta\in\Theta'\cap\Lambda_\delta}g_\theta\right\|_{L^{q_n}(\Q_p^n)}^{q_n}\right)^{1/q_n}.
        \end{split}
    \end{equation*}
    Iterating as in Lemma \ref{decstep} over $1\leq k\leq\frac{n_{J+1}-m_{J+1}}{\eps}$, together with a terminal triangle inequality and Cauchy-Schwarz, we obtain
    \begin{equation*}
        \begin{split}
        \left\|\sum_{\theta\in\Theta\cap\Lambda_\delta}g_\theta\right\|_{L^{q_n}(\Q_p^n)}\leq&\, p^{\eps^{-1}Jn}(\log_p\delta^{-1})C_{n_{J+1}-m_{J+1},(n_{J+1}-m_{J+1})\eps}^{\frac{n_{J+1}-m_{J+1}}{\eps}}s_{J+1}^{-(m_{J+1}+3)\eps-(1-\frac{\mathfrak{D}_{n_{J+1}-m_{J+1}}}{q_n})}\\
        &\times\left(\sum_{\Theta'\in\mathcal{P}(\Theta,\mathbf{s}_{J+1}^\circ)}\left\|\sum_{\theta\in\Theta'\cap\Lambda_\delta}g_\theta\right\|_{L^{q_n}(\Q_p^n)}^{q_n}\right)^{1/q_n}.
        \end{split}
    \end{equation*}
    Undoing the change-of-variable, we achieve the desired result.
    
    If instead the first option of \ref{algnotdone} holds, we set $n_{J+1}=n_J-1$ and, for each choice $m_J\leq m_{J+1}\leq m$ and $s_{J+1}\in p^{-\eps^{-1}\N}$ with
    \begin{equation*}
        \left(\delta\prod_{j=1}^Js_j^{-(n_J-m_j)}\right)^{\frac{1}{n_{J+1}-m_{J+1}}}\leq s_{J+1}<1,
    \end{equation*}
    we set
    \begin{equation*}
        \begin{split}
            \Omega_{\theta,m_{J+1},s_{J+1}}=\Big\{\,\,&\Big[\gamma_{\mathbf{m}_J,\mathbf{s}_J}^{(1)}\big(\mathbf{s}_J^\circ\theta\big),\ldots,\gamma_{\mathbf{m}_J,\mathbf{s}_J}^{(n)}\big(\mathbf{s}_J^\circ\theta\big)\Big]\cdot\mathcal{D}_{\mathbf{m}_J,\mathbf{n}_J}^{\mathbf{s}_J}(\bm{\lambda})\in \Omega_{\theta,m_{J+1}}:\\
            &\exists\iota\in[1,m-m_{J+1}]\text{ s.t. }s_{J+1}^\iota\leq |\lambda_{m_{J+1}+\iota}|_p,\\
            &\forall\iota\in[1,m-m_{J+1}]\,|\lambda_{m_{J+1}+\iota}|_p< s_{J+1}^\iota p^{\iota\eps^{-1}}\hspace{7em}\Big\}.
        \end{split}
    \end{equation*}
    By an identical argument to the previous case (pigeonholing the ranges of $\lambda$, comparing with a cylinder, decoupling and iteration), we obtain that there are some $m_{J+1}$ and $s_{J+1}$ as above such that
    \begin{equation*}
        \begin{split}
        \left\|\sum_{\theta\in\Theta\cap\Lambda_\delta}g_\theta\right\|_{L^{q_n}(\Q_p^n)}\leq&\, p^{Jn\eps^{-1}}(\log_p\delta^{-1})C_{n_{J+1}-m_{J+1},(n_{J+1}-m_{J+1})\eps}^{\frac{n_{J+1}-m_{J+1}}{\eps}}s_{J+1}^{-(m_{J+1}+3)\eps-\big(1-\frac{\mathfrak{D}_{n_{J+1}-m_{J+1}}}{q_n}\big)}\\
        &\times\left(\sum_{\Theta'\in\mathcal{P}(\Theta,\mathbf{s}_{J+1}^\circ)}\left\|\sum_{\theta\in\Theta'\cap\Lambda_\delta}g_\theta\right\|_{L^{q_n}(\Q_p^n)}^{q_n}\right)^{1/q_n},
        \end{split}
    \end{equation*}
    and once again by changing variables we are done.
    \end{proof}

    Observe that, for each iteration of Lemma \ref{alginduction}, we obtain a new decoupling of $F$ into frequency-localized pieces indexed by parameters $\mathbf{s}_J,\mathbf{m}_J,\mathbf{n}_J$ with the condition that, when $J$ increases by $1$, either $m_J$ increases by at least $1$ or $n_J$ decreases by at least $1$, or already $\mathbf{s}_J$ has localized all the way to $\delta$. Since $n_1=n$ and $0\leq m_1\leq m$, we see that after $\mathfrak{J}\leq 2n$ steps the output of Lemma \ref{alginduction} is an estimate of the form
    \begin{equation*}
        \begin{split}
        \|F\|_{L^{q_n}(\Q_p^n)}&\leq p^{\mathfrak{J}^2n\eps^{-1}}\sum_{\mathbf{s}_\mathfrak{J},\mathbf{m}_{\mathfrak{J}},\mathbf{n}_\mathfrak{J}}\prod_{J=1}^\mathfrak{J}\left[C_{n_J-m_J,(n_J-m_J)\eps}^{\frac{n_J-m_J}{\eps}}s_J^{-\eps-(1-\frac{\mathfrak{D}_{n_J-m_J}}{q_n})}\right]\\
        &\times\left(\sum_{(\Theta_0,\ldots,\Theta_{\mathfrak{J}})\in T_{\mathbf{s}_{\mathfrak{J}},\mathbf{m}_{\mathfrak{J}},\mathbf{n}_{\mathfrak{J}}}}\left\|\sum_{\theta\in\Theta_{\mathfrak{J}}\cap\Lambda_\delta}F_\theta\right\|_{L^{q_n}(\Q_p^n)}^{q_n}\right)^{1/q_n}.
        \end{split}
    \end{equation*}
    Note also that there are $\leq (\log_p(\delta^{-1})mn)^{2n}$ choices of tuples $\mathbf{s}_\mathfrak{J},\mathbf{m}_{\mathfrak{J}},\mathbf{n}_\mathfrak{J}$ in the initial sum. Pigeonholing, we obtain that for some choice $\mathbf{s}_\mathfrak{J},\mathbf{m}_{\mathfrak{J}},\mathbf{n}_\mathfrak{J}$, we have the upper bound
    \begin{equation*}
        \begin{split}
        \|f\|_{L^{q_n}(\Q_p^n)}&\leq p^{\mathfrak{J}^2n\eps^{-1}}(\log(\delta^{-1})n)^{4n}\prod_{J=1}^\mathfrak{J}\left[C_{n_J-m_J,(n_J-m_J)\eps}^{\frac{n_J-m_J}{\eps}}s_J^{-\eps-(1-\frac{\mathfrak{D}_{n_J-m_J}}{q_n})}\right]\\
        &\times\left(\sum_{\Theta\in\mathcal{P}(\Z_p,\mathbf{s}_{\mathfrak{J}}^\circ)}\left\|\sum_{\theta\in\Theta\cap\Lambda_\delta}f_\theta\right\|_{L^{q_n}(\Q_p^n)}^{q_n}\right)^{1/q_n}.
        \end{split}
    \end{equation*}
    By the triangle inequality and H\"older, we reach our terminal decoupling
    \begin{equation*}
        \begin{split}
            \|f\|_{L^{q_n}(\Q_p^n)}&\leq p^{\mathfrak{J}^2n\eps^{-1}}(\log(\delta^{-1})n)^{4n}\prod_{J=1}^\mathfrak{J}\left[C_{n_J-m_J,(n_J-m_J)\eps}^{\frac{n_J-m_J}{\eps}}s_J^{-\eps-(1-\frac{\mathfrak{D}_{n_J-m_J}}{q_n})}\right]\\
            &\times(\delta^{-1}\mathbf{s}_{\mathfrak{J}}^\circ)^{1-\frac{1}{q_n}}\left(\sum_{\theta\in\Lambda_\delta}\left\|f_\theta\right\|_{L^{q_n}(\Q_p^n)}^{q_n}\right)^{1/q_n}.
        \end{split}
    \end{equation*}
    It remains to analyze the losses we have incurred. Rearranging factors, we have
    \begin{equation*}
        \left[\prod_{J=1}^\mathfrak{J}s_J^{-\eps-(1-\frac{\mathfrak{D}_{n_J-m_J}}{q_n})}\right](\mathbf{s}_{\mathfrak{J}}^\circ)^{1-\frac{1}{q_n}}=(\mathbf{s}_{\mathfrak{J}}^\circ)^{-\eps}\left[\prod_{J=1}^{\mathfrak{J}}s_J^{1+\ldots+(n_J-m_J)}\right]^{1/q_n}.
    \end{equation*}
    Note that $n_1=n$ and $n_{\mathfrak{J}}=m+1$. Let $j_1,\ldots,j_{n-m}$ be the indices such that $n_{j_k+1}=n_{j_k}-1$. At each such index, we have the identity
    \begin{equation*}
        s_{j_k+1}^{n_{j_k}-m_{j_k}}\prod_{\eta=1}^{j_k}s_\eta^{(n_{j_k}-1-m_\eta)}=\delta.
    \end{equation*}
    Multiplying these identities together, we obtain
    \begin{equation*}
        \prod_{J=1}^{\mathfrak{J}}s_J^{1+\ldots+(n_J-m_J)}\leq\delta^{n-m}.
    \end{equation*}
    Thus, we have demonstrated
    \begin{equation*}
        \|f\|_{L^{q_n}(\Q_p^n)}\leq p^{\mathfrak{J}^2n\eps^{-1}}(\log(\delta^{-1})n)^{4n}\left[\prod_{J=1}^\mathfrak{J}C_{n_J-m_J,(n_J-m_J)\eps}^{\frac{n_J-m_J}{\eps}}\right](\mathbf{s}_{\mathfrak{J}}^\circ)^{-\eps}\delta^{-1+\frac{n-m+1}{q_n}}\left(\sum_{\theta\in\Lambda_\delta}\left\|f_\theta\right\|_{L^{q_n}(\Q_p^n)}^{q_n}\right)^{1/q_n},
    \end{equation*}
    and by the trivial bound $\mathbf{s}_{\mathfrak{J}}^\circ\geq\delta$ and Theorem \ref{momentcurve}, together with Remark \ref{rectsetfamdec}, we have the upper bound
    \begin{equation*}
        \|f\|_{L^{q_n}(\Q_p^n)}\leq\exp\left(10^4\eps^{-4n\log n-1}n^{4 n^2+4n}(\log p)\right)\delta^{-1+\frac{n-m+1}{q_n}-\eps}\left(\sum_{\theta\in\Lambda_\delta}\left\|f_\theta\right\|_{L^{q_n}(\Q_p^n)}^{q_n}\right)^{1/q_n}.
    \end{equation*}

    It remains only to remove the special assumptions on $\delta$ and $\eps$. Fix $\eps=\frac{1}{\ell\kappa}$ for some $\ell\in\N$, and suppose that $\delta\in p^{-\N}$ is such that
    \begin{equation*}
        p^{-\kappa^2(K+1)}<\delta<p^{-\kappa^2K},\quad K\in\N.
    \end{equation*}
    Then, by what we have proven for $\delta'=p^{-\kappa^2K}$, if $\Lambda_{\delta'}\subseteq\Lambda_\delta$ is any $\delta'$-separated subset,
    \begin{equation*}
        \left\|\sum_{\theta\in\Lambda_{\delta'}}f_\theta\right\|_{L^{q_n}(\Q_p^n)}\leq\exp\left(10^4\eps^{-4n\log n-1}n^{4 n^2+4n}(\log p)\right)(\delta')^{-1+\frac{n-m+1}{q_n}-\eps}\left(\sum_{\theta\in\Lambda_{\delta'}}\left\|f_\theta\right\|_{L^{q_n}(\Q_p^n)}^{q_n}\right)^{1/q_n}.
    \end{equation*}
    Controlling $f$ by a sum over $\delta'/\delta$-many subsets $\Lambda_{\delta'}$, we conclude that
    \begin{equation*}
        \|f\|_{L^{q_n}(\Q_p^n)}\leq \exp\left(10^4\eps^{-4n\log n-1}n^{6 n^2+4n}(\log p)\right)\delta^{-1+\frac{n-m+1}{q_n}-\eps}\left(\sum_{\theta\in\Lambda_{\delta}}\left\|f_\theta\right\|_{L^{q_n}(\Q_p^n)}^{q_n}\right)^{1/q_n}.
    \end{equation*}
    In the case $p^{-\kappa^2}<\delta\leq p^{-1}$, a trivial inequality suffices to get the same result.

    Next, we remove the special assumption on $\eps$. If $\ell\in\N$ is such that
    \begin{equation*}
        \frac{(n!)^{-2n}}{\ell+1}<\eps<\frac{(n!)^{-2n}}{\ell},
    \end{equation*}
    then for any $\delta\in p^{-\N}$, using $(\ell+1)(n!)^{2n}\leq 2\eps^{-1}$, we have shown that
    \begin{equation*}
        \|f\|_{L^{q_n}(\Q_p^n)}\leq \exp\left(10^4\eps^{-4n\log n-1}n^{6 n^2+6n}(\log p)\right)\delta^{-1+\frac{n-m+1}{q_n}-\eps}\left(\sum_{\theta\in\Lambda_{\delta}}\left\|f_\theta\right\|_{L^{q_n}(\Q_p^n)}^{q_n}\right)^{1/q_n},
    \end{equation*}
    and we are done (noting that a trivial inequality suffices for the same result when $\eps>(n!)^{-2n}$).
\end{proof}

\section{Appendix A: Decoupling lemmas over \texorpdfstring{$\Q_p$}{Qp}}

In this section, we record various elementary lemmas regarding Fourier decoupling over $\Q_p$. Each of these is a cousin of a standard result over $\R$, and some of ours will be even stronger. Several of these lemmas have already been demonstrated in \cite{li2022introduction}.

Throughout this section, $\Omega$ will denote a subset of $\Q_p^d$ and $\Theta$ will denote a family of subsets of $\Omega$, such that $\Omega=\bigcup_{\theta\in\Theta}\theta$. We will also emphasize that \emph{all functions $f_\theta$ are locally constant and of compact support}, and indicate the corresponding class via $\mathcal{S}(\Q_p^d)$. We emphasize that ``locally constant'' means that there is some scale $\lambda\in p^{\Z}$ such that $\|x-y\|\leq\lambda$ implies $f_\theta(x)=f_\theta(y)$. The class $\mathcal{S}(\Q_p^d)$ is the appropriate replacement for Schwartz functions in the $p$-adic setting; they are precisely the ``Schwartz-Bruhat functions'' over $\Q_p^d$.

For exponents $2\leq q\leq r\leq\infty$, $q<\infty$, define $\operatorname{Dec}_{\ell^qL^r}(\Theta)$ to be the infimal $C>0$ such that, for any family $\{f_\theta:\theta\in\Theta\}$ such that $\hat{f}_\theta$ is supported in $\theta$, for each $\theta$,
\begin{equation*}
    \left\|\sum_{\theta\in\Theta}f_\theta\right\|_{L^r(\Q_p^d)}\leq C\left(\sum_{\theta\in\Theta}\left\|f_\theta\right\|_{L^r(\Q_p^d)}^q\right)^{1/q}.
\end{equation*}

Observe that we have not insisted that the sets $\theta\in\Theta$ are pairwise disjoint; in applications, this will often be true, but for many technical results it is convenient to allow $O(1)$ overlap between the caps.

The following is demonstrated in \cite{li2022introduction}, Prop. 4.4, in the case $q=2$ and for specific choices of $\theta$. The proof of this version is identical.

\begin{lemma}[Interpolation of decoupling constants]\label{interpolate} If $\frac{1}{r}=\frac{\alpha}{r_0}+\frac{1-\alpha}{r_1}$, $\alpha\in[0,1]$, and $1\leq q\leq\min(r_0,r_1)$, then for any partition $\Omega=\bigcup_{\theta\in\Theta}\theta$ with every $\theta$ a separate affine image of $\Z_p^d$, we have
\begin{equation*}
    \operatorname{Dec}_{\ell^qL^r}(\Theta)\leq\operatorname{Dec}_{\ell^qL^{r_0}}(\Theta)^\alpha\operatorname{Dec}_{\ell^qL^{r_1}}(\Theta)^{1-\alpha}.
\end{equation*}

\end{lemma}

As a consequence, we obtain the following. Observe that we do \textit{not} need to assume that the $\theta$ are all congruent, in contrast to the Euclidean case.

\begin{lemma}[Flat decoupling] For any $\Omega\subseteq\Q_p^d$ and any partition $\Theta$ of $\Omega$ composed of affine images $\theta\in\Theta$ of $\Z_p^d$ and any $q,r\geq 2$,
\begin{equation*}
    \operatorname{Dec}_{\ell^qL^p}(\Theta)\leq(\#\Theta)^{1-\frac{1}{r}-\frac{1}{q}}.
\end{equation*}
    
\end{lemma}
\begin{proof}
    Fix any family $\{f_\theta:\theta\in\Theta\}$. Then, by Plancherel, these elements are pairwise orthogonal in $L^2(\Q_p^d)$; thus
    \begin{equation}\label{orth}
        \left\|\sum_{\theta\in\Theta}f_\theta\right\|_{L^2(\Q_p^d)}\leq\left(\sum_{\theta\in\Theta}\|f_\theta\|_{L^2(\Q_p^d)}^2\right)^{1/2},
    \end{equation}
    so $\operatorname{Dec}_{\ell^2L^2}(\Theta)\leq 1$. By the triangle inequality and Cauchy-Schwarz,
    \begin{equation}\label{triv}
        \left\|\sum_{\theta\in\Theta}f_\theta\right\|_{L^\infty(\Q_p^d)}\leq(\#\Theta)^{1/2}\left(\sum_{\theta\in\Theta}\|f_\theta\|_{L^\infty(\Q_p^d)}^2\right)^{1/2}.
    \end{equation}
    By Lemma \ref{interpolate}, the statements \ref{orth} and \ref{triv} give
    \begin{equation*}
        \operatorname{Dec}_{\ell^2L^r}(\Theta)\leq\operatorname{Dec}_{\ell^2L^\infty}(\Theta)^{\frac{r-2}{r}}\leq(\#\Theta)^{\frac{1}{2}-\frac{1}{r}}.
    \end{equation*}
    By H\"older,
    \begin{equation*}
        \left(\sum_{\theta\in\Theta}\|f_\theta\|_{L^r(\Q_p^d)}^2\right)^{1/2}\leq(\#\Theta)^{\frac{1}{2}-\frac{1}{q}}\left(\sum_{\theta\in\Theta}\|f_\theta\|_{L^r(\Q_p^d)}^q\right)^{1/q}.
    \end{equation*}
    Thus
    \begin{equation*}
        \left\|\sum_{\theta\in\Theta}f_\theta\right\|_{L^r(\Q_p^d)}\leq(\#\Theta)^{\frac{1}{2}-\frac{1}{r}}\left(\sum_{\theta\in\Theta}\|f_\theta\|_{L^r(\Q_p^d)}^2\right)^{1/2}\leq(\#\Theta)^{1-\frac{1}{r}-\frac{1}{q}}\left(\sum_{\theta\in\Theta}\|f_\theta\|_{L^r(\Q_p^d)}^q\right)^{1/q},
    \end{equation*}
    and so $\operatorname{Dec}_{\ell^qL^r}(\Theta)\leq(\#\Theta)^{1-\frac{1}{r}-\frac{1}{q}}$, as claimed.
    
\end{proof}

\begin{lemma}[Affine invariance of decoupling constants] Suppose $A$ is an invertible affine map $\Q_p^d\to\Q_p^d$. Then $\operatorname{Dec}_{\ell^qL^r}(A\Theta)=\operatorname{Dec}_{\ell^qL^r}(\Theta)$, where $A\Theta=\{A\theta:\theta\in\Theta\}$.

\end{lemma}

\begin{proof}
    We first take $A$ to be linear for simplicity. Suppose $\{f_\theta:\theta\in\Theta\}$ are such that $\hat{f}_\theta$ is supported in $\theta$. Define $g_\theta=(\hat{f}_\theta\circ A^{-1})^\vee$. Then $\hat{g}_\theta$ is supported in $A\theta$, so
    \begin{equation}\label{aff}
        \left\|\sum_{\theta\in\Theta}g_\theta\right\|_{L^r(\Q_p^d)}\leq\operatorname{Dec}_{\ell^qL^r}(A\Theta)\left(\sum_{\theta\in\Theta}\|g_\theta\|_{L^r(\Q_p^d)}^q\right)^{1/q}.
    \end{equation}
    Observe that the following change-of-variables holds:
    \begin{equation*}
        g_\theta(x)=\int_{\Q_p^d}\hat{f}_\theta(A(\xi))\chi(x\cdot\xi)d\xi=\frac{1}{\mu(A[\Z_p^d])}\int_{\Q_p^d}\hat{f}_\theta(\omega)\chi(A^{-\top}(x)\cdot \omega)d\omega=\frac{1}{\mu(A[\Z_p^d])}f(A^{-\top}(x)),
    \end{equation*}
    so that
    \begin{equation*}
        \|g_\theta\|_{L^r(\Q_p^d)}=\mu(A[\Z_p^d])^{-1+\frac{1}{r}}\|f_\theta\|_{L^r(\Q_p^d)}.
    \end{equation*}
    In particular, \ref{aff} rearranges to
    \begin{equation*}
        \left\|\sum_{\theta\in\Theta}f_\theta\right\|_{L^r(\Q_p^d)}\leq\operatorname{Dec}_{\ell^qL^r}(A\Theta)\left(\sum_{\theta\in\Theta}\|f_\theta\|_{L^r(\Q_p^d)}^q\right)^{1/q}.
    \end{equation*}
    Since the $\{f_\theta:\theta\in\Theta\}$ were arbitrary, we conclude that
    \begin{equation*}
        \operatorname{Dec}_{\ell^qL^r}(\Theta)\leq\operatorname{Dec}_{\ell^qL^r}(A\Theta).
    \end{equation*}
    Since this holds for all invertible linear $A$, we conclude that $\operatorname{Dec}_{\ell^qL^r}(\Theta)=\operatorname{Dec}_{\ell^qL^r}(A\Theta)$ for all invertible linear $A$.

    Finally, we note that decoupling constants are trivially invariant under translation, as Fourier translation is equivalent to spatial modulation, which does not affect absolute values. Thus the claim holds.
\end{proof}

\begin{lemma}[Local decoupling]\label{localdec}
    Suppose every $\theta\in\Theta$ is of the form $A_\theta[\Z_p^d]+v_\theta$ for linear isomorphisms $A_\theta:\Q_p^d\to\Q_p^d$. Set
    \begin{equation*}
        \eta:=\max_{\theta\in\Theta}\|A_\theta^{-1}\|,
    \end{equation*}
    where $\|\,\cdot\,\|$ is the usual $\ell^\infty\to\ell^\infty$ operator norm. Write $\operatorname{Dec}_{\ell^qL^r}^{\mathrm{loc}}(\Theta)$ for the infimal $C>0$ such that, for any family $\{f_\theta:\theta\in\Theta\}$ such that $\hat{f}_\theta$ is supported in $\theta$, and for any $x\in\Q_p^d$, we have
    \begin{equation*}
        \left\|\sum_{\theta\in\Theta}f_\theta\right\|_{L^r(B(x,\eta))}\leq C\left(\sum_{\theta\in\Theta}\left\|f_\theta\right\|_{L^r(B(x,\eta))}^q\right)^{1/q}.
    \end{equation*}
    Then we have
    \begin{equation*}
        \operatorname{Dec}_{\ell^qL^r}(\Theta)=\operatorname{Dec}_{\ell^qL^r}^{\mathrm{loc}}(\Theta).
    \end{equation*}

\end{lemma}

\begin{proof}
    Let $\{f_\theta\}_{\theta\in\Theta}$ be any family as stated. Let $x\in\Q_p^d$ be arbitrary. Write
    \begin{equation*}
        g(y)=\mathds{1}_{B(x,\eta)}(y),\quad\hat{g}(\xi)=\chi(-x\cdot \xi)\mathds{1}_{B(0,\eta^{-1})}(\xi).
    \end{equation*}
    Then we have
    \begin{equation*}
        \left\|\sum_{\theta\in\Theta}f_\theta\right\|_{L^r(B(x,\eta))}=\left\|\sum_{\theta\in\Theta}gf_\theta\right\|_{L^r(\Q_p^d)},
    \end{equation*}
    and
    \begin{equation*}
        \widehat{gf_\theta}(\xi)=[\chi(-x\cdot\,)\mathds{1}_{B(0,\eta^{-1})}]*\hat{f}_\theta(\xi),
    \end{equation*}
    which is still supported in $\theta$. Thus
    \begin{equation*}
        \left\|\sum_{\theta\in\Theta}gf_\theta\right\|_{L^r(\Q_p^d)}\leq\operatorname{Dec}_{\ell^qL^r}(\Theta)\left(\sum_{\theta\in\Theta}\|gf_\theta\|_{L^r(\Q_p^d)}^q\right)^{1/q}.
    \end{equation*}
    Thus we have $\operatorname{Dec}_{\ell^qL^r}^{\mathrm{loc}}(\Theta)\leq\operatorname{Dec}_{\ell^qL^r}(\Theta)$.

    We consider the reverse inequality. We redefine $g=\mathds{1}_{B(0,\eta)}$. Let $X$ be the set of standard representatives of $\Q_p^d/B(0,\eta)$; i.e. we represent $x+B(0,\eta)$ by $y$ when $\eta^{-1}y$ has zero integer part. Then we have:
    \begin{equation*}
        \begin{split}
        \int_{\Q_p^d}\left|\sum_{\theta\in\Theta}f_\theta\right|^r&=\sum_{x\in X}\int_{B(x,\eta)}\left|g(y-x)\sum_{\theta\in\Theta}f_\theta(y)\right|^rd\mu(y)\\
        &\leq\operatorname{Dec}_{\ell^qL^r}^{\mathrm{loc}}(\Theta)^r\sum_{x\in X}\left(\sum_{\theta\in\Theta}\left[\int_{B(x,\eta)}\left|f_\theta(y)\right|^rd\mu(y)\right]^{q/r}\right)^{r/q}.
        \end{split}
    \end{equation*}
    By Minkowski, we have
    \begin{equation*}
        \sum_{x\in X}\left(\sum_{\theta\in\Theta}\left[\int_{B(x,\eta)}\left|f_\theta(y)\right|^rd\mu(y)\right]^{q/r}\right)^{r/q}\leq\left(\sum_{\theta\in\Theta}\left(\sum_{x\in X}\int_{B(x,\eta)}|f_\theta(y)|^rd\mu(y)\right)^{q/r}\right)^r.
    \end{equation*}
    Taking $r$th roots, we obtain the inequality $\operatorname{Dec}_{\ell^qL^r}(\Theta)\leq\operatorname{Dec}_{\ell^qL^r}^{\mathrm{loc}}(\Theta)$.
    
\end{proof}

\begin{lemma}[Decoupling tensorizes] Let $\Omega_1\subseteq\Q_p^d,\Omega_2\subseteq\Q_p^e$ be any sets and let $\Theta_1,\Theta_2$ be any partitions of $\Omega_1,\Omega_2$, respectively. Write $\Theta$ for the partition of $\Omega_1\times\Omega_2$ by sets of the form $\theta\times\tau$ ($\theta\in\Theta_1,\tau\in\Theta_2$). Suppose $q\leq r$. Then
\begin{equation*}
    \operatorname{Dec}_{\ell^qL^r}(\Theta)=\operatorname{Dec}_{\ell^qL^r}(\Theta_1)\operatorname{Dec}_{\ell^qL^r}(\Theta_2)
\end{equation*}
\begin{proof}

    First consider any family $\{f_\theta^1:\theta\in\Theta_1\}$ and $\{f_\tau^2:\tau\in\Theta_2\}$ with $\hat{f}_\theta^1$ supported in $\theta$ and $\hat{f}_\tau^2$ supported in $\tau$. Define $g_{(\theta,\tau)}:\Q_p^{d+e}\to\C$ by
    \begin{equation*}
        g_{(\theta,\tau)}(x,y)=f_\theta^1(x)f_{\tau}^2(y)
    \end{equation*}
    Then $\hat{g}_{(\theta,\tau)}(\xi,\omega)=\hat{f}_\theta^1(\xi)\hat{f}_{\tau}^2(\omega)$ is supported in $\theta\times\tau$. In particular,
    \begin{equation*}
        \left\|\sum_{\theta\times\tau\in\Theta}g_{(\theta,\tau)}\right\|_{L^r(\Q_p^{d+e})}\leq\operatorname{Dec}_{\ell^qL^r}(\Theta)\left(\sum_{\theta\times\tau\in\Theta}\|g_{(\theta,\tau)}\|_{L^r(\Q_p^{d+e})}^q\right)^{1/q}
    \end{equation*}
    Processing both sides of this, observe that
    \begin{equation*}
        \left\|\sum_{\theta\times\tau\in\Theta}g_{(\theta,\tau)}\right\|_{L^r(\Q_p^{d+e})}=\left\|\sum_{\theta\in\Theta_1} f_\theta^1\right\|_{L^r(\Q_p^d)}\left\|\sum_{\tau\in\Theta_2}f_\tau^2\right\|_{L^r(\Q_p^e)}
    \end{equation*}
    and
    \begin{equation*}
        \left(\sum_{\theta\times\tau\in\Theta}\|g_{(\theta,\tau)}\|_{L^r(\Q_p^{d+e})}^q\right)^{1/q}=\left(\sum_{\theta\in\Theta_1}\|f_\theta^1\|_{L^r(\Q_p^{d})}^q\right)^{1/q}\left(\sum_{\tau\in\Theta_2}\|f_\tau^2\|_{L^r(\Q_p^e)}^q\right)^{1/q}
    \end{equation*}
    Picking $\{f_\theta^1\}_{\theta\in\Theta_1}$ and $\{f_\tau^2\}_{\tau\in\Theta_2}$, not all zero, such that
    \begin{equation*}
        \begin{split}
        &\left\|\sum_{\theta\in\Theta_1} f_\theta^1\right\|_{L^r(\Q_p^d)}\geq(1-\eps)\operatorname{Dec}_{\ell^qL^r}(\Theta_1)\left(\sum_{\theta\in\Theta_1}\|f_\theta^1\|_{L^r(\Q_p^d)}^q\right)^{1/q},\\
        &\left\|\sum_{\tau\in\Theta_2} f_\tau^2\right\|_{L^r(\Q_p^d)}\geq(1-\eps)\operatorname{Dec}_{\ell^qL^r}(\Theta_2)\left(\sum_{\tau\in\Theta_2}\|f_\tau^1\|_{L^r(\Q_p^d)}^q\right)^{1/q}
        \end{split}
    \end{equation*}
    we see that
    \begin{equation*}
        (1-\eps)^2\operatorname{Dec}_{\ell^qL^r}(\Theta_1)\operatorname{Dec}_{\ell^qL^r}(\Theta_2)\leq\operatorname{Dec}_{\ell^qL^r}(\Theta).
    \end{equation*}
    Taking $\eps\to 0$ we obtain the inequality
    \begin{equation*}
        \operatorname{Dec}_{\ell^qL^r}(\Theta_1)\operatorname{Dec}_{\ell^qL^r}(\Theta_2)\leq\operatorname{Dec}_{\ell^qL^r}(\Theta)
    \end{equation*}
    
    It remains to establish the reverse inequality.
    
    Let $\{g_{(\theta,\tau)}\}_{(\theta,\tau)\in\Theta}$ be a family with $\hat{g}_{(\theta,\tau)}$ supported in $\theta\times\tau$. Then, for each fixed $y\in\Q_p^e$, observe that $x\mapsto g_{(\theta,\tau)}(x,y)$ has Fourier support contained in the set $\theta$. Thus
    \begin{equation*}
        \int_{\Q_p^e}\int_{\Q_p^d}\left|\sum_{\theta\in\Theta_1}\sum_{\tau\in\Theta_2}g_{(\theta,\tau)}(x,y)\right|^rdxdy\leq\operatorname{Dec}_{\ell^qL^r}(\Theta_1)^r\int_{\Q_p^e}\left[\sum_{\theta\in\Theta_1}\left(\int_{\Q_p^d}\left|\sum_{\tau\in\Theta_2}g_{(\theta,\tau)}(x,y)\right|^rdx\right)^{q/r}\right]^{r/q}dy
    \end{equation*}
    By Minkowski, using $q\leq r$,
    \begin{equation*}
        \begin{split}
        \int_{\Q_p^e}\Big[\sum_{\theta\in\Theta_1}\Big(\int_{\Q_p^d}\Big|&\sum_{\tau\in\Theta_2}g_{(\theta,\tau)}(x,y)\Big|^rdx\Big)^{q/r}\Big]^{r/q}dy\\
        &\leq\Big[\sum_{\theta\in\Theta_1}\Big(\int_{\Q_p^e}\int_{\Q_p^d}\Big|\sum_{\tau\in\Theta_2}g_{(\theta,\tau)}(x,y)\Big|^rdxdy\Big)^{q/r}\Big]^{r/q},
        \end{split}
    \end{equation*}
    and we may apply Fubini and decouple further to obtain for each $\theta$
    \begin{equation*}
        \int_{\Q_p^e}\int_{\Q_p^d}\left|\sum_{\tau\in\Theta_2}g_{(\theta,\tau)}(x,y)\right|^rdxdy\leq\operatorname{Dec}_{\ell^qL^r}(\Theta_2)^r\int_{\Q_p^d}\Big[\sum_{\tau\in\Theta_2}\Big(\int_{\Q_p^e}|g_{(\theta,\tau)}(x,y)|^rdy\Big)^{q/r}\Big]^{r/q}dx.
    \end{equation*}
    Collecting all the preceding,
    \begin{equation*}
        \begin{split}
        \Big\|\sum_{\theta\in\Theta_1}\sum_{\tau\in\Theta_2}g_{(\theta,\tau)}\Big\|_{L^r(\Q_p^{d+e})}&\leq\operatorname{Dec}_{\ell^qL^r}(\Theta_1)\operatorname{Dec}_{\ell^qL^r}(\Theta_2)\\
        &\times\Big\{\sum_{\theta\in\Theta_1}\Big[\int_{\Q_p^d}\Big(\sum_{\tau\in\Theta_2}\big(\int_{\Q_p^e}|g_{(\theta,\tau)}(x,y)|^rdy\big)^{q/r}\Big)^{r/q}dx\Big]^{q/r}\Big\}^{1/q}
        \end{split}
    \end{equation*}
    Applying Minkowski again,
    \begin{equation*}
        \begin{split}
        \Big\{\sum_{\theta\in\Theta_1}&\Big[\int_{\Q_p^d}\Big(\sum_{\tau\in\Theta_2}\big(\int_{\Q_p^e}|g_{(\theta,\tau)}(x,y)|^rdy\big)^{q/r}\Big)^{r/q}dx\Big]^{q/r}\Big\}^{1/q}\\
        &\leq\Big[\sum_{\theta\in\Theta_1}\sum_{\tau\in\Theta_2}\Big(\int_{\Q_p^d}\int_{\Q_p^e}|g_{(\theta,\tau)}(x,y)|^rdydx\Big)^{q/r}\Big]^{1/q},
        \end{split}
    \end{equation*}
    from which we obtain the estimate
    \begin{equation*}
        \Big\|\sum_{\theta\times\tau\in\Theta}g_{(\theta,\tau)}\Big\|_{L^r(\Q_p^{d+e})}\leq\operatorname{Dec}_{\ell^qL^r}(\Theta_1)\operatorname{Dec}_{\ell^qL^r}(\Theta_2)\Big(\sum_{\theta\times\tau\in\Theta}\|g_{(\theta,\tau)}\|_{L^r(\Q_p^{d+e})}^q\Big)^{1/q}.
    \end{equation*}
    Since this holds for all arrangements $\{g_{(\theta,\tau)}\}_{\theta,\tau}$ as in the definition of the decoupling constant for $\Theta$, we conclude that
    \begin{equation*}
        \operatorname{Dec}_{\ell^qL^r}(\Theta)\leq\operatorname{Dec}_{\ell^qL^r}(\Theta_1)\operatorname{Dec}_{\ell^qL^r}(\Theta_2),
    \end{equation*}
    so we have equality, as claimed.
    
\end{proof}

\end{lemma}

A special case of the previous lemma is the following:

\begin{lemma}[Cylindrical decoupling] Let $\Omega\subseteq\Q_p^d$ be any set and $\Theta$ be a partition of $\Omega$. Write $\widetilde{\Theta}$ for the partition of $\Omega\times\Q_p^e$ by $\tilde{\Theta}=\{\theta\times\Q_p^e:\theta\in\Theta\}$. Suppose $q\leq r$. Then
\begin{equation*}
    \operatorname{Dec}_{\ell^qL^r}(\widetilde{\Theta})=\operatorname{Dec}_{\ell^qL^r}(\Theta).
\end{equation*}
\end{lemma}

The following is at least as ubiquitous in decoupling methods as affine invariance.
\begin{lemma}[Multiplicativity of decoupling constants]\label{multiplicativity} Let $2\leq q,r\leq\infty$ and $q<\infty$. Let $\Theta$ be a finite set-family in $\Q_p^d$. Suppose that, for each $\theta\in\Theta$, there is a further set-family $\Theta_\theta$ with the property that $\bigcup_{\psi\in\Theta_\theta}\psi=\theta$. Then it holds that
\begin{equation*}
    \mathrm{Dec}_{\ell^qL^r}\big(\bigsqcup_{\theta\in\Theta}\Theta_\theta\big)\leq\mathrm{Dec}_{\ell^qL^r}(\Theta)\times\max_{\theta\in\Theta}\mathrm{Dec}_{\ell^qL^r}(\Theta_\theta).
\end{equation*}
    
\end{lemma}
\begin{proof}
    Immediate from the structure of decoupling inequalities.
\end{proof}

The following lemma is sometimes helpful.
\begin{lemma}[$\ell^2L^r$ recoupling]\label{recoupling} Let $2\leq r\leq\infty$, and assume that $\theta\in\Theta$ are pairwise disjoint affine images of $\Z_p^n$. Then, for any family $\{f_\theta:\theta\in\Theta\}$ such that $\hat{f}_\theta$ is supported in $\theta$, we have
\begin{equation*}
    \left(\sum_{\theta\in\Theta}\|f_\theta\|_{L^r(\Q_p^d)}^2\right)^{1/2}\leq (\#\Theta)^{\frac{1}{2}-\frac{1}{r}}\left\|\sum_{\theta\in\Theta}f_\theta\right\|_{L^r(\Q_p^d)}.
\end{equation*}
\end{lemma}
\begin{proof}
    The special cases $r=2,r=\infty$ are trivial. For the rest, we interpolate.
\end{proof}

We also recall the main result of \cite{li2022introduction}:
\begin{theorem}
    Fix any $\delta\in p^{-\N}$. Consider the region $\Omega$ defined by
    \begin{equation*}
        \Omega=\{(x,y)\in\Q_p^2:|x|_p\leq 1,|y-x^2|_p\leq\delta^2\}.
    \end{equation*}
    We let $\mathcal{T}=\mathcal{P}(\Z_p,\delta)$ to be the partition of $\Z_p$ into closed balls of radius $\delta$. For each $\tau\in\mathcal{T}$, define
    \begin{equation*}
        \theta_{\tau}=\{(x,y)\in\Q_p^2:x\in \tau,|y-x^2|_p\leq\delta^2\}.
    \end{equation*}
    Clearly $\{\theta_\tau\}_{\tau\in\mathcal{T}}$ form a decomposition $\Theta$ of $\Omega$. Then we have
    \begin{equation*}
        \operatorname{Dec}_{\ell^2L^r}(\Theta)\lesssim_{\eps,p,r}\delta^{-\eps}(1+\delta^{-(\frac{1}{2}-\frac{3}{r})}).
    \end{equation*}
\end{theorem}

\section{Appendix B: Decoupling for the \texorpdfstring{$p$}{p}-adic moment curve}

	We sketch a proof of $\ell^2L^{n(n+1)}$ decoupling for the moment curve $t\mapsto\gamma(t)$ in $\Q_p^n$ by modifying an existing argument for the same result in $\R^n$. This fact is of interest in its own right; however, the proof is nearly identical to the proof over $\R$ for most approaches, so we have suppressed it to this appendix. One slight novelty is the tracking of constants throughout, for the purpose of achieving an explicit effective bound for our main application.

    The result to be shown is the following:
    \begin{theorem}[$\ell^2L^{n(n+1)}$ decoupling for the moment curve in $\Q_p^n$]\label{momentcurve}
        Let $n\in\N$. For every $\eps>0$ there is a constant $C_{n,\eps}\geq 1$ such that for all $\delta\in p^{-\N}$, one has the estimate
        \begin{equation*}
            \operatorname{Dec}_{\ell^2L^{n(n+1)}}(\{\mathcal{U}_I\}_{I\in\mathcal{P}(\Z_p,\delta)})\leq C_{n,\eps}\delta^{-\eps},
        \end{equation*}
        where $\mathcal{P}(\Z_p,\delta)$ is a partition of $\Z_p$ into balls of radius $\delta$, and $\mathcal{U}_I$ is the standard anisotropic neighborhood of the moment curve over $I$; see the start of the next section. Moreover, the constant $C_{n,\eps}$ may be taken to be
        \begin{equation*}
            C_{n,\eps}=\exp\left(10^4(\log p)\eps^{-4n\log n}n^{10n^2}\right).
        \end{equation*}
    \end{theorem}
    \begin{remark}
        Optimizing over $\eps$, one can show that the decoupling constant is bounded by something of the form $\exp(C_n(\log\delta^{-1})^{1-c_n})$, for suitable explicit $C_n\geq 1, c_n\in(0,1)$. See Theorem \ref{vinogradov} for details, in the application to solution counting.
    \end{remark}
 
    Most of the tools used in the standard approaches to decoupling are identical between $\R^n$ and $\Q_p^n$, with some caveats: for one, over $\Q_p$ many heuristic uncertainty statements from the Euclidean setting become literally true, which allows one to dispense of various technical weights and convolutions; for another, some of the induction-on-dimension steps require some special geometric observations (via projections), which require modification in the $p$-adic setting.

    To be more precise about the latter: it is a classical fact that bilinear forms generally, and the dot product in particular, possess isotropy on $\Q_p^n$ for $n$ sufficiently large. We recall two results in particular:
    \begin{theorem}[Chapter 4, Lemma 2.7 of \cite{cassels1978}] Let $n\geq 5$ and $p$ arbitrary. Then every quadratic form over $\Q_p^n$ has isotropy.
    \end{theorem}
    \begin{theorem}\label{isotropy} Let $p$ be odd and $n\geq 3$. Then the dot product $(x_1,\ldots,x_n)\cdot(y_1,\ldots,y_n)=x_1y_1+\ldots+x_ny_n$ has isotropy.
    \end{theorem}
    We briefly recall a proof of Theorem \ref{isotropy}. We first instead study $q(x,y,z)=x^2+y^2+z^2$ over $\F_p$. The set of values $\{x^2:x\in\F_p\}$ and $\{-y^2-1:y\in\F_p\}$ each have cardinality $\frac{p+1}{2}$, while $\#\F_p=p$, so the two sets must intersect at a value where $x^2+y^2+1^2=0$, which establishes isotropy over $\F_p$. Consequently, the representatives of $x,y$ in $\{0,\ldots,p-1\}$ in $\Z$ solve $q(x,y,1)\equiv 0$ mod $p$. On the other hand, formally differentiating $q$ with respect to $x$ and $y$,
    \begin{equation*}
        \partial_xq(x,y,1)=2x,\quad\partial_yq(x,y,1)=2y.
    \end{equation*}
    Since $x^2+y^2+1\equiv 0$ mod $p$, we may assume without loss of generality that $x\not\equiv 0$ mod $p$. Since $p$ is odd,
    \begin{equation*}
        \partial_xq(x,y,1)\not\equiv 0\quad\text{mod $p$}.
    \end{equation*}
    By Hensel's lemma (see Chapter 3, Lemma 4.1 of \cite{cassels1978}), there is a root of $t\mapsto q(t,y,1)$ in $\Z_p$, which establishes Theorem \ref{isotropy}.

    As a consequence, some of the proofs of induction on dimension-type estimates fail. As it turns out, the differences are entirely superficial; when the arguments are converted into linear-algebraic manipulations, the proofs hold as usual.
	
	As many of the arguments in the proof of decoupling require little modification, we will simply review the short proof of moment-curve decoupling in \cite{guo2021short} and supply the needed modifications. In particular, our argument will not be completely self-contained, and will instead point to the latter paper for the proofs of certain technical steps for which no modification is needed.
	
	We insist at the outset that we will only consider the case $p>n$, to avoid certain technical issues. It happens that the same result holds for general $p$ (indeed, for arbitrary non-Archimedean local fields of characteristic $0$), though the argument requires attending to certain additional technicalities that we are able to avoid.
 
    We will write throughout $|\cdot|_p$ for the $p$-adic norm on $\Q_p$, and $|\cdot|$ for the usual Euclidean norm on $\R$ and $\C$. We will also equip $\Q_p^n$ with the usual choice of norm, $\|x\|=\max_{1\leq i\leq n}|x_i|_p$. This notation will overlap with the Lebesgue norm $\|\cdot\|_{L^{q}(\Q_p^n)}$. In each case, it will be clear from context which is intended.
	
	\subsection{Bilinear-to-linear reduction}

	For $\delta\in p^{-\N}$ and a ball $\mathcal{U}$, write $\mathcal{P}(\mathcal{U},\delta)$ for the partition of $\mathcal{U}$ into (closed) balls of radius $\delta$; more generally, if $\delta\in(0,1)$ is not necessarily a power of $p$, then we understand $\mathcal{P}(\mathcal{U},\delta)$ to be a partition into balls of radius $\rho$, where $\rho$ is the greatest number in $p^{-\N}$ below $\delta$. For a convex $\mathbf{C}^{n}$ curve $\zeta$ with bounded derivatives, as defined in Appendix C, we define the systems of boxes $\mathcal{U}_{I,t}^\zeta$, for $I\subseteq\Z_p$ a metric ball of radius $\rho$ and $t\in I$, as
	\begin{equation*}
		\mathcal{U}_{I,t}^\zeta:=\Big\{x\in\Q_p^n:\exists\{\lambda_k\}_{k=1}^n\in \prod_{k=1}^nB_{\rho^k}(0)\text{ s.t. }x=\zeta(t)+\sum_{k=1}^n\lambda_k\zeta^{(k)}(t)\Big\}.
	\end{equation*}
	Throughout this section, when the $\zeta$ superscript on $\mathcal{U}_{I,t}$ is suppressed, it will be assumed that $\mathcal{U}_{I,t}=\mathcal{U}_{I,t}^{\gamma}$ where $\gamma(t)=(\frac{t}{1!},\ldots,\frac{t^n}{n!})$ is the moment curve. We will write as well
    \begin{equation*}
        \mathcal{U}_I^\zeta=\bigcup_{t\in I}\mathcal{U}_{I,t}^\zeta.
    \end{equation*}
    This choice of caps is useful for technical reasons that appear in the proof below; in practice, they can generally be compared with other natural caps at slightly coarser scales.
 
	Our goal will be to bound the linear decoupling constant $\operatorname{Dec}_{\ell^2L^{q_n}}(\{\mathcal{U}_I\}_{I\in\mathcal{P}(\Z_p,\delta)})$. We record for later reference an abbreviation:
    \begin{definition}\label{lindef}
        For $\delta\in p^{-\N}$, define $\mathcal{D}_n(\delta)=\operatorname{Dec}_{\ell^2L^{q_n}}(\{\mathcal{U}_I\}_{I\in\mathcal{P}(\Z_p,\delta)})$.
    \end{definition}
 
    We will need a bilinear analogue as well:
	
	\begin{definition}[Symmetric bilinear decoupling constant]\label{bildef} Fix $\delta\in p^{-\N}$. We define the symmetric bilinear decoupling constant $\mathcal{B}_n(\delta)$ to be the infimal (real) constant $C$ such that the following holds. Suppose $I,J\in\mathcal{P}(\Z_p,p^{-1})$ are distinct. For each $I_i\in\mathcal{P}(I,\delta)$, let $f_i\in\mathcal{S}(\Q_p^n)$ be such that $\hat{f}_i$ is supported in $\mathcal{U}_{I_i}$; similarly, for each $J_i\in\mathcal{P}(J,\delta)$, let $g_i\in\mathcal{S}(\Q_p^n)$ be such that $\hat{g}_i$ is supported in $\mathcal{U}_{J_i}$. Then
		\begin{equation*}
			\int_{\Q_p^n}|f_I|^{q_n/2}|g_J|^{q_n/2}\leq C^{q_n}\left(\sum_i\|f_i\|_{L^{q_n}(\Q_p^n)}^2\right)^{q_n/4}\left(\sum_i\|g_i\|_{L^{q_n}(\Q_p^n)}^2\right)^{q_n/4}.
		\end{equation*}
		
	\end{definition}
	
	By H\"older we have the trivial $\mathcal{B}_n(\delta)\leq\mathcal{D}_n(\delta)$. Before proceeding, we record the following standard converse, adapted to the $p$-adic setting:
	\begin{proposition}[Bilinear-to-linear reduction]\label{biltolinprop}
		If $\delta=p^{-N}$, then
		\begin{equation}\label{biltolin}
			\mathcal{D}_n(\delta)\leq p^{1/2}\Big(1+\sum_{j=1}^N\mathcal{B}_n(p^{j-1}\delta)^2\Big)^{1/2}.
		\end{equation}
	\end{proposition}
	\begin{proof}
		We formulate a Whitney cube decomposition for $\Z_p^n$; due to the ultrametric on $\Z_p$, this will be simpler than the Euclidean analogue. For each $j\geq 1$, define $\mathcal{W}_j$ as
		\begin{equation*}
			\mathcal{W}_j:=\Big\{B_{p^{-j}}(x)\times B_{p^{-j}}(y):x,y\in\Z_p,|x-y|_p=p^{-j+1}\Big\}.
		\end{equation*}
		Observe that, if $|x-x'|\leq p^{-j}$, then $B_{p^{-j}}(x)=B_{p^{-j}}(x')$, so $\mathcal{W}_j$ contains exactly $p^j(p^j-1)$ elements. Additionally note that $\bigcup_{j\geq 1}\mathcal{W}_j$ defines a partition of $\Z_p^2\setminus\Delta$, where $\Delta\subseteq\Z_p\times\Z_p$ denotes the diagonal.
		
		To verify the estimate \ref{biltolin}, let $\{f_I\}_{I\in\mathcal{P}(\Z_p,\delta)}$ be a tuple with $\hat{f}_I$ supported in $\mathcal{U}_I$. Then we may write
		\begin{equation*}
			\begin{split}
				\Big\|\sum_{I\in\mathcal{P}(\Z_p,\delta)}&f_I\Big\|_{L^{q_n}(\Q_p^n)}=\Big\|\sum_{I,I'\in\mathcal{P}(\Z_p,\delta)}f_I\overline{f_{I'}}\Big\|_{L^{q_n/2}(\Q_p^n)}^{1/2}\\
				&\leq\Big[\sum_{I\in\mathcal{P}(\Z_p,\delta)}\big\|f_I\big\|_{L^{q_n}(\Q_p^n)}^2+\sum_{j=1}^N\sum_{J=J_0\times J_1\in\mathcal{W}_j}\big\|f_{J_0}\overline{f_{J_1}}\big\|_{L^{q_n/2}(\Q_p^n)}\Big]^{1/2}.
			\end{split}
		\end{equation*}
		For each $1\leq j\leq N$ and $J=J_0\times J_1\in\mathcal{W}_j$, by decoupling and affine rescaling we have
		\begin{equation*}
			\Big\|f_{J_0}\overline{f_{J_1}}\Big\|_{L^{q_n/2}(\Q_p^n)}\leq\mathcal{B}_n(p^{j-1}\delta)^2\left(\sum_{K\in\mathcal{P}(J_0,\delta)}\|f_K\|_{L^{q_n}(\Q_p^n)}^2\right)^{1/2}\left(\sum_{K\in\mathcal{P}(J_1,\delta)}\|f_K\|_{L^{q_n}(\Q_p^n)}^2\right)^{1/2},
		\end{equation*}
		so that (appealing to the AM-GM inequality)
		\begin{equation*}
			\left\|\sum_{I\in\mathcal{P}(\Z_p,\delta)}f_I\right\|_{L^{q_n}(\Q_p^n)}\leq\Big(1+(p-1)\sum_{j=1}^N\mathcal{B}_n(p^{j-1}\delta)^2\Big)^{1/2}\left(\sum_{I\in\mathcal{P}(\Z_p,\delta)}\big\|f_I\big\|_{L^{q_n}(\Q_p^n)}^2\right)^{1/2},
		\end{equation*}
		which implies
		\begin{equation*}
			\mathcal{D}_n(\delta)\leq p^{1/2}\Big(1+\sum_{j=1}^N\mathcal{B}_n(p^{j-1}\delta)^2\Big)^{1/2},
		\end{equation*}
		as was to be shown.
		
	\end{proof}
	
	We also will need a system of asymmetric bilinear decoupling constants.
	\begin{definition}[Asymmetric bilinear decoupling constant]\label{asdef} Fix $\delta=p^{-\beta}\in p^{-\N}$. For $s,t\in[0,1]$ with $s\beta,t\beta\in\Z$, define the asymmetric bilinear decoupling constant $\mathcal{B}_{n,k,s,t}(\delta)$ to be the infimal (real) constant $C$ such that the following holds. Suppose $I,J$ are distinct balls of radius at most $\delta^s,\delta^t$, respectively, contained in distinct members of $\mathcal{P}(\Z_p,p^{-1})$. For each $I_i\in\mathcal{P}(I,\delta)$, let $f_i\in\mathcal{S}(\Q_p^n)$ be such that $\hat{f}_i$ is supported in $\mathcal{U}_{I_i}$; similarly, for each $J_i\in\mathcal{P}(J,\delta)$, let $g_i\in\mathcal{S}(\Q_p^n)$ be such that $\hat{g}_i$ is supported in $\mathcal{U}_{J_i}$. Then
		\begin{equation*}
			\int_{\Q_p^n}|f_I|^{q_k}|g_J|^{q_n-q_k}\leq C^{q_n}\left(\sum_i\|f_i\|_{L^{q_n}(\Q_p^n)}^2\right)^{\frac{q_k}{2}}\left(\sum_i\|g_i\|_{L^{q_n}(\Q_p^n)}^2\right)^{\frac{q_n-q_k}{2}}.
		\end{equation*}
		
	\end{definition}
	
	We control the symmetric bilinear decoupling constant by the asymmetric bilinear decoupling constants:
	
	\begin{lemma}\label{symbyasym}
		Let $0\leq k<n$, $\delta=p^{-\beta}\in p^{-\N}$, and $s,t\in[0,1]$ such that $s\beta,t\beta\in\Z$. Then
		\begin{equation*}
			\mathcal{B}_n(\delta)\leq\delta^{-sq_k/q_n}\delta^{-t(q_n-q_k)/q_n}\mathcal{B}_{n,k,s,t}(\delta).
		\end{equation*}
	\end{lemma}
	\begin{proof}
		Identical to the proof of Lemma 3.4 of \cite{guo2021short}; we validate the particular constant. Let $I,I'\in \mathcal{P}(\Z_p,p^{-1})$ be distinct. Fix $\{f_K\}_{K\in \mathcal{P}(I,\delta)\cup\mathcal{P}(I',\delta)}$ be a tuple as in the statement of Definition \ref{bildef}. Suppose $k\neq 0$. By several applications of H\"older,
        \begin{equation}\label{symbyasymholder}
            \int_{\Q_p^n}|f_I|^{q_n/2}|f_{I'}|^{q_n/2}\leq\left(\int_{\Q_p^n}|f_I|^{q_k}|f_{I'}|^{q_n-q_k}\right)^{1/2}\left(\int_{\Q_p^n}|f_I|^{q_n-q_k}|f_{I'}|^{q_k}\right)^{1/2},
        \end{equation}
        and considering the first factor:
        \begin{equation*}
            \int_{\Q_p^n}|f_I|^{q_k}|f_{I'}|^{q_n-q_k}\leq[\#\mathcal{P}(I,\delta^s)]^{q_k-1}[\#\mathcal{P}(I',\delta^t)]^{q_n-q_k-1}\sum_{\substack{J\in\mathcal{P}(I,\delta^s)\\J'\in\mathcal{P}(I',\delta^t)}}\int_{\Q_p^n}|f_J|^{q_k}|f_{J'}|^{q_n-q_k},
        \end{equation*}
        and recalling Definition \ref{asdef} we have
        \begin{equation*}
            \int_{\Q_p^n}|f_J|^{q_k}|f_{J'}|^{q_n-q_k}\leq\mathcal{B}_{n,k,s,t}(\delta)^{q_n}\left(\sum_{K\in\mathcal{P}(J,\delta)}\|f_K\|_{q_n}^2\right)^{q_k/2}\left(\sum_{K\in\mathcal{P}(J',\delta)}\|f_K\|_{q_n}^2\right)^{(q_n-q_k)/2}.
        \end{equation*}
        Applying the trivial bounds $\#\mathcal{P}(I,\delta^s),\#\mathcal{P}(I',\delta^t)$ and on the operator norms of the inclusions $\ell^2\hookrightarrow\ell^{q_k},\ell^{q_n-q_k}$, we conclude that
        \begin{equation*}
            \begin{split}
                \int_{\Q_p^n}|f_I|^{q_k}|f_{I'}|^{q_n-q_k}&\leq p^{2-q_n}\delta^{-s(q_k-1)-t(q_n-q_k-1)}\mathcal{B}_{n,k,s,t}(\delta)^{q_n}\\
                &\times\left(\sum_{K\in\mathcal{P}(I,\delta)}\|f_K\|_{q_n}^2\right)^{q_k/2}\left(\sum_{K\in\mathcal{P}(I',\delta)}\|f_K\|_{q_n}^2\right)^{(q_n-q_k)/2}
            \end{split}
        \end{equation*}
        Considering the other factor in \ref{symbyasymholder}, we break $f_I$ into $\delta^{-t}$-many pieces and $f_{I'}$ into $\delta^{-s}$-many pieces and apply H\"older, and we obtain an identical estimate. Since $s,t\geq 0$ and $q_n\geq 2$, the result follows.

        Finally, we observe that when $k=0$, the same calculation (disregarding the terms involving a $k$) may be done to obtain an estimate of the form
        \begin{equation*}
            \int_{\Q_p^n}|f_I|^{q_n/2}|f_{I'}|^{q_n/2}\leq p^{2-q_n}\delta^{s+t}\delta^{-tq_n}\mathcal{B}_{n,k,s,t}(\delta)^{q_n}\left(\sum_{K\in\mathcal{P}(I,\delta)}\|f_K\|_{q_n}^2\right)^{q_n/4}\left(\sum_{K\in\mathcal{P}(I',\delta)}\|f_K\|_{q_n}^2\right)^{q_n/4}
        \end{equation*}
        and we are done.
	\end{proof}

    We also record that the asymmetric bilinear decoupling constants can be controlled by the linear decoupling constants:
    \begin{lemma}\label{asymbylin} If $1\leq k\leq n-1$, $\delta=p^{-\beta}\in p^{-\N}$, and $s,t\in[0,1]$ are such that $\beta t,\beta t\frac{n-k+1}{k}\in\Z$, then
    \begin{equation*}
        \mathcal{B}_{n,k,\frac{n-k+1}{k}t,t}(\delta)\leq\mathcal{D}_n(\delta^{1-\frac{n-k+1}{k}t})^{q_k/q_n}\mathcal{D}_n(\delta^{1-t})^{(q_n-q_k)/q_n}
    \end{equation*}
        
    \end{lemma}
    \begin{proof}
        Immediate from H\"older; see also the calculation in the ``Proof of Theorem 1.2" in \cite{guo2021short}.
    \end{proof}

    For $k\leq n$ and $\omega\in\Z_p$, write $V^k(\omega)=\operatorname{span}_{\Q_p^n}\{\gamma^{(1)}(\omega),\ldots,\gamma^{(k)}(\omega)\}$. The following is superficially identical to an estimate in \cite{guo2021short}, but we emphasize that we are considering the $p$-adic norm on both sides.

    \begin{lemma}[$p$-adic Vandermonde determinant] Under the assumption $p>n$, we have
        \begin{equation}\label{vandermonde}
            |\det[\gamma^{(1)}(t),\ldots,\gamma^{(k)}(t),\gamma^{(1)}(s),\ldots,\gamma^{(n-k)}(s)]|_p=|s-t|_p^{k(n-k)}
        \end{equation}
    \end{lemma}
    \begin{proof}
        We have, for $1\leq i\leq n-k$,
        \begin{equation*}
            \gamma^{(i)}(s)=\sum_{j=i}^n\frac{1}{(j-i)!}\gamma^{(j)}(t)(t-s)^{j-i}
        \end{equation*}
        Plugging in to the left-hand side of \ref{vandermonde},
        \begin{equation*}
            \begin{split}
            \det&[\gamma^{(1)}(t),\ldots,\gamma^{(k)}(t),\gamma^{(1)}(s),\ldots,\gamma^{(n-k)}(s)]\\
            &=\sum_{j_1,\ldots,j_{n-k}:n\geq j_a\geq a}\det[\gamma^{(1)}(t),\ldots,\gamma^{(k)}(t),\gamma^{(j_1)}(t),\ldots,\gamma^{(j_{n-k})}(t)]\prod_{a=1}^{n-k}\frac{(t-s)^{j_a-a}}{(j_a-a)!}
            \end{split}
        \end{equation*}
        Observe that the determinant summand vanishes when $(j_1,\ldots,j_{n-k})$ is not a permutation of $(k+1,\ldots,n)$. On the other hand, when $(j_1,\ldots,j_{n-k})$ is a permutation of $(k+1,\ldots,n)$, we see that
        \begin{equation*}
            \prod_{a=1}^{n-k}(t-s)^{j_a-a}=(t-s)^{k(n-k)}
        \end{equation*}
        so that
        \begin{equation*}
            \begin{split}
                \det&[\gamma^{(1)}(t),\ldots,\gamma^{(k)}(t),\gamma'(s),\ldots,\gamma^{(n-k)}(s)]\\
                &=(t-s)^{k(n-k)}\sum_{\substack{(j_1,\ldots,j_{n-k})\sim(k+1,\ldots,n)\\j_a\geq a}}\operatorname{sgn}(j_1,\ldots,j_{n-k})\prod_{a=1}^{n-k}\frac{1}{(j_a-a)!}.
            \end{split}
        \end{equation*}
        Here $\sim$ denotes permutation, and $\operatorname{sgn}$ denotes the sign of the implied permutation.

        Finally, observe that
        \begin{equation*}
            \det[\gamma^{(1)}(0),\ldots,\gamma^{(k)}(0),\gamma^{(1)}(1),\ldots,\gamma^{(n-k)}(1)]=\sum_{\substack{(j_1,\ldots,j_{n-k})\sim(k+1,\ldots,n)\\j_a\geq a}}\operatorname{sgn}(j_1,\ldots,j_{n-k})\prod_{a=1}^{n-k}\frac{1}{(j_a-a)!}
        \end{equation*}
        whereas the left-hand side may be computed (c.f. \cite{kalman1984generalized},  display (14)) as
        \begin{equation*}
            \left[\prod_{j=1}^n\frac{1}{j!}\right]\left[\prod_{j=1}^kj!\right]\left[\prod_{j=1}^{n-k}j!\right](1-0)^{k(n-k)}
        \end{equation*}
        so that, since $p>n$,
        \begin{equation*}
            |\det[\gamma^{(1)}(t),\ldots,\gamma^{(k)}(t),\gamma^{(1)}(s),\ldots,\gamma^{(n-k)}(s)]|_p=|s-t|_p^{k(n-k)}
        \end{equation*}
        and we may conclude \ref{vandermonde}.
    \end{proof}

    Before reaching the main estimate of the theorem, we demonstrate an equivalence between the model decoupling constant $\mathcal{D}_n(\delta)=\mathcal{D}_n^\gamma(\delta)$ and that of general convex curves. We first need a technical lemma:

    \begin{lemma}[Stability of linear decoupling constants]\label{stability} Let $n\geq 2$. Suppose $\zeta:\Z_p\to\Q_p^n$ is $\mathbf{C}^{n}$. Suppose $\delta,\delta'\in p^{-\N}$ are such that $\delta<\delta'$. Then we have the estimate
    \begin{equation*}
        \mathcal{D}_n^\zeta(\delta)\leq(\delta'/\delta)^{1/2}\mathcal{D}_n^\zeta(\delta').
    \end{equation*}
    \end{lemma}
    \begin{proof}
        For any choice of functions $\{f_I\}_{I\in\mathcal{P}(\Z_p,\delta)}$ with $\hat{f}_I$ supported in $\mathcal{U}_I^\zeta$, then we have
        \begin{equation*}
            \left\|\sum_{I\in\mathcal{P}(\Z_p,\delta)}f_I\right\|_{L^{q_n}}\leq\mathcal{D}_n^\zeta(\delta')\left(\sum_{J\in\mathcal{P}(\Z_p,\delta')}\left\|\sum_{I\in\mathcal{P}(J,\delta)}f_I\right\|_{L^{q_n}}^2\right)^{1/2},
        \end{equation*}
        and the desired estimate follows from the triangle inequality and Cauchy-Schwarz.
    \end{proof}
	
	\begin{proposition}[Decoupling for convex curves]\label{convdec} Let $k\geq 2$. Suppose $\zeta:\Z_p\to\Q_p^k$ is a $\mathbf{C}^{k+1}$ curve that is convex and has bounded derivatives, in the sense that it satisfies \ref{convexcond} and \ref{bdddercond}. For each $\delta\in p^{-\N}$, write $\mathcal{D}_k^\zeta(\delta)$ for the $\ell^2L^{q_k}$ decoupling constant associated to the partition $\{\mathcal{U}_I^\zeta\}_{I\in\mathcal{P}(\Z_p,\delta)}$. Suppose $\mathcal{D}_k(\rho)\leq C_{k,\eps}\rho^{-\eps}$ for all $\eps=\frac{1}{\ell},\ell\in\N$, and $\rho\in p^{-\N}$. Then, for each $\delta\in p^{-\N}$ and each $\eps=\frac{4}{\ell}$ with $\ell\in\N$, we have the estimate
    \begin{equation*}
        \mathcal{D}_k^\zeta(\delta)\leq\mathcal{E}_{k,\eps}^\zeta\delta^{-\eps},
    \end{equation*}
    where the constant $\mathcal{E}_{k,\eps}^\zeta$ may be taken as
    \begin{equation*}
        \mathcal{E}_{k,\eps}^\zeta=C_{k,\eps/4}^{2k\lceil\log(4\eps^{-1})\rceil}p^{(1+\eps^{-1})k^{2k\lceil\log(8\eps^{-1})\rceil}}\times\max(1,c^{-1}C^{k-1}\|\zeta\|_{\mathbf{C}_\circ^{k+1}}).
    \end{equation*}
		
	\end{proposition}
		\begin{proof}
			This is essentially identical to the proof of Lemma 3.6 of \cite{guo2021short}; we highlight the needed modifications to produce the $p$-adic analogue. Fix $\eps=\frac{4}{\ell}$ for some $\ell\in\N$, and write $Z=k^{2k\lceil\log 2\ell\rceil}$. Write $\alpha=\max(1,c^{-1}C^{k-1}\|\zeta\|_{\mathbf{C}_\circ^{k+1}})$, where $c,C$ are the constants from \ref{convexcond} and \ref{bdddercond}. Write $r\in\N$ for the smallest integer such that $\alpha\leq p^{rZ}$. Choose any $\delta\in p^{-\ell Z\N}$ such that $\delta^{-\eps/2}>p^{rZ}$. Fix $\kappa=\delta^\eps p^{rZ}$, so that $\kappa<\alpha^{-1}$ and $\kappa<\delta^{\eps/2}$. Furthermore, writing
            \begin{equation*}
                m_*=\left\lceil\frac{\log(4\eps^{-1})}{\log(1+k^{-1})}\right\rceil\leq 2k\lceil\log(4\eps^{-1})\rceil,
            \end{equation*}
            we see that, for each $1\leq m\leq m_*$, we have
            \begin{equation*}
                \kappa^{(\frac{k+1}{k})^m}\in p^{-\N},
            \end{equation*}
            and we have the inequalities
            \begin{equation*}
                \kappa^{(\frac{k+1}{k})^{m_*-1}}\leq\delta^{\frac{k}{k+1}}<\kappa.
            \end{equation*}
   
            With these parameters chosen, suppose $I\in\mathcal{P}(\Z_p,\kappa)$, and $\{f_{I'}\}_{I'\in\mathcal{P}(I,\kappa^{\frac{k+1}{k}})}$ are such that $f_{I'}$ has Fourier support in $\mathcal{U}_{I'}^\zeta$. If $c_I\in I$, then for any other $\xi\in I$ we may write
			\begin{equation*}
				\zeta_i(\xi)=\zeta_i(c_I)+\sum_{j=1}^k\frac{\zeta_i^{(j)}(c_I)}{k!}(\xi-c_I)^j+(c_I-\xi)^{k+1}\Lambda_{k+1}^i(c_I,\xi)
			\end{equation*}
			where $\Lambda_{k+1}^i$ is continuous and vanishes along the diagonal. Write $A_I=A_{c_I}^\zeta$. By Lemma \ref{matrixinverselemma}, $A_I^{-1}$ has $\leq c^{-1}C^{k-1}$ operator norm. Then the curve
			\begin{equation*}
				\eta(\xi)=A_I^{-1}(\zeta(\xi+c_I)-\zeta(c_I))
			\end{equation*}
			satisfies
			\begin{equation*}
				\eta_i(\xi)=\frac{\xi^i}{i!}+\xi^{k+1}\Lambda_{k+1}^{i}(c_I,\xi),
			\end{equation*}
            and
            \begin{equation*}
                \|\eta\|_{\mathbf{C}_\circ^{k+1}}\leq c^{-1}C^{k-1}\|\zeta\|_{\mathbf{C}_\circ^{k+1}}.
            \end{equation*}
            By Lemma \ref{phicomp}(c), it further holds that
			\begin{equation*}
				\eta_i^{(j)}(\xi)=\frac{\xi^{i-j}}{(i-j)!}\delta_{i\geq j}+\xi\sum_{\iota =1}^{j+1}\overline{\Phi}_{j+1}\eta_i(0,\ldots,0,\xi,\ldots,\xi),
			\end{equation*}
			where $\overline{\Phi}_{j+1}$ is defined in Appendix C, and there are $\iota$-many $0$'s. By the affine invariance of decoupling constants,
			\begin{equation*}
				\mathcal{D}_k(\{\mathcal{U}_{I'}^\zeta\}_{I'\in\mathcal{P}(I,\kappa^{\frac{k+1}{k}})})=\mathcal{D}_k(\{\mathcal{U}_{I'-c_I}^\eta\}_{I'\in\mathcal{P}(I,\kappa^{\frac{k+1}{k}})})
			\end{equation*}
			We claim that the boxes $\mathcal{U}_{I'-c_I}^\eta$ and $\mathcal{U}_{I'-c_I}^\gamma$ are comparable. Indeed, if $x\in\mathcal{U}_{I'-c_I}^\eta$, then there are $\{\lambda_j\}_{j=1}^n\in \prod_{j=1}^kB(0,\kappa^{j\frac{k+1}{k}})$ and $\xi_{I'-c_I}\in I'-c_I\subseteq B_\kappa(0)$ such that
			\begin{equation*}
				x=\eta(\xi_{I'-c_I})+\sum_{j=1}^k\eta^{(j)}(\xi_{I'-c_I})\lambda_j,
			\end{equation*}
			so that
			\begin{equation*}
                x_i=\gamma_i(\xi_{I'-c_I})+\sum_{j=1}^k\gamma_i^{(j)}(\xi_{I'-c_I})\lambda_j+E,
			\end{equation*}
            where $|E|_p\leq c^{-1}C^{k-1}\|\zeta\|_{\mathbf{C}_\circ^{k+1}}\kappa^{2+\frac{1}{k}}$, i.e. $\mathcal{U}_{I'-c_I}^\eta$ is contained in a $c^{-1}C^{k-1}\|\zeta\|_{\mathbf{C}_\circ^{k+1}}\kappa^{2+\frac{1}{k}}$-neighborhood of $\mathcal{U}_{I'-c_I}^\gamma$. Since the family $\{\mathcal{U}_{I'-c_I}^\gamma\}_{I'\in\mathcal{P}(I,\kappa^{\frac{k+1}{k}})}$ are $\kappa^{\frac{k+1}{k}}$-separated, we obtain
			\begin{equation*}
				\operatorname{Dec}_{\ell^2L^{q_k}}\Big(\{\mathcal{U}_{I'-c_I}^\eta\}_{I'\in\mathcal{P}(I,\kappa^{\frac{k+1}{k}})}\Big)\leq\operatorname{Dec}_{\ell^2L^{q_k}}\Big(\{\mathcal{U}_{I'-c_I}^\gamma\}_{I'\in\mathcal{P}(I,\kappa^{\frac{k+1}{k}})}\Big),
    			\end{equation*}
    		because $c^{-1}C^{k-1}\|\zeta\|_{\mathbf{C}_\circ^{k+1}}<\kappa^{-1}$, so that (using the affine invariance to compare $\zeta$ to $\eta$, and the affine rescaling of $\zeta$),
    		\begin{equation*}
    			\operatorname{Dec}_{\ell^2L^{q_k}}\Big(\{\mathcal{U}_{I'}^\zeta\}_{I'\in\mathcal{P}(I,\kappa^{\frac{k+1}{k}})}\Big)\leq\operatorname{Dec}_{\ell^2L^{q_k}}\Big(\{\mathcal{U}_{I'}^\gamma\}_{I'\in\mathcal{P}(I,\kappa^{\frac{k+1}{k}})}\Big).
    		\end{equation*}
            On the other hand, affine rescaling implies
            \begin{equation*}
                \operatorname{Dec}_{\ell^2L^{q_k}}\Big(\{\mathcal{U}_{I'}^\gamma\}_{I'\in\mathcal{P}(I,\kappa^{\frac{k+1}{k}})}\Big)\leq\mathcal{D}_k(\kappa^{1/k}),
            \end{equation*}
            so that
            \begin{equation*}
                \operatorname{Dec}_{\ell^2L^{q_k}}\Big(\{\mathcal{U}_{I'}^\zeta\}_{I'\in\mathcal{P}(I,\kappa^{\frac{k+1}{k}})}\Big)\leq\mathcal{D}_k(\kappa^{1/k}).
            \end{equation*}
    
            We have established this whenever $\kappa>\delta^{\frac{k}{k+1}}$ and $I\in\mathcal{P}(\Z_p,\kappa)$. If we iterate this $m$-many times, where $\kappa^{(\frac{k+1}{k})^m}>\delta^{\frac{k}{k+1}}$, we obtain
            \begin{equation*}
                \operatorname{Dec}_{\ell^2L^{q_k}}\Big(\{\mathcal{U}_{I'}^\zeta\}_{I'\in\mathcal{P}(I,\kappa^{m\frac{k+1}{k}})}\Big)\leq \prod_{j=1}^m\mathcal{D}_k(\kappa^{(\frac{k+1}{k})^{j-1}/k}).
            \end{equation*}
    
            Observe that $\#\mathcal{P}(\Z_p,\kappa)\leq\kappa^{-1}$, so by flat decoupling we have
            \begin{equation*}
                \operatorname{Dec}_{\ell^2L^{q_k}}\Big(\{\mathcal{U}_{I}^\zeta\}_{I\in\mathcal{P}(\Z_p,\kappa)}\Big)\leq\kappa^{\frac{1}{q_k}-\frac{1}{2}}.
            \end{equation*}
            If $\kappa^{(\frac{k+1}{k})^{m+1}}\leq\delta^{\frac{k}{k+1}}$ as well, then for each $I'\in\mathcal{P}(\Z_p,\kappa^{(\frac{k+1}{k})^m})$ we have
            \begin{equation*}
                \operatorname{Dec}_{\ell^2L^{q_k}}\Big(\{\mathcal{U}_J^\zeta\}_{J\in\mathcal{P}(I,\delta)}\Big)\leq(\delta\kappa^{-(\frac{k+1}{k})^m})^{\frac{1}{2}-\frac{1}{q_k}}\leq(\kappa^{-(\frac{k+1}{k})^2})^{\frac{1}{2}-\frac{1}{q_k}}.
            \end{equation*}
           Finally, from the hypothesis that $\mathcal{D}_k(\rho)\leq C_{k,\eps}\rho^{-\eps}$ for all $\eps>0$ and all $\rho\in p^{-\N}$, we obtain
            \begin{equation*}
                \mathcal{D}_k^\zeta(\delta)\leq C_{k,\eps/4}^m\kappa^{(-(\frac{k+1}{k})^2-1)(\frac{1}{2}-\frac{1}{q_k})-\frac{\eps}{4}[(\frac{k+1}{k})^m-1]}.
            \end{equation*}
            Observe that $m<m_*$. Recalling that $\kappa>\alpha^{-1}p^{-Z}\delta^\eps$, we conclude
            \begin{equation*}
                \mathcal{D}_k^\zeta(\delta)\leq C_{k,\eps/4}^{2k\lceil\log(4\eps^{-1})\rceil}p^{2k^{2k\lceil\log(4\eps^{-1})\rceil}}\times\max(1,c^{-1}C^{k-1}\|\zeta\|_{\mathbf{C}_\circ^{k+1}})\times\delta^{-\eps}.
            \end{equation*}
            
            It remains to remove the special assumptions on the value of $\delta$. Take $\eps=\frac{4}{\ell}$ for some $\ell\in\N$. Suppose first that $\delta\in p^{-\N}$ is such that $\delta > p^{-\eps^{-1}Z}\alpha^{-1}$, where $\alpha$ is as above. Then we may trivially bound
            \begin{equation*}
                \mathcal{D}_k^\zeta(\delta)\leq\delta^{-1}\delta^{-\eps},
            \end{equation*}
            and
            \begin{equation*}
                \delta^{-1}\leq p^{\eps^{-1}k^{2k\lceil\log(4\eps^{-1})\rceil}}\max(1,c^{-1}C^{k-1}\|\zeta\|_{\mathbf{C}_\circ^{k+1})}).
            \end{equation*}
            We are done in this case, thanks to $C_{k,\eps}\geq 1$. Suppose instead that $\delta\in p^{-\N}$ satisfies the inequalities
            \begin{equation*}
                p^{-\ell k^{2k\lceil\log 2\ell\rceil}(K+1)}<\delta<p^{-\ell k^{2k\lceil\log 2\ell\rceil}K},\quad K\in\N.
            \end{equation*}
            Comparing $\mathcal{D}_k^\zeta(\delta)$ to $\mathcal{D}_k^\zeta(\delta')$, where $\delta'=p^{-\ell k^{2k\lceil\log 2\ell\rceil}K}$, we obtain the estimate 
            \begin{equation*}
                \mathcal{D}_k^\zeta(\delta)\leq C_{k,\eps/4}^{2k\lceil\log(4\eps^{-1})\rceil}p^{(1+\eps^{-1})k^{2k\lceil\log(8\eps^{-1})\rceil}}\times\max(1,c^{-1}C^{k-1}\|\zeta\|_{\mathbf{C}_\circ^{k+1}})\times\delta^{-\eps}.
            \end{equation*}
		\end{proof}
	
	\subsection{Lower dimensional estimates}

    We next establish the induction-on-dimension estimates in the $p$-adic setting. This is the component of the argument that requires the most careful rewriting; the core steps for inducting on dimension are essentially the same, but are usually phrased essentially in the language of Euclidean geometry. We choose instead to phrase things via matrix algebra, and the difficulties disappear.
    
    A critical input of these estimates in the Euclidean setting is the Fourier slicing theorem, which requires some slight rewording in our setting; the dot product $\cdot:\Q_p^n\times\Q_p^n\to\Q_p$ is isotropic for every $n\geq 3$ and odd $p$ (see the start of this section of the appendix). A near relative is available, which we produce now.
    \begin{lemma}[$p$-adic Fourier slicing]\label{slicing}
        Suppose that $f:\Q_p^n$ has Fourier support inside of $\Omega\subseteq\Q_p^n$. Let $H\subseteq\Q_p^n$ be a $k$-dimensional linear subspace, $B:\Q_p^k\to \Q_p^n$ a linear isomorphism onto $H$, and $z\in\Q_p^n$ arbitrary. Write $f^z$ for the function $H\to\C$, $f^z(x)=f(x+z)$. Then $f^z\circ B$ has Fourier support in the set $B^{\top}\Omega$.
    \end{lemma}
    \begin{proof}

        By extension of bases, we may write $B=B'\circ\iota$ where $\iota:\Q_p^k\hookrightarrow\Q_p^n$ is the inclusion into the first $k$ coordinates and $B':\Q_p^n\to\Q_p^n$ is a linear isomorphism. By the usual change-of-variable,
        \begin{equation*}
            \widehat{f\circ B'}=|\det B'|_p^{-1}\cdot\hat{f}\circ (B')^{-\top}.
        \end{equation*}
        It follows that we may assume that $H$ is the subspace $\{x_{k+1}=\ldots=x_n=0\}$, for which $H^\perp=\{x_1=\ldots=x_k=0\}$. Then, for $y\in H$ and $z\in H^\perp$,
        \begin{equation*}
            \begin{split}
                f^z(y)&=\int\chi((y+z)\cdot\xi)\hat{f}(\xi)d\xi\\
                &=\int_H\chi(y\cdot\xi')\int_{H^\perp}\chi(z\cdot\xi'')\hat{f}(\xi',\xi'')d\xi''d\xi'.
            \end{split}
        \end{equation*}
        It follows that $\widehat{f^z}$ is supported in the set $\{\xi'\in H:\exists\xi''\text{ s.t. }\xi'+\xi''\in\Omega\}$. The result follows.

    \end{proof}

    We apply this to decouple functions using lower-dimensional estimates.
	
	\begin{lemma}\label{lowerdimdecest}  
		Let $k<n$ and assume that $\mathcal{D}_k(\delta)\leq C_{k,\eps}\delta^{-\eps}$ for all $\eps=\frac{1}{\ell},\ell\in\N$, and $\delta\in p^{-\N}$. Let $\delta=\delta^{-\beta}\in p^{-\N}$ and $\{f_I\}_{I\in\mathcal{P}(\Z_p,\delta)}$ have Fourier support in $\{\mathcal{U}_I\}_{I\in\mathcal{P}(\Z_p,\delta)}$. If $0\leq s,t\leq 1$ satisfy $0\leq s\leq(n-k+1)t/k$ and $s\beta,t\beta,t\beta(n-k+1)/k\in\Z$, then for any $J_1\in\mathcal{P}(\Z_p,\delta^s),J_2\in\mathcal{P}(\Z_p,\delta^t)$ in distinct cosets of $p\Z_p$, and for any $\eps=\frac{4}{\ell}$ with $\ell\in\N_{\geq 4}$, we have
		\begin{equation}
            \begin{split}
    			\int|f_{J_1}|^{q_k}|f_{J_2}|^{q_n-q_k}&\leq C_{k,\eps/4}^{2kq_k\lceil\log(4\eps^{-1})\rceil}p^{(1+\eps^{-1})k^{2k\lceil\log(8\eps^{-1})\rceil+2}}\\
                &\times\delta^{-q_k\left[\frac{(n-k+1)t}{k}-s\right]\eps}\left[\sum_{J\in\mathcal{P}(J_1,\delta^{(n-k+1)t/k})}\left(\int_{\Q_p^n}|f_J|^{q_k}|f_{J_2}|^{q_n-q_k}\right)^{2/q_k}\right]^{q_k/2}.
            \end{split}
		\end{equation}
	\end{lemma}
	\begin{proof}
		Pick any $\omega\in J_2$ and write $V=\text{span}(\gamma^{(1)}(\omega),\ldots,\gamma^{(n-k)}(\omega))$. Let $H=V^\perp$ be the orthogonal space to $V$ in $\Q_p^n$; observe that $\dim H=k$, since the dot product is still nondegenerate. Set $t'=(n-k+1)t/k$. 

        Define
        \begin{equation*}
            B=A_{-\omega}^\top\begin{bmatrix}
            0 \\ I_k
        \end{bmatrix};
        \end{equation*}
        it follows that $B$ defines a linear isomorphism $\Q_p^k\to H$. Write $\mu_H=B_*\mu_{\Q_p^k}$.
        
        We use Lemma \ref{slicing} to estimate integrals of $f$ along $H+z$. Let $B$ be a linear isomorphism $\Q_p^k\to H$. Observe that $f_{J_1}^z\circ B$ has Fourier support in $\mathcal{U}_{J_1}^{B^{\top}\gamma}$. By Fubini,
		\begin{equation*}
            \int_{\Q_p^n}|f_{J_1}|^{q_k}|f_{J_2}|^{q_n-q_k}d\mu=\int_{z\in\Q_p^n}\fint_{x\in B_H(z,\delta^{-t'k})}|f_{J_1}|^{q_k}(x)|f_{J_2}|^{q_n-q_k}(x)d\mu_H(x)d\mu_{\Q_p^n}(z)
		\end{equation*}
        Note $B_H(z,\delta^{-t'k})=B_H(z,\delta^{-(n-k+1)t})$. Write $A$ for the matrix with $j$'th column $\gamma^{(j)}(\omega)$, $1\leq j\leq n$. By uncertainty, $f_{J_2}$ is constant on translates of the set
        \begin{equation*}
            \mathcal{U}_{J_2}^*=A_\omega^{-\top}\cdot\operatorname{diag}\big(\delta^t,\ldots,\delta^{nt}\big)[\Z_p^n]
        \end{equation*}
        If $y\in B_H(0,\delta^{-t'k})$, then
        \begin{equation*}
            \operatorname{diag}\big(\delta^{-t},\ldots,\delta^{-nt}\big)(A_\omega^{\top}y)=\begin{bmatrix}
                0\\
                \vdots\\
                0\\
                \delta^{-(n-k+1)t}\gamma^{(n-k+1)}(\omega)\cdot y\\
                \vdots\\
                \delta^{-nt}\gamma^{(n)}(\omega)\cdot y
            \end{bmatrix}\in\Z_p^n,
        \end{equation*}
        since $\left|\delta^{-rt}\gamma^{(r)}(\omega)\cdot y\right|_p\leq\delta^{rt-(n-k+1)t}\leq 1$ for each $r\geq n-k+1$. Consequently, we have $B_H(0,\delta^{-t'k})\subseteq \mathcal{U}_{J_2}^*$, and so by uncertainty we have that $|f_{J_2}|^{q_n-q_k}$ is constant on $B_H(z,\delta^{-t'\ell})$. Thus
		\begin{equation*}
			\begin{split}
				\int_{z\in\Q_p^n}&\fint_{x\in B_H(z,\delta^{-t'k})}|f_{J_1}|^{q_k}(x)|f_{J_2}|^{q_n-q_k}(x)\\
				&=\int_{z\in\Q_p^n}|f_{J_2}|^{q_n-q_k}(z)\fint_{x\in B_H(z,\delta^{-t'k})}|f_{J_1}|^{q_k}(x)
			\end{split}
		\end{equation*}
  
        Then
        \begin{equation*}
            \begin{split}
            \int_{x\in B_H(z,\delta^{-t'\ell})}|f_{J_1}(x)|^{q_k}d\mu_H(x)&=\int_{x\in B_H(0,\delta^{-t'\ell})}|f_{J_1}^z(x)|^{q_k}d\mu_H\\
            &=\int_{y\in B^{-1}[B_H(0,\delta^{-t'\ell})]}|f_{J_1}^z(B(y))|^{q_k}_pd\mu_{\Q_p^k}(y),
            \end{split}
        \end{equation*}
        where we have used that $B_*\mu_{\Q_p^k}=\mu_H$. By Lemma \ref{slicing}, for each $J\in\mathcal{P}(J_1,\delta^{t'})$, the function $f_J\circ B$ is supported in the set $\mathcal{U}_J^{B^\top\gamma}$. By Lemma \ref{vandermonde}, the curve $B^\top\gamma:\Z_p\to\Q_p^k$ satisfies
        \begin{equation*}
            \max_{1\leq j\leq k}\max_{1\leq r\leq k}\sup_{\theta\in J_1}|B^\top\gamma_j^{(r)}(\theta)|_p\leq 1
        \end{equation*}
        and
        \begin{equation*}
            \inf_{\theta\in J_1}|\det[B^\top\gamma^{(1)}(\theta),\ldots,B^\top\gamma^{(k)}(\theta)]|_p\geq 1.
        \end{equation*}
        By Lemma \ref{localdec}, Prop. \ref{convdec}, and the inductive assumption,
		\begin{equation*}
            \begin{split}
			\int_{x\in B_H(z,\delta^{-t'k})}|f_{J_1}|^{q_k}(x)&\leq C_{k,\eps/4}^{2kq_k\lceil\log(4\eps^{-1})\rceil}p^{(1+\eps^{-1})k^{2k\lceil\log(8\eps^{-1})\rceil+2}}\\
            &\times(\#\mathcal{P}(J_1,\delta^{t'}))^{q_k\eps}\left(\sum_{J\in\mathcal{P}(J_1,\delta^{t'})}\|f_J\|_{L^{q_k}(B_H(z',\delta^{-b'k}))}^2\right)^{q_k/2}.
            \end{split}
		\end{equation*}
		Consequently,
		\begin{equation*}
            \begin{split}
			 \int_{\Q_p^n}|f_{J_1}|^{q_k}|f_{J_2}|^{q_n-q_k}\leq &\, C_{k,\eps/4}^{2kq_k\lceil\log(4\eps^{-1})\rceil}p^{(1+\eps^{-1})k^{2k\lceil\log(8\eps^{-1})\rceil+2}}\delta^{tk(n-k+1)-q_k\eps(t'-s)}\\
            &\times\int_{z\in\Q_p^n}\left(|f_{J_2}|^{2\frac{q_n-q_k}{q_k}}(z)\sum_{J\in\mathcal{P}(J_1,\delta^{t'})}\|f_J\|_{L^{q_k}(B_H(z,\delta^{-t'k}))}^2\right)^{q_k/2},
            \end{split}
		\end{equation*}
		which by Minkowski is bounded by
		\begin{equation*}
            \begin{split}
			&C_{k,\eps/4}^{2kq_k\lceil\log(4\eps^{-1})\rceil}p^{(1+\eps^{-1})k^{2k\lceil\log(8\eps^{-1})\rceil+2}}\delta^{tk(n-k+1)-q_k\eps(t'-s)}\\
            &\times\left(\sum_{J\in\mathcal{P}(J_1,\delta^{t'})}\left(\int_{z\in\Q_p^n}|f_{J_2}|^{q_n-q_k}(z)\|f_J\|_{L^{q_k}(B_H(z,\delta^{-t'k}))}^{q_k}\right)^{2/q_k}\right)^{q_k/2}.
            \end{split}
		\end{equation*}
        Finally,
        \begin{equation*}
            \begin{split}
                \int_{z\in\Q_p^n}|f_{J_2}|^{q_n-q_k}(z)\|f_J\|_{L^{q_k}(B_H(z,\delta^{-t'k}))}^{q_k}&=\int_{z\in\Q_p^n}|f_{J_2}(z)|^{q_n-q_k}\int_{x\in B_H(z,\delta^{-t'k})}|f_J(x)|^{q_k}\\
                &=\delta^{-t(n-k+1)k}\int_{z\in\Q_p^n}|f_{J_2}(z)|^{q_n-q_k}|f_J(z)|^{q_k},
            \end{split}
        \end{equation*}
        by virtue of local constancy; hence we have shown
        \begin{equation*}
            \begin{split}
            \int|f_{J_1}|^{q_k}|f_{J_2}|^{q_n-q_k}&\leq C_{k,\eps/4}^{2kq_k\lceil\log(4\eps^{-1})\rceil}p^{(1+\eps^{-1})k^{2k\lceil\log(8\eps^{-1})\rceil+2}}\\
            &\times\delta^{-q_k\eps(t'-s)}\left[\sum_{J\in\mathcal{P}(J_1,\delta^{(n-k+1)t/k})}\left(\int_{\Q_p^n}|f_J|^{q_k}|f_{J_2}|^{q_n-q_k}\right)^{2/q_k}\right]^{q_k/2},
            \end{split}
        \end{equation*}
        as was to be verified.
	\end{proof}
	
	\begin{corollary} If $k<n$ and $\mathcal{D}_k(\delta)\leq C_{k,\eps}\delta^{-\eps}$ for all $\eps=\frac{1}{\ell},\ell\in\N$, then for any $\delta,s,t$ satisfying the hypotheses of Lemma \ref{lowerdimdecest},
    \begin{equation*}
        \mathcal{B}_{n,k,s,t}(\delta)\leq C_{k,\eps/4}^{2kq_kq_n^{-1}\lceil\log(4\eps^{-1})\rceil}p^{(1+\eps^{-1})q_n^{-1}k^{2k\lceil\log(8\eps^{-1})\rceil+2}}\delta^{-\frac{q_k}{q_n}\left[\frac{(n-k+1)t}{k}-s\right]\eps}\mathcal{B}_{n,k,\frac{n-k+1}{k}t,t}(\delta).
    \end{equation*}
		
	\end{corollary}

    We record another application of H\"older; here we are able to go without an application of the uncertainty principle.
    \begin{lemma}\label{holder}
        If $1\leq k\leq n-1$, and if $\delta\in(0,1)$ and $s,t\in(0,1)$  are as above, then
        \begin{equation*}
            \mathcal{B}_{n,k,s,t}(\delta)\leq\mathcal{B}_{n,n-k,t,s}(\delta)^{\frac{1}{n-k+1}}\mathcal{B}_{n,k-1,s,t}(\delta)^{\frac{n-k}{n-k+1}}.
        \end{equation*}
    \end{lemma}
    \begin{proof}
        Let $\{f_i\}_i,\{g_i\}_i$ be families as in Def. \ref{asdef}. Then, writing $\theta_k=1/(n-k+1)$, we see that
        \begin{equation*}
            \begin{split}
            \int_{\Q_p^n}|\sum_if_i|^{q_k}|\sum_ig_j|^{q_n-q_k}&=\int_{\Q_p^n}|\sum_if_i|^{\theta_k(q_n-q_{n-k})}|\sum_if_i|^{(1-\theta_k)q_{k-1}}|\sum_ig_i|^{\theta_kq_{n-k}}|\sum_ig_i|^{(1-\theta_k)(q_n-q_{k-1})}\\
            &\leq\left(\int_{\Q_p^n}|\sum_if_i|^{q_n-q_{n-k}}|\sum_ig_i|^{q_{n-k}}\right)^{\theta_k}\left(\int_{\Q_p^n}|\sum_if_i|^{q_{k-1}}|\sum_ig_i|^{q_n-q_{k-1}}\right)^{1-\theta_k}
            \end{split}
        \end{equation*}
        from which the result is clear.
    \end{proof}
    The following consequence is identical to Lemma 4.2 of \cite{guo2021short}:
    \begin{lemma}\label{lowerdim}
        Suppose $\mathcal{D}_k(\delta)\leq C_{k,\eps}\delta^{-\eps}$ for all $1\leq k\leq n-1$ and all $\delta,\eps>0$. Suppose $1\leq k\leq n-1$ and let $\eps=\frac{4}{\ell}$ for some $\ell\in\N_{\geq 4}$. Then, for every $t\in[0,1]$ such that $t\leq\frac{k(n-k)}{(k+1)(n-k+1)}$ and either $k=1$ or $t\leq\frac{k-1}{n-k+2}$, we have, for each $\delta\in p^{-\N}$ for which $\delta^t,\delta^{\frac{n-k+1}{k}t}\in p^{-\N}$,
        \begin{equation*}
            \begin{split}
                \mathcal{B}_{n,k,\frac{n-k+1}{k}t,t}(\delta)&\leq p^{8\eps^{-4n\log n}}C_{n-k,\eps/4}^{2\lceil\log(4\eps^{-1})\rceil}C_{k-1,\eps/4}^{2k\lceil\log(4\eps^{-1})\rceil}\delta^{-\frac{n-k}{n}\frac{k-1+(k+1)(n-k+1)}{(k+1)(n-k+1)}t\eps}\\
                &\times\mathcal{B}_{n,n-k,\frac{k+1}{k}\frac{n-k+1}{n-k}t,\frac{n-k+1}{k}t}(\delta)^{\frac{1}{n-k+1}}\mathcal{B}_{n,k-1,\frac{n-k+2}{k-1}t,t}(\delta)^{\frac{n-k}{n-k+1}}.
            \end{split}
        \end{equation*}
    \end{lemma}

    We also record the following:
    \begin{lemma}\label{asequivlin}
        For any $\delta\in p^{-\N}$ and any $s,t$ such that $\mathcal{B}_{n,0,s,t}(\delta)$ is defined,
        \begin{equation*}
            \mathcal{B}_{n,0,s,t}(\delta)=\mathcal{D}_n(\delta^{1-t}).
        \end{equation*}
    \end{lemma}
    \begin{proof}
        For any particular tuples $\{f_i\}_i,\{g_i\}_i$, the inequality in Def. \ref{asdef} is just the linear decoupling inequality for the tuple $\{g_i\}_i$. The result follows by parabolic rescaling.
    \end{proof}

    \subsection{Induction on scales}

    In this section we run an induction on scales argument in order to prove Theorem \ref{momentcurve}. We mirror the arguments in Section 4 of \cite{guo2021short}; however, in order to verify the appropriate quantitative estimate, we produce a modified version. The former runs an analysis on the tropicalized quantities $\eta,\{A_k\}_k$, defined as the optimal exponents on various decoupling constants; the resulting analysis is clean, but does not admit estimates on the corresponding constants. We run a suitable finitary version of this argument, which gives somewhat loose (but explicit) estimates on the constant.

    For the reader's convenience, we state in our current notation the statement we will prove.
    \begin{theorem}[Moment curve decoupling]\label{momentcurve2} For each $\eps>0$ and $n\in\N$, there is a constant $C_{n,\eps}$ such that
    \begin{equation*}
        \mathcal{D}_n(\delta)\leq C_{n,\eps}\delta^{-\eps},
    \end{equation*}
    for all $\delta\in p^{-\N}$. Moreover, the constant $C_{n,\eps}$ may be taken to be
    \begin{equation*}
        C_{n,\eps}=\exp\left(10^4(\log p)\eps^{-4n\log n}n^{10n^2}\right).
    \end{equation*}
        
    \end{theorem}
    \begin{proof}[Proof of Theorem \ref{momentcurve2}]
        By induction, we will establish
    
        We will instead prove the stronger inequality
        \begin{equation*}
            \mathcal{D}_n(\delta)\leq\exp(10^4(\log p)(\eps/48n)^{-4n\log n}(48n)^{n^2}n^{4+5n})\delta^{-\eps},
        \end{equation*}
        for each $\eps=\frac{1}{\ell}$ for some $\ell\in\N$. We have spelled out a constant in a useful inductive form; after proving this for all $n$, we will go back and prove the original statement.
    
        We argue by induction on $n$. The case $n=1$ is trivial, so we assume that $n\geq 2$ and the estimate holds for all $\mathcal{D}_k$, $k=1,\ldots,n-1$. First take $\eps=(48n)^{-n^2\ell}$, for some $\ell\in\N$; we will later remove this assumption. For each $H\in\N$, we write
        \begin{equation*}
            \mathcal{T}^H=\left\{\frac{a}{b}\in(0,1)\cap\Q:a\in\Z,b=\prod_{j=1}^Hk_j\quad\text{for some}\quad(k_1,\ldots,k_H)\in[n]^H\right\}
        \end{equation*}
        for the rational numbers $t$ of ``depth'' $\leq H$. We write$N=\lfloor\frac{5(n-1)^2}{n\log n}\rfloor$, which will control the number of steps in our analysis. We will assume that $\delta\in p^{-(n!)^N\N}$ and $\delta<n^{-3000 n^4\eps^{-1}}$.
        We write
        \begin{equation*}
            \eta^\delta=\frac{\log\mathcal{D}_n(\delta)}{\log(\delta^{-1})},
        \end{equation*}
        and, when $0\leq k\leq n-1$ and $t\in\mathcal{T}^{N}$,
        \begin{equation*}
            A_k^\delta(t)=\frac{\log(\mathcal{B}_{n,k,\frac{n-k+1}{k}t,t}(\delta))}{\log(\delta^{-1})}.
        \end{equation*}
        By Lemma \ref{asequivlin}, we have $A_0^\delta(t)=(1-t)\eta^{\delta^{1-t}}$. We adopt the abbreviation

        \begin{equation*}
        q_\eps(\delta)=\frac{8(\log p)(\eps/48n)^{-4n\log n}}{\log(\delta^{-1})}\left(1+2500n\lceil\log((48n)^2\eps^{-1})\rceil(\eps/48n)^{4\log n}(48n)^{n^2}n^{4+n}\right),
        \end{equation*}
        a quantity controlling the logarithm of the prefactors in Lemma \ref{lowerdim} for each $1\leq k\leq n-1$, divided by $\log(\delta^{-1})$, using the inductive hypothesis. It will be important that, for every $n\geq 2$ and our choice of $\eps$, the second factor is $\leq 2$.
        
        By Lemma \ref{lowerdim}, selecting $\frac{\eps}{12n}$ for the statement's $\eps$, and the inductive hypothesis, for each $1\leq k\leq n-1$, and for each $t\in\mathcal{T}^{N-1}$ with $t\leq\frac{k(n-k)}{(k+1)(n-k+1)}$, and either $k=1$ or $t\leq\frac{k-1}{n-k+2}$,
        \begin{equation*}
            A_k^\delta(t)\leq\frac{1}{n-k+1}A_{n-k}^\delta\left(\frac{n-k+1}{k}t\right)+\frac{n-k}{n-k+1}A_{k-1}^\delta(t)+\frac{t\eps}{6n}+q_{\eps}(\delta).
        \end{equation*}
        We write, for $0\leq k\leq n-1$,
        \begin{equation*}
            \mathfrak{a}_k^\delta(t)=\frac{\eta^{\delta}-A_k^\delta(t)}{t},\quad t\in\mathcal{T}^N,
        \end{equation*}
        so that
        \begin{equation*}
            \mathfrak{a}_k^\delta(t)\geq\frac{1}{k}\mathfrak{a}_{n-k}^\delta\left(\frac{n-k+1}{k}t\right)+\frac{n-k}{n-k+1}\mathfrak{a}_{k-1}^\delta(t)-\frac{\eps}{6n}-\frac{q_{\eps}(\delta)}{t},\quad t\in\mathcal{T}^{N-1}.
        \end{equation*}
        We note as well the trivial bounds
        \begin{equation*}
            0\leq \mathfrak{a}_k^\delta(t)\leq \frac{1}{2t\log\delta^{-1}},
        \end{equation*}
        arising from the triangle inequality, Cauchy-Schwartz, and the fact that decoupling constants are at least $1$. Note that $\mathfrak{a}_0^\delta(t)=\eta^{\delta^{1-t}}$. We write $\mathfrak{a}^\delta(t)$ for the $(n-1)\times 1$ row vector composed of the $\mathfrak{a}_k^\delta(t)$. Define $M$ to be the $(n-1)\times(n-1)$ matrix
        \begin{equation*}
            M_{i,j}=\begin{cases}
                \frac{n-j}{n-j+1} & i = j + 1 \neq n - j\\
                \frac{1}{i} & i = n - j \neq j + 1\\
                1 & i = n - j = j + 1\\
                 0 & \text{otherwise}
            \end{cases}.
        \end{equation*}
        
        Trivially, for each choice $0<\rho<c<\frac{1}{2n}$ and each $0\leq \iota\leq N-1$,
        \begin{equation*}
            \min_{t\in\mathcal{T}^{\iota+1}\cap(\frac{1}{n}\rho,nc)}\mathfrak{a}_k^\delta(t)\leq\min_{t\in\mathcal{T}^{\iota}\cap(\rho,c)}\mathfrak{a}_k^\delta(t).
        \end{equation*}
        Let $c_0=10^{-2}n^{-2}$. Observe that $c_0$ satisfies the inequality
        \begin{equation*}
            n^{60(n-1)^2\eps^{-1}\frac{1}{c_0\log\delta^{-1}}}c_0<\frac{1}{2n}.
        \end{equation*}
        It follows that, for some $0\leq\iota\leq N$, the numbers $\rho=n^{-\iota}c_0,c=n^\iota c_0$ satisfy
        \begin{equation*}
            \min_{t\in\mathcal{T}^{\iota+1}(\frac{1}{n}\rho,nc)}\mathfrak{a}_k^\delta(t)>-\frac{\eps}{12n}+\min_{t\in\mathcal{T}^{\iota}\cap(\rho,c)}\mathfrak{a}_k^\delta(t),\quad\forall 1\leq k\leq n-1,
        \end{equation*}
        and
        \begin{equation*}
            c<\frac{1}{2n}.
        \end{equation*}
        
        Define, for each $0\leq k\leq n-1$,
        \begin{equation*}
            m_{k}^\delta=\min_{t\in\mathcal{T}^{\iota}\cap(\rho,c)}\mathfrak{a}_k^\delta(t),\quad\tilde{m}_k^\delta=\min_{t\in\mathcal{T}^{\iota+1}\cap(\frac{1}{n}\rho,nc)}\mathfrak{a}_k^\delta(t).
        \end{equation*}
        We see that the vectors $(m_{k}^\delta)_{k=1}^{n-1}$ satisfy
        \begin{equation*}
            \tilde{m}^\delta\geq m^\delta-\frac{\eps}{12n}\mathbf{1}
        \end{equation*}
        and
        \begin{equation*}
            m^\delta\geq \tilde{m}^\delta.M-\left\{\frac{\eps}{6n}+\rho^{-1}q_{\eps}(\delta)\right\}\mathbf{1}.
        \end{equation*}
        Combining, and taking a matrix multiplication on the right with the column vector $\mathbf{1}^\top$, we obtain
        \begin{equation*}
            \sum_{k=1}^{n-1}\tilde{m}_{k}^\delta\geq\sum_{k=1}^{n-1}\tilde{m}_k^\delta+\frac{n-1}{n}\eta^{\delta^{1-t}}-\frac{\eps}{4}-n\rho^{-1}q_{\eps}(\delta),\quad\text{for some $t\in\mathcal{T}^{N}\cap(\frac{\rho}{n},cn)$},
        \end{equation*}
        which may be written as
        \begin{equation*}
            \eta^{\delta^{\alpha}}\leq \frac{\eps}{2}+2n\rho^{-1}q_{\eps}(\delta),\quad\text{for some $\alpha\in\mathcal{T}^{N}\cap(\frac{1}{2},1)$}.
        \end{equation*}
        It follows from a short calculation that
        \begin{equation*}
            \mathcal{D}_n(\delta)\leq\exp(6400(\log p)(\eps/48n)^{-4n\log n}n^{4+5n})\delta^{-\eps}.
        \end{equation*}
        
        It remains to remove the special size and arithmetic assumptions on $\delta$ and $\eps$. We have shown the estimate for all $\delta\in p^{-(n!)^N\N}$ with $\delta<n^{-3000 n^4\eps^{-1}}$. Suppose instead that $K\in\N$ is such that
        \begin{equation*}
            p^{-(n!)^N(K+1)}<\delta<p^{-(n!)^NK}.
        \end{equation*}
        Then, by stability of $\mathcal{D}_n$, Lemma \ref{stability} with $\delta'=p^{-(n!)^NK}$, we obtain the bound
        \begin{equation*}
            \mathcal{D}_n(\delta)\leq p^{(n!)^N}\exp(6400(\log p)(\eps/48n)^{-4n\log n}n^{4+5n})\delta^{-\eps}.
        \end{equation*}
        But since $\eps\leq (48n)^{-n^2}$, it is quick to see via Stirling that
        \begin{equation*}
            p^{(n!)^N}\exp(6400(\log p)(\eps/48n)^{-4n\log n}n^{4+5n})\leq\exp(7000(\log p)(\eps/48n)^{-4n\log n}n^{4+5n}).
        \end{equation*}
        We have proven the bound for all $\delta\in p^{-\N}$ with $\delta<\min(p^{-(n!)^N},n^{-3000 n^4\eps^{-1}})$. By a trivial estimate, the same holds for all $\delta\in p^{-\N}$.

        Thus we are done in the case $\eps=(48n)^{-n^2\ell}$ for some $\ell\in\N$. Suppose instead $\eps=\frac{1}{\ell}$ for some $\ell\in\N$. We have two cases. In the first case, there is $\ell'\in\N$ so that
        \begin{equation*}
            (48n)^{-n^2(\ell'+1)}\leq\eps\leq(48n)^{-n^2\ell'}
        \end{equation*}
        Then we conclude, for each $\delta\in p^{-\N}$,
        \begin{equation*}
            \mathcal{D}_n(\delta)\leq \exp(7000(\log p)(\eps/48n)^{-4n\log n}(48n)^{n^2}n^{4+5n})\delta^{-\eps},
        \end{equation*}
        which fits in the estimate we wished to conclude. In the second case, we have $\eps>(48n)^{-n^2}$, and again a trivial estimate suffices.
        
    \end{proof}
    \begin{remark}
        One may compare the above analysis with the problem of bounding a constant quantity $\eta$, given that it relates to a system as
        \begin{equation*}
            \begin{cases}
                q(t)\in[0,\frac{C}{T-t}],\quad 0\leq t\leq T,\\
                \dot{q}\leq-\eta+O(\frac{\eps}{T-t}).
            \end{cases}
        \end{equation*}
    \end{remark}
    
    We conclude by recording the following standard corollary.
    
    \begin{corollary}\label{convexdec}
        Suppose $\zeta:\Z_p\to\Q_p^n$ is a $\mathbf{C}^{n+1}$-curve that is convex and has bounded derivatives, as defined in section \ref{curvesection}. Then
		\begin{equation*}
			\operatorname{Dec}_{\ell^2L^{q_n}}(\{\mathcal{U}_I^\zeta\}_{I\in\mathcal{P}(\Z_p,\delta)})\lesssim_{\gamma,\eps}\delta^{-\eps}\quad\forall\eps>0,\delta\in p^{-\N}.
		\end{equation*}
    \end{corollary}
    \begin{proof}
        Follows from Thm. \ref{momentcurve} and Prop. \ref{convdec}.
    \end{proof}

    \begin{remark}\label{rectsetfamdec}
        If we consider the alternate set-family $\{\mathcal{U}_I'\}_{I\in\mathcal{P}(\Z_p,\delta)}$, defined by
        \begin{equation*}
            \mathcal{U}_{I,\theta}'=\left\{x\in\Q_p^n:|x_j-\gamma_j(\theta)|_p\leq\delta^j\,\forall 1\leq j\leq n\right\},\quad (\theta\in I),
        \end{equation*}
        and
        \begin{equation*}
            \mathcal{U}_I'=\bigcup_{\theta\in I}\mathcal{U}_{I,\theta}',
        \end{equation*}
        then, for any $\delta\in p^{-n\N}$, $J\in\mathcal{P}(\Z_p,\delta)$ and $I\in\mathcal{P}(\Z_p,\delta^{1/n})$ with $J\subseteq I$, it holds that
        \begin{equation*}
            \mathcal{U}_J'\subseteq\mathcal{U}_{I}.
        \end{equation*}
        Consequently, we have the decoupling estimate
        \begin{equation*}
            \mathrm{Dec}_{\ell^2L^{n(n+1)}}\left(\Big\{\bigcup_{J\in\mathcal{P}(I,\delta)}\mathcal{U}_J'\Big\}_{I\in\mathcal{P}(\Z_p,\delta^{1/n})}\right)\leq C_{n,n\eps}\delta^{\eps},
        \end{equation*}
        for each $\eps$ and $\delta$.
    \end{remark}

    \section{Appendix C: Curves in \texorpdfstring{$\Q_p^n$}{Qpn}}\label{curvesection}

    We take as a reference \cite{schikhof1985}.
	
	The curves $\zeta:\Z_p\to\Q_p^n$ under consideration will be assumed to be $\mathbf{C}^{k}$, for various $k\in\N\cup\{\infty\}$; as such, we recall some of the basics of ultrametric calculus. Consider an arbitrary function $f:\Z_p\to\Q_p$. If $k\in\N$ and $a_1,\ldots,a_{k+1}\in\Z_p$ are distinct, we write $\Phi_k$ for the Newton quotient
    \begin{equation*}
        \Phi_kf(a_1,\ldots,a_{k+1})=\sum_{j=1}^{k+1}\frac{f(a_j)}{\prod_{i\neq j}(a_j-a_i)}.
    \end{equation*}
    We will also write $\Phi_0f=f$. A function $f$ is said to be $\mathbf{C}^k$ if, for every $0\leq j\leq k$, the function $\Phi_jf$ extends to a continuous function $\overline{\Phi}_j:\Z_p^{k+1}\to\Q_p$. $f$ is said to be $\mathbf{C}^\infty$ if $f\in \mathbf{C}^k$ for every $k$. This definition is extended to curves $\zeta:\Z_p\to\Q_p^n$ in the obvious way.

    The following summarizes the basic facts about the mappings $\Phi_k$ that we will need.
    \begin{proposition}\label{phicomp}
        Let $f:\Z_p\to\Q_p$ be a $\mathbf{C}^{n}$ function. Then each of the following holds.
        \begin{enumerate}[label=(\alph*)]
            \item If $f$ is $\mathbf{C}^{n}$, then for each $1\leq k\leq n$ and $a\in\Z_p$, we have
            \begin{equation*}
                f^{(k)}(a)=\overline{\Phi}_kf(a,\ldots,a),
            \end{equation*}
            where $f^{(k)}(a)$ is the usual $k^{th}$ derivative of $f$ at $a$.
            \item $f$ admits Taylor expansions
            \begin{equation}\label{taylor}
        		f(x)=f(y)+\sum_{j=1}^{n}\frac{f^{(j)}(y)}{j!}(x-y)^j+(x-y)^{n}\Lambda_{n+1}(x,y)\quad\forall x,y\in\Z_p,
        	\end{equation}
            where the remainder term $\Lambda_{n+1}(x,y)$ is of the form
            \begin{equation*}
                \Lambda_{n+1}(x,y)=\overline{\Phi}_nf(x,y,y,\ldots,y)-f^{(n)}(y).
            \end{equation*}
            \item For any elements $x_1,\ldots,x_{k+1},y_1,\ldots,y_{k+1}\in\Z_p$, we have
            \begin{equation*}
                \overline{\Phi}_kf(x_1,\ldots,x_{k+1})-\overline{\Phi}_kf(y_1,\ldots,y_{k+1})=\sum_{j=1}^{k+1}(x_j-y_j)\overline{\Phi}_{k+1}(x_1,\ldots,x_j,y_j,\ldots,y_{k+1}).
            \end{equation*}
        \end{enumerate}
    \end{proposition}
    \begin{proof}
        Taken from \cite{schikhof1985}; (a) is Theorem 29.5, (b) is Theorem 29.4, and (c) is Lemma 29.2(iii). Note that our definition of $\Phi$ is equivalent to that of \cite{schikhof1985}, by the latter's Exercise 29.A.
    \end{proof}
    We define the $\mathbf{C}^{k}$ norm of a $\mathbf{C}^{k}$ function $f$ by
    \begin{equation*}
        \|f\|_{\mathbf{C}^k}=\max_{0\leq j\leq k}\sup_{x\in\Z_p^{j+1}}|\Phi_kf(x)|_p.
    \end{equation*}
    We similarly define the $\mathbf{C}^k$ seminorm of a $\mathbf{C}^k$ function $f$ by
    \begin{equation*}
        \|f\|_{\mathbf{C}_\circ^k}=\max_{1\leq j\leq k}\sup_{x\in\Z_p^{j+1}}|\Phi_kf(x)|_p.
    \end{equation*}
    If $\zeta:\Z_p\to\Q_p^n$ is a $\mathbf{C}^{k}$ curve, then we write
    \begin{equation*}
        \|\zeta\|_{\mathbf{C}^{k}}=\max_{1\leq i\leq k}\|\zeta_i\|_{\mathbf{C}^{k}},\quad\|\zeta\|_{\mathbf{C}_\circ^{k}}=\max_{1\leq i\leq k}\|\zeta_i\|_{\mathbf{C}_\circ^{k}}
    \end{equation*}
    In particular, $|\zeta_i^{(j)}(x)|_p\leq\|\zeta\|_{\mathbf{C}^{k}}$ for each $1\leq i\leq n$ and $0\leq j\leq k$. We note in passing that, for any linear transformation $B$ of $\Q_p^n$, we have the bound
    \begin{equation*}
        \|B\zeta\|_{\mathbf{C}^k}\leq\|B\|\cdot\|\zeta\|_{\mathbf{C}^k}.
    \end{equation*}
    We also note that our definition of $\|\cdot\|_{\mathbf{C}^k}$ is strictly stronger than the simpler quantity
    \begin{equation*}
        \|f\|_{\mathbf{C}^k}^{\star}=\max_{0\leq k\leq n}\sup_{x\in\Z_p}|f^{(k)}(x)|_p.
    \end{equation*}
    Indeed, if $f=\mathds{1}_{B(0,p^{-N})}$ is the indicator of the ball of radius $p^{-N}$, then $f\in \mathbf{C}^\infty$ and
    \begin{equation*}
        \sup_{x\in\Z_p}|f(x)|_p=1,\quad\max_{j\in\N}\sup_{x\in\Z_p}|f^{(k)}(x)|_p=0.
    \end{equation*}
    On the other hand,
    \begin{equation*}
        \left|\frac{f(p^{N-1})-f(0)}{p^{N-1}-0}\right|_p=p^{N-1},
    \end{equation*}
    so $\|f\|_{\mathbf{C}^1}\geq p^{N-1}$. It follows that $\|f\|_{\mathbf{C}^k}\geq p^{N-1}\|f\|_{\mathbf{C}^k}^\star$, for each $k\geq 1$; thus, $\|\cdot\|_{\mathbf{C}^k}$ defines a strictly finer topology on $\mathbf{C}^\infty$ than $\|\cdot\|_{\mathbf{C}^k}^\star$.

	We will want our curves $\zeta$ to be convex and of bounded derivatives; the former condition amounts to
	\begin{equation}\label{convexcond}
		\inf_{t\in\Z_p}\big|\det[\zeta^{(1)}(t),\ldots,\zeta^{(n)}(t)]\big|_p\geq c,
	\end{equation}
	and the latter condition is that 
	\begin{equation}\label{bdddercond}
		\max_{1\leq i,j\leq n}\sup_{t\in\Z_p}\big|\zeta_i^{(j)}(t)\big|_p\leq C.
	\end{equation}
    for various choices $c\gtrsim 1,C\lesssim 1$; we name these bounds so as to track quantitative dependence in the sequel.

    As an immediate consequence, we have:
    \begin{lemma}\label{matrixinverselemma}
        Suppose $\zeta:\Z_p\to\Q_p^n$ is $\mathbf{C}^{n}$. Suppose further that $\zeta$ satisfies \ref{convexcond} and \ref{bdddercond}. Then, for each $t\in\Z_p$, one has
        \begin{equation*}
            \|[\zeta^{(1)}(t),\ldots,\zeta^{(n)}(t)]^{-1}\|\leq c^{-1}C^{n-1}.
        \end{equation*}
        and
        \begin{equation*}
            \|[\zeta^{(1)}(t),\ldots,\zeta^{(n)}(t)]\|\leq C^{n}
        \end{equation*}
        Here $\|\cdot\|$ is operator norm with respect to the max-norm on $\Q_p^n$.
    \end{lemma}
    \begin{proof}
        We only verify the first estimate; the second will follow by identical arrangements. The matrix norm is given by
        \begin{equation*}
            \begin{split}
            \max_{1\leq j\leq n}\max_{u_1,\ldots,u_n\in\Z_p}&\left|\sum_{k=1}^n\left([\zeta^{(1)}(t),\ldots,\zeta^{(n)}(t)]^{-1}\right)_{j,k}u_k\right|_p\\
            &=\max_{1\leq j\leq n}\max_{u_1,\ldots,u_n\in\Z_p}\max_{1\leq k\leq n}\left|\left([\zeta^{(1)}(t),\ldots,\zeta^{(n)}(t)]^{-1}\right)_{j,k}u_k\right|_p\\
            &=\max_{1\leq j,k\leq n}\left|\left([\zeta^{(1)}(t),\ldots,\zeta^{(n)}(t)]^{-1}\right)_{j,k}\right|_p
            \end{split}
        \end{equation*}
        By the cofactor form of matrix inverses, for each $1\leq j,k\leq n$,
        \begin{equation*}
            \left|\left([\zeta^{(1)}(t),\ldots,\zeta^{(n)}(t)]^{-1}\right)_{j,k}\right|_p=|\det[\zeta^{(1)}(t),\ldots,\zeta^{(n)}(t)]|_p^{-1}|C_{k,j}|_p
        \end{equation*}
        where $C_{k,j}$ is the $(k,j)$-cofactor. To estimate the latter, we simply use the (ultrametric) triangle inequality via
        \begin{equation*}
            |C_{k,j}|_p\leq\max_{1\leq i,j\leq n}|\zeta_i^{(j)}(t)|_p^{n-1}
        \end{equation*}
    \end{proof}

    \begin{remark}
        When $\zeta=\gamma$ is the moment curve $t\mapsto (t,\ldots,\frac{t^n}{n!})$ in $\Q_p^n$, the constants in \ref{convexcond} and \ref{bdddercond} are $c=C=1$.
    \end{remark}

    Next, we consider the local comparison between curves $\zeta$ that are convex and with bounded derivatives to the model curve $\gamma(t)=(t,\ldots,\frac{t^n}{n!})$. Recall the notation
    \begin{equation*}
        A_{\theta,\lambda}^\zeta=[\zeta^{(1)}(\theta),\ldots,\zeta^{(n)}(\theta)]\cdot\mathrm{diag}(\lambda,\ldots,\lambda^n)
    \end{equation*}
    If $\zeta$ is suppressed, then we understand $A_t$ to refer to the matrix $A_t^\gamma$. For fixed $\theta\in\Z_p$ and $\lambda\in p\Z_p$, define $\zeta_{\theta,\lambda}$ to be the rescaled curve
    \begin{equation*}
        \zeta_{\theta,\lambda}(t)=[A_{\theta,\lambda}^\zeta]^{-1}(\zeta(\theta+\lambda t)-\zeta(\theta))
    \end{equation*}
    The rescaling is motivated by the fact that the degree $n$ Taylor approximation of the function $t\mapsto\zeta(\theta+\lambda t)$ near $t=0$ is
    \begin{equation*}
        \zeta(\theta+\lambda t)\approx\zeta(\theta)+A_{\theta,\lambda}^\zeta\cdot\gamma(t),\quad (\,|t|_p\ll 1)
    \end{equation*}
    Of course, if $\zeta$ is a polynomial of degree $\leq n$, the Taylor approximation is an identity. In particular, for such a $\zeta$, $\zeta_{\theta,\lambda}=\gamma$ is our moment curve; a trivial observation is the special case $\gamma=\gamma_{\theta,\lambda}$ for each $\theta,\lambda$.

    Note in particular the identity
    \begin{equation*}
        \Phi_k(\zeta_{\lambda,\theta})_i(a_1,\ldots,a_{k+1})=\lambda^k\Phi_k\zeta_i(c+\lambda a_1,\ldots,c+\lambda a_{k+1}),\quad k\geq 1,
    \end{equation*}
    which implies the scaling relations
    \begin{equation*}
        \|\zeta_{\lambda,\theta}\|_{\mathbf{C}^k(\lambda\Z_p)}\leq \|\zeta\|_{\mathbf{C}^k},\quad k\geq 1,\lambda\in p\Z_p,
    \end{equation*}
    \begin{equation*}
        \|\zeta_{\lambda,\theta}\|_{\mathbf{C}_\circ^k(\lambda\Z_p)}\leq |\lambda|_p^{-1}\|\zeta\|_{\mathbf{C}_\circ^k},\quad k\geq 1,\lambda\in p\Z_p.
    \end{equation*}
    
\printbibliography

\end{document}